\documentclass{article}
\usepackage{amssymb}
\usepackage{latexsym}
\usepackage{amsmath}
\usepackage{mathrsfs}
\usepackage{color}
\usepackage{CJK}
\usepackage{graphicx}

\newtheorem{theorem}{Theorem}[section]

\newtheorem{lemma}{Lemma}[section]
\newtheorem{proposition}{Proposition}[section]
\newtheorem{remark}{Remark}[section]
\newtheorem{definition}{Definition}[section]

\numberwithin{equation}{section}

\newenvironment{proof}{\medskip\par\noindent{\bf Proof.}\ }{\qquad
\raisebox{-0.5mm}{\rule{1.5mm}{4mm}}\vspace{6pt}}

\newcommand{\bbr}{\mathbb{R}}
\newcommand{\h}{H^1_0(\Omega)}

\newcommand{\bbn}{\mathbb{N}}
\newcommand{\ve}{\varepsilon}

\numberwithin{equation}{section}

\begin{document}


\title
{\Large\bf Spikes of the two-component elliptic system in $\bbr^4$ with
Sobolev critical exponent\thanks{Supported by NSFC (11701554, 11771319) and the
Fundamental Research Funds for the Central Universities (2017XKQY091). \;\; E-mails: wuyz850306@cumt.edu.cn (Y. Wu); wzou@math.tsinghua.edu.cn (W. Zou)}}

\author{\bf Yuanze Wu$^{a}$    \&    Wenming Zou$^b$    \\
\footnotesize$^a${\it School of Mathematics, China
University of Mining and Technology,}\\
\footnotesize{\it Xuzhou 221116, P. R. China }\\
\footnotesize$^b${\it  Department of Mathematical Sciences, Tsinghua University,}\\
\footnotesize{\it Beijing 100084, P. R. China }}
\date{}
\maketitle

\noindent{\bf Abstract:} Consider the following elliptic system:
\begin{equation*}
\left\{\aligned&-\ve^2\Delta u_1+\lambda_1u_1=\mu_1u_1^3+\alpha_1u_1^{p-1}+\beta u_2^2u_1\quad&\text{in }\Omega,\\
&-\ve^2\Delta u_2+\lambda_2u_2=\mu_2u_2^3+\alpha_2u_2^{p-1}+\beta u_1^2u_2\quad&\text{in }\Omega,\\
&u_1,u_2>0\quad\text{in }\Omega,\quad u_1=u_2=0\quad\text{on }\partial\Omega,\endaligned\right.
\end{equation*}
where $\Omega\subset\bbr^4$ is a bounded domain, $\lambda_i,\mu_i,\alpha_i>0(i=1,2)$ and $\beta\not=0$ are constants, $\ve>0$ is a small parameter and $2<p<2^*=4$.  By using the variational method, we study the existence of the ground state solution to this system for $\ve>0$ small enough.  The concentration behavior of the ground state solution as $\ve\to0^+$ is also studied.  Furthermore, by combining the elliptic estimates and local energy estimates, we also obtain the location of the spikes as $\ve\to0^+$.  To the best of our knowledge, this is the first attempt  devoted to the spikes in the Bose-Einstein condensate   in  $\bbr^4$.

\vspace{6mm} \noindent{\bf Keywords:} Elliptic system; Sobolev critical exponent; Spike; Semi-classical solution; Variational method.

\vspace{6mm}

\maketitle

\section{Introduction}
In this paper, we study the following elliptic system:
\begin{equation*}
\left\{\aligned&-\ve^2\Delta u_1+\lambda_1u_1=\mu_1u_1^3+\alpha_1u_1^{p-1}+\beta u_2^2u_1\quad&\text{in }\Omega,\\
&-\ve^2\Delta u_2+\lambda_2u_2=\mu_2u_2^3+\alpha_2u_2^{p-1}+\beta u_1^2u_2\quad&\text{in }\Omega,\\
&u_1,u_2>0\quad\text{in }\Omega,\quad u_1=u_2=0\quad\text{on }\partial\Omega,\endaligned\right.\eqno{(\mathcal{S}_{\ve})}
\end{equation*}
where $\Omega\subset\bbr^4$ is a bounded domain, $\lambda_i,\mu_i,\alpha_i>0(i=1,2)$ and $\beta\not=0$ are constants, $\ve>0$ is a small parameter and $2<p<2^*=4$.

It is well known that the solution of $(\mathcal{S}_{1})$ (i.e., $\ve=1$) in low dimension space $\bbr^N$ ($1\leq N\leq3$) for $\alpha_1=\alpha_2=0$ is related to the solitary wave
solutions of the following two coupled nonlinear Schr\"odinger equations which is also known in the literature as Gross-Pitaevskii equations (e.g. \cite{HMEWC98,TV09}):
\begin{equation*}
\left\{\aligned&-\iota\frac{\partial}{\partial t}\Psi_1=\Delta \Psi_1+\mu_1|\Psi_1|^2\Psi_1+\beta|\Psi_2|^{2}\Psi_1,\\
&-\iota\frac{\partial}{\partial t}\Psi_2=\Delta \Psi_2+\mu_2|\Psi_2|^2\Psi_2+\beta|\Psi_1|^{2}\Psi_2,\\
&\Psi_i=\Psi_i(t,x)\in H^1(\bbr^N; \mathbb{C}),\ \ i=1,2,\ \ N=1,2,3.\endaligned\right.%
\end{equation*}
Here, $\iota$ is the imaginary unit.  Such a system appears in many different physical problems.  For example, in the Hartree-Fock theory, the Gross-Pitaevskii equations can be used to describe a
binary mixture of Bose-Einstein condensates in two different hyperfine states $|1\rangle$ and $|2\rangle$ (cf. \cite{EGBB97}).  The solutions $\Psi_j(j=1,2)$ are the corresponding condensate amplitudes and $\mu_i$ are the intraspecies and interspecies scattering lengths.  $\beta$ is the interaction of the states $|1\rangle$ and $|2\rangle$ and the interaction is attractive if $\beta>0$ and repulsive if $\beta<0$.  The Gross-Pitaevskii equations also arises in nonlinear optics (cf. \cite{AA99}).  To obtain the solitary wave
solutions, we set $\Psi_i(t,x)=e^{-\iota t\lambda_i}u_i(x)$ for both $i=1,2$.  Then $u_i$ satisfy $(\mathcal{S}_{1})$ for $\alpha_1=\alpha_2=0$.
Due to the important applications in physics, the System~$(\mathcal{S}_{1})$ in low dimensions ($1\leq N\leq3$) for $\alpha_1=\alpha_2=0$ has been studied extensively in the last decades.  We refer the readers to \cite{BDW10,CTV05,CLZ141,DW09,LW051,LW06,LW08,LW10,NTTV10,R14,S07,TV09,WW08,WWZ17} and the references therein, where various existence theorems of the solitary wave solutions were established.

\vskip0.12in

Recently, the System~$(\mathcal{S}_{1})$ (i.e., $\ve=1$) for $\alpha_1=\alpha_2=0$ in higher dimension  space $\bbr^N$ $(N\geq4)$ has begun to attract attention (cf. \cite{CZ121,TT12}).  Note that the cubic nonlinearities and the coupled terms   are all of  critical growth when   $N=4$ and even super critical for $N\geq5$ with respect to the Sobolev critical exponent.  Thus, the study on these cases is much more complicated than that in low dimensions in the view point of calculus of variation.  By applying the truncation argument, Tavares and Terracini in \cite{TT12} proved that the $k$--component System~$(\mathcal{S}_{1})$ for $\alpha_1=\alpha_2=0$ has infinitely many sign-changing solution for all $N\geq2$ and $k\geq2$ with $\mu_1,\mu_2\leq0$, but where $\lambda_i>0$ for all $i=1,2,\cdots,k$ are appeared as the Lagrange multipliers, not given in advance.  In \cite{CZ121}, by establishing the threshold for the compactness of the $(PS)$ sequence to the System~$(\mathcal{S}_{1})$ for $\alpha_1=\alpha_2=0$ in dimension four and making some careful and complicated analysis, Chen and Zou proved that the System~$(\mathcal{S}_{1})$ for $\alpha_1=\alpha_2=0$ has a positive ground state solution for $N=4$ and $-\sigma_1<\lambda_i<0$ for all $i=1,2$, where $\sigma_1$ is the first eigenvalue of $-\Delta$ in $L^2(\Omega)$.  There are also some other studies on elliptic systems with the Sobolev critical exponent, we refer the readers to \cite{AFP09,CZ131,CLZ141,CZ15,CL15,LZ16} and the references therein.

\vskip0.16in

When $\ve>0$ is a small parameter, the solutions of the System~$(\mathcal{S}_{\ve})$ are  called as the semi-classical bound state solutions.  Such solutions have been studied extensively for the single equation in the past thirty years, we refer the readers to \cite{B10,BZZ13,NW95,W93} and the references therein.  To the best knowledge, the first result devoted to such solutions of the elliptic systems is contributed by Lin and Wei in \cite{LW05}, where the System~$(\mathcal{S}_{\ve})$ with $\alpha_1=\alpha_2=0$ in the low dimensions ($1\leq N\leq3$) has been studied.  By using the variational method, the authors proved that such system has a ground state solution for $\ve>0$ small enough.  They also studied the concentration behaviors of this ground state solution and described the locations of the spikes as $\ve\to0^+$.  Since then, many works have been devoted to the semi-classical solutions of the elliptic systems in low dimensions ($N=1, 2, 3$) and various similar results have been established, we refer the readers to \cite{B15,IT11,LW06,LW13,LP14,MPS06,WS16} and the references therein.

\vskip0.16in

Motivated by the above facts, we wonder what happens to the System~$(\mathcal{S}_{\ve})$ in dimension four and does similar results hold?
It is   worth pointing out that in dimension four case, the system~$(\mathcal{S}_{\ve})$ is of critical growth, i.e.,
the cubic terms $u_1^3, u_2^3$ and the coupling terms $u_2^2u_1,  u_1^2u_2$ are all of critical growth in the sense of the Sobolev embedding.
It is well know that, for the Sobolev critical equation or system, the existence of nontrivial solution is very fragile!  Recall that, by using the Pohozaev identity as in \cite{CZ15} and \cite{ZZ12}, we can see that
\begin{equation}\label{zwm=1}
\left\{\aligned&-\Delta u+\lambda_iu=\mu_iu^3\quad&\text{in }\bbr^4,\\
&u>0\text{ in }\bbr^4, \endaligned\right.
\end{equation}
and
\begin{equation}\label{zwm=2}
\left\{\aligned&-\Delta u_1+\lambda_1u_1=\mu_1u_1^3+\beta u_2^2u_1\quad&\text{in }\bbr^4,\\
&-\Delta u_2+\lambda_2u_2=\mu_2u_2^3+\beta u_1^2u_2\quad&\text{in }\bbr^4,\\
&u_1,u_2>0\text{ in }\bbr^4, \endaligned\right.
\end{equation}
has no solution for $\lambda_1,\lambda_2>0$,  \eqref{zwm=1}-\eqref{zwm=2} are the limit equation and system of $(\mathcal{S}_{\ve})$ respectively for $\beta<0$ and $\beta>0$ if $\alpha_1=\alpha_2=0$.  In fact, from the view point of calculus of variation, $\lambda_1,\lambda_2>0$ and  $\alpha_1,\alpha_2>0$ seem to be necessary.  We may also observe this fact from the other hand. Indeed, without loss of generality, we assume $0\in\Omega$.  Then $\Omega_{\ve}\to\bbr^4$ as $\ve\to0^+$, where $\Omega_{\ve}=\{x\in\bbr^4\mid \ve x\in\Omega\}$.  Now, if $\lambda_1,\lambda_2<0$, then by scaling, the system~$(\mathcal{S}_{\ve})$ is equivalent to the following
\begin{equation*}
\left\{\aligned&-\Delta v_1+\lambda_1v_1=\mu_1(v_1)^3+\alpha_1(v_1)^{p-1}+\beta (v_2)^2v_1\quad&\text{in }\Omega_{\ve},\\
&-\Delta v_2+\lambda_2v_2=\mu_2(v_2)^3+\alpha_2(v_2)^{p-1}+\beta (v_1)^2v_2\quad&\text{in }\Omega_{\ve},\\
&v_1,v_2>0\quad\text{in }\Omega_{\ve},\quad v_1=v_2=0\quad\text{on }\partial\Omega_{\ve}.\endaligned\right.
\end{equation*}
Since $\lambda_1^*(\Omega_{\ve})\to0^+$ as $\ve\to0^+$, we can easy to see that $\lambda_i<-\lambda_1^*(\Omega_{\ve})$ for $\ve$ small enough, where $\lambda_1^*(\Omega_{\ve})$ is the first eigenvalue of $-\Delta$ in $L^2(\Omega_\ve)$.  By multiplying the above system with $\varphi_{1,\ve}$ and integral by parts, we can see that the above system has no solution for $\ve$ small enough with $\beta>0$, where $\varphi_{1,\ve}$ is the positive eigenfunction to $\lambda_1^*(\Omega_{\ve})$.  It implies that the system~$(\mathcal{S}_{\ve})$ has {\it no} solution for $\ve$ small enough with $\beta>0$ if $\lambda_1,\lambda_2<0$.  In the case $\lambda_1,\lambda_2>0$, but  {if $\alpha_1=\alpha_2=0$}, then as stated in \cite{CZ121}, by the Pohozaev identity, it is easy to see that $(\mathcal{S}_{\ve})$ also has {\it no}  solution when $\Omega$ is star-shaped.
Therefore, the subcritical terms $u_1^{p-1}, u_2^{p-1}$ in  $(\mathcal{S}_{\ve})$ will  play  an important role which will ensure the existence
of nontrivial solutions.

\vskip0.23in

 To the best of our knowledge, the concentration behavior of the ground state solution and the location of the spikes as $\ve\to0^+$
 of the System~$(\mathcal{S}_{\ve})$ with $N=4$ have not been studied yet in the literatures.  Thus, we
shall explore this problem in the current paper. We first give the existence results for the System~$(\mathcal{S}_{\ve})$ in  $\bbr^4$, which is stated as follows.
\begin{theorem}\label{thm0001}
Let $\lambda_i,\mu_i,\alpha_i>0(i=1,2)$ and $\beta\not=0$.  Then there exist three positive constant $\alpha_0$, $\beta_0$ and $\beta_1$ with $\beta_0<\beta_1$ such that $(\mathcal{S}_{\ve})$ has a nontrivial solution $\overrightarrow{\mathbf{u}}_\ve$ when  $\ve>0$ is small enough and one of the  following three cases holds:
\begin{enumerate}
\item[$(1)$]  $\beta>\beta_1$,
\item[$(2)$]  $-\sqrt{\mu_1\mu_2}\leq\beta<\beta_0$,
\item[$(3)$]  $\beta<-\sqrt{\mu_1\mu_2}$ and $|\overrightarrow{\mathbf{\alpha}}|<\alpha_0$,
\end{enumerate}
where $\overrightarrow{\mathbf{\alpha}}=(\alpha_1, \alpha_2)$.
\end{theorem}

\begin{remark}\quad
\begin{enumerate}
\item[$(a)$]  The existence of $\alpha_0$ seems to be technique. Indeed, in the case $\beta<-\sqrt{\mu_1\mu_2}$, the functional
\begin{eqnarray*}
\frac{\mu_1}{4}\int_{\Omega}|u_1|^4dx+\frac{\mu_2}{4}\int_{\Omega}|u_2|^4dx+\frac{\beta}{2}\int_{\Omega}|u_1|^{2}|u_2|^{2}dx
\end{eqnarray*}
is indefinite in $H^1_0(\Omega)\times H^1_0(\Omega)$.  Since $2<p<4$, it is hard for us to obtain the boundedness of the $(PS)$ sequence of the corresponding functional to $(\mathcal{S}_{\ve})$.  Due to this reason, we treat the subcritical terms as a perturbation and use a truncation argument to deal with them.
\item[$(b)$]  We mainly use the Nehari's manifold approach to prove Theorem~\ref{thm0001}.  However, due to the subcritical terms, the fibering map related to $(\mathcal{S}_{\ve})$ is very complex.  In this situation, we will apply the implicit function theorem as in \cite{W17} and the Miranda's Theorem as in \cite{CLZ141} to recover the properties of the Nehari manifold which are  crucial in proving Theorem~\ref{thm0001}.
\end{enumerate}
\end{remark}

Next, we establish a result on the concentration behavior of $\overrightarrow{\mathbf{u}}_\ve$ as $\ve\to0^+$.  Since the concentration behavior depends on $\beta$, we re-denote $\overrightarrow{\mathbf{u}}_\ve$ by $\overrightarrow{\mathbf{u}}_{\ve, \beta}$.
\begin{theorem}\label{thm0002}
Let $p_i^{\ve,\beta}$ be the maximum point of $u_i^{\ve,\beta}$ respectively for $i=1,2$, where $\overrightarrow{\mathbf{u}}_{\ve,\beta}=(u_1^{\ve,\beta}, u_2^{\ve,\beta})$ is the solution of $(\mathcal{S}_{\ve})$.  Then we have
\begin{enumerate}
\item[$(a)$]  If $\beta<0$, then $(u_1^{\ve,\beta}(\ve y+p_1^{\ve,\beta}), u_2^{\ve,\beta}(\ve y+p_2^{\ve,\beta}))\to(v_1^0, v_2^0)$ strongly in $H^1(\bbr^4)\times H^1(\bbr^4)$ as $\ve\to0^+$, where $v_i^0$ is respectively a ground state solution of the following equation
\begin{equation}\label{eqnew0001}
\left\{\aligned&-\Delta u+\lambda_iu=\mu_iu^3+\alpha_iu^{p-1}\quad&\text{in }\bbr^4,\\
&u>0\text{ in }\bbr^4,\quad u\to0\text{ as }|x|\to+\infty.\endaligned\right.
\end{equation}
Moreover, $\frac{|p_1^{\ve,\beta}-p_2^{\ve,\beta}|}{\ve}\to+\infty$ as $\ve\to0^+$.
\item[$(b)$]  If $\beta>0$, then $(u_1^{\ve,\beta}(\ve y+p_1^{\ve,\beta}), u_2^{\ve,\beta}(\ve y+p_2^{\ve,\beta}))\to(v_1^*, v_2^*)$ strongly in $H^1(\bbr^4)\times H^1(\bbr^4)$ as $\ve\to0^+$, where $(v_1^*, v_2^*)$ is the ground state solution of the following elliptic system
\begin{equation}\label{eqnew0002}
\left\{\aligned&-\Delta u_1+\lambda_1u_1=\mu_1u_1^3+\alpha_1u_1^{p-1}+\beta u_2^2u_1\quad&\text{in }\bbr^4,\\
&-\Delta u_2+\lambda_2u_2=\mu_2u_2^3+\alpha_2u_2^{p-1}+\beta u_1^2u_2\quad&\text{in }\bbr^4,\\
&u_1,u_2>0\text{ in }\bbr^4,\quad u_1,u_2\to0\text{ as }|x|\to+\infty.\endaligned\right.
\end{equation}
Moreover, $\frac{|p_1^{\ve,\beta}-p_2^{\ve,\beta}|}{\ve}\to0$ as $\ve\to0^+$.
\end{enumerate}
\end{theorem}

\vskip0.3in

\begin{remark}\quad
\begin{enumerate}
\item[$(a)$]  As the results in \cite{LW05}, by Theorem~\ref{thm0002}, we can see that the concentration behavior of $\overrightarrow{\mathbf{u}}_{\ve,\beta}$ as $\ve\to0^+$ is quite different for $\beta<0$ and $\beta>0$, which is caused by different limits of the coupled term $\int_{\Omega}(u_1^{\ve,\beta})^2(u_2^{\ve,\beta})^2dx$ as $\ve\to0^+$.  Indeed, based on the observations on the limits of the energy values as $\ve\to0^+$ (see Lemma~\ref{lem0005}), we can deduce that $\int_{\Omega}(u_1^{\ve,\beta})^2(u_2^{\ve,\beta})^2dx\to0$ for $\beta<0$ as $\ve\to0^+$.  This, together with some suitable scaling centered at the maximum points $p_1^{\ve,\beta}, p_2^{\ve,\beta}$ and the uniformly elliptic estimates, implies that the spikes will repel each other and behave like
    two separate spikes for $\beta<0$, that is $\frac{|p_1^{\ve,\beta}-p_2^{\ve,\beta}|}{\ve}\to+\infty$ as $\ve\to0^+$.  In the case $\beta>0$, by the observations on the limits of the energy values as $\ve\to0^+$ (see also Lemma~\ref{lem0005}), we can deduce that $\int_{\Omega}(u_1^{\ve,\beta})^2(u_2^{\ve,\beta})^2dx\geq C+o(1)$.  It also together with some suitable scaling centered at the maximum points $p_1^{\ve,\beta}, p_2^{\ve,\beta}$, implies that the spikes will attract each other and behave like a single spike, that is $\frac{|p_1^{\ve,\beta}-p_2^{\ve,\beta}|}{\ve}$ is bounded as $\ve\to0^+$.  Combining with the radial
    symmetric of the solutions of the limit system for $\beta>0$, we actually obtain that $\frac{|p_1^{\ve,\beta}-p_2^{\ve,\beta}|}{\ve}\to0$ as $\ve\to0^+$ since $p_1^{\ve,\beta}, p_2^{\ve,\beta}$ are respectively the maximum points.  However, compare with the arguments in \cite{LW05} for such results, the further difficulty in proving Theorem~\ref{thm0002} is that the uniform boundedness of $u_i^{\ve,\beta}(\ve y+p_i^{\ve,\beta})$ in $L^\infty(\bbr^4)$, which is crucial in proving such result, can not be obtained by applying the Moser iteration directly, which is caused by the critical growth of the cubic nonlinearities and the coupled term in dimension four.  In order to overcome the difficulty, we first prove the strong convergence of $(u_1^{\ve,\beta}(\ve y+p_1^{\ve,\beta}), u_2^{\ve,\beta}(\ve y+p_2^{\ve,\beta}))$ in $H^1(\bbr^4)\times H^1(\bbr^4)$ as $\ve\to0^+$ and then apply the Moser iteration as in \cite{BZZ13,CZ121}.
\item[$(b)$]  An interesting  phenomenon about the semi-classical solution is that it will concentrate around the maximum points and convergence strongly to the nontrivial solution of the limit equation or system under some suitable scaling centered at the maximum points.  Thus, the existence of nontrivial solutions of the limit equation or system is very important for  observing such a phenomenon.  Now, by using the Pohozaev identity as in \cite{CZ15} and \cite{ZZ12}, we can see that
    \eqref{zwm=1}- \eqref{zwm=2} have no solution for $\lambda_1,\lambda_2>0$.
   It means that  the equations   \eqref{eqnew0001} and \eqref{eqnew0002}, which are the limit equation or system of $(\mathcal{S}_{\ve})$ respectively for $\beta<0$ and $\beta>0$, have  no solution if $\alpha_1=\alpha_2=0$.  Hence, $\alpha_1,\alpha_2>0$ seems to be necessary in proving Theorem~\ref{thm0002}.
\end{enumerate}
\end{remark}

Finally, we  study the locations of the spikes $p_i^\ve$ as $\ve\to0^+$.

\vskip0.2in

\begin{theorem}\label{thm0003}
Let $p_i^{\ve,\beta}$ be the maximum point of $u_i^{\ve,\beta}$ respectively for $i=1,2$.  Then we have
\begin{enumerate}
\item[$(a)$]  If $\beta>0$, then $\lim_{\ve\to0^+}\mathcal{D}_i^{\ve,\beta}=\mathcal{D}$, where
$$
\mathcal{D}=\max_{p\in\Omega}\text{dist}(p, \partial\Omega), \quad \mathcal{D}_i^{\ve,\beta}=\text{dist}(p_i^{\ve,\beta}, \partial\Omega),\quad i=1,2.$$
\item[$(b)$]  If $\beta<0$, then $\lim_{\ve\to0^+}\varphi(p_1^{\ve,\beta},p_2^{\ve,\beta})=\varphi(P_1^*, P_2^*)$, where
\begin{eqnarray*}
\varphi(P_1,P_2)=\min_{i=1,2}\{\min\{\sqrt{\lambda_i}|P_1-P_2|, \sqrt{\lambda_i}dist(P_i, \partial\Omega)\}\}
\end{eqnarray*}
and $(P_1^*, P_2^*)$ satisfies $\varphi(P_1^*,P_2^*)=\max_{(P_1,P_2)\in\Omega^2}\varphi(P_1,P_2)$.
\end{enumerate}
\end{theorem}

\vskip0.2in

\begin{remark}
The argument in \cite{LW05} for proving a similar result to $(b)$ of Theorem~\ref{thm0003} is based on the nondegenerate of $v_i^0$ ($i=1,2$) in the case $1\leq N\leq3$, where $v_i^0$ is given in Theorem~\ref{thm0002}.  To our best knowledge, it is still   unknow whether $v_i^0$ ($i=1,2$) is nondegenerate or not  in the case $N=4$, thus the method of  \cite{LW05} is invalid in our case.  Note that by Theorem~\ref{thm0002}, the spikes satisfy the phenomenon of phase separation for $\beta<0$ as $\ve\to0^+$ and such phenomenon is always related to the sign-changing solutions of the single equation.  Thus, inspired by \cite{DSW11}, where a similar situation was faced in the study of the locations of the spikes to a sign-changing solution of a single equation with jumping nonlinearities, we will introduce the function $\varphi_{b_1,b_2}^*(p_1^\ve,p_2^\ve)$ given by \eqref{eqn5453} and the set $\Lambda_{b_1,b_2}(p_1^\ve,p_2^\ve)$ given by \eqref{eqn5454} for our case to prove $(b)$ of Theorem~\ref{thm0003}.
\end{remark}

\begin{remark}
By some necessary modifications, Theorems~\ref{thm0001}--\ref{thm0003} can be generalized slightly to the following system:
\begin{equation*}
\left\{\aligned&-\ve^2\Delta u_1+\lambda_1u_1=\mu_1u_1^3+\alpha_1u_1^{p_1-1}+\beta u_2^2u_1\quad&\text{in }\Omega,\\
&-\ve^2\Delta u_2+\lambda_2u_2=\mu_2u_2^3+\alpha_2u_2^{p_2-1}+\beta u_1^2u_2\quad&\text{in }\Omega,\\
&u_1,u_2>0\quad\text{in }\Omega,\quad u_1=u_2=0\quad\text{on }\partial\Omega,\endaligned\right.
\end{equation*}
where $2<p_1,p_2<4$.  However, we do not want to trap into some unnecessary complex calculations which is coursed by $p_1\not=p_2$.  Due to this reason, we only give the proofs of Theorems~\ref{thm0001}--\ref{thm0003} for $p_1=p_2=p$.
\end{remark}

\vskip0.1in

\noindent{\bf\large Notations.} Throughout this paper, $C$ and $C'$ are indiscriminately used to denote various absolute positive constants.  We also list some notations used frequently below.
\begin{eqnarray*}
&\mathcal{B}_{u,p,\Omega}^{p}=\int_{\Omega}|u|^pdx,\quad&\overrightarrow{\mathbf{u}}=(u_1,u_2,\cdots,u_k),\\
&\overrightarrow{\mathbf{t}}\circ\overrightarrow{\mathbf{u}}=(t_1u_1, t_2u_2,\cdots, t_ku_k),\quad&\mathbb{B}_{r}(x)=\{y\in\bbr^4\mid|y-x|<r\},\\
&t\overrightarrow{\mathbf{u}}=(tu_1, tu_2,\cdots,tu_k),\quad&\overrightarrow{\mathbf{u}}_n=(u_1^n, u_2^n,\cdots,u_k^n),\\
&\bbr^4_{+}=\{x=(x_1,x_2,x_3,x_4)\in\bbr^4\mid x_4>0\},\quad&\bbr^{+}=(0, +\infty).
\end{eqnarray*}
$O(|\overrightarrow{\mathbf{b}}|)$ is used to denote the quantities who tend towards zero as $|\overrightarrow{\mathbf{b}}|\to0$, where $|\overrightarrow{\mathbf{b}}|$ is the usual norm in $\bbr^k$ of the vector $\overrightarrow{\mathbf{b}}$.


\section{The existence result}
let $\mathcal{H}_{i,\ve,\Omega}$ be the Hilbert space of $\h$ equipped with the inner product
$$
\langle u,v\rangle_{i,\ve,\Omega}=\int_{\Omega}\ve^2\nabla u\nabla v+\lambda_i uv dx.
$$
For   $i=1,2$, $\lambda_i>0$ and $\ve>0$, $\mathcal{H}_{i,\ve,\Omega}$ is a   Hilbert space  and the corresponding norm is  given by $\|u\|_{i,\ve,\Omega}=\langle u,u\rangle_{i,\ve,\Omega}^{\frac12}$.  Set $\mathcal{H}_{\ve,\Omega}=\mathcal{H}_{1,\ve,\Omega}\times\mathcal{H}_{2,\ve,\Omega}$.  Then $\mathcal{H}_{\ve,\Omega}$ is a Hilbert space with the inner product
$$
\langle \overrightarrow{\mathbf{u}},\overrightarrow{\mathbf{v}}\rangle_{\ve,\Omega}=\sum_{i=1}^2\langle u_i,v_i\rangle_{i,\ve,\Omega}.
$$
The corresponding norm is given by $\|\overrightarrow{\mathbf{u}}\|_{\ve,\Omega}=\langle \overrightarrow{\mathbf{u}},\overrightarrow{\mathbf{u}}\rangle_{\ve,\Omega}^{\frac12}$.  Here, $u_i,v_i$ are the $i$th component of $\overrightarrow{\mathbf{u}}, \overrightarrow{\mathbf{v}}$ respectively.  For the simplicity, we re-denote the footnote $_{i,1,\Omega}$ by $_{i,\Omega}$.

Let $\chi_\beta(s)$ be a smooth function in $[0, +\infty)$ such that $\chi_\beta(s)\equiv1$ for $\beta>-\sqrt{\mu_1\mu_2}$ and
\begin{equation}\label{eqnew8000}
\chi_\beta(s)=\left\{\aligned&1,\quad&0\leq s\leq1,\\
&0,\quad &s\geq2\\\endaligned\right.
\end{equation}
for $\beta\leq-\sqrt{\mu_1\mu_2}$.  Moreover, we also request $-2\leq\chi_\beta'(s)\leq0$ in $[0, +\infty)$.  Let
\begin{eqnarray}\label{eqn0008}
T^2\geq(16+\sum_{i=1}^2\frac{4p}{\mu_i(p-2)})\mathcal{S}^2
\end{eqnarray}
be a constant and define
\begin{eqnarray*}
\mathcal{J}_{\ve,\Omega,T}(\overrightarrow{\mathbf{u}})&=&\sum_{i=1}^2(\frac12\|u_i\|_{i,\ve,\Omega}^2-\frac{\alpha_i}{p}\mathcal{B}_{ u_i,p,\Omega}^{p}\chi_\beta\bigg(\frac{\|\overrightarrow{\mathbf{u}}\|_{\ve,\Omega}^2}{T^2\ve^4}\bigg)-\frac{\mu_i}{4}\mathcal{B}_{ u_i,4,\Omega}^{4})\\
&&-\frac{\beta}{2}\mathcal{B}_{|u_1|^{2}|u_2|^{2},1,\Omega},
\end{eqnarray*}
where $\mathcal{B}_{ u,p,\Omega}^{p}=\int_{\Omega}|u|^pdx$.
Then it is easy to see that $\mathcal{J}_{\ve,\Omega,T}(\overrightarrow{\mathbf{u}})$ is of $C^2$ and
\begin{eqnarray*}
\mathcal{J}_{\ve,\Omega,T}'(\overrightarrow{\mathbf{u}})\overrightarrow{\mathbf{v}}&=&\sum_{i=1}^2\bigg((1-\frac{2\alpha_i}{pT^2\ve^4}\mathcal{B}_{ u_i,p,\Omega}^{p}\chi_\beta'\bigg(\frac{\|\overrightarrow{\mathbf{u}}\|_{\ve,\Omega}^2}{T^2\ve^4}\bigg))\langle u_i,v_i\rangle_{i,\ve,\Omega}\\
&&-\alpha_i\chi_\beta\bigg(\frac{\|\overrightarrow{\mathbf{u}}\|_{\ve,\Omega}^2}{T^2\ve^4}\bigg)\int_{\Omega}|u_i|^{p-2}u_iv_idx-\mu_i\int_{\Omega}u_i^3v_idx\bigg)\\
&&-\beta\int_{\Omega}(u_2^2u_1v_1+u_1^2u_2v_2)dx
\end{eqnarray*}
for all $\overrightarrow{\mathbf{u}},\overrightarrow{\mathbf{v}}\in\mathcal{H}_{\ve,\Omega}$.

\begin{definition}\label{def0001}
 The vector $(u_1, u_2)=\overrightarrow{\mathbf{u}}\in\mathcal{H}_{\ve,\Omega}$ is called a nontrivial critical point of $\mathcal{J}_{\ve,\Omega,T}(\overrightarrow{\mathbf{u}})$ if $\mathcal{J}_{\ve,\Omega,T}'(\overrightarrow{\mathbf{u}})=0$ in $\mathcal{H}_{\ve,\Omega}^{-1}$ with $u_1\not\equiv0$ and $u_2\not\equiv0$.  $\overrightarrow{\mathbf{u}}\in\mathcal{H}_{\ve,\Omega}$ is called a semi-trivial critical point of $\mathcal{J}_{\ve,\Omega,T}(\overrightarrow{\mathbf{u}})$ if $\mathcal{J}_{\ve,\Omega,T}'(\overrightarrow{\mathbf{u}})=0$ in $\mathcal{H}_{\ve,\Omega}^{-1}$ with $u_1\not\equiv0$ or $u_2\not\equiv0$.  Here, $\mathcal{J}_{\ve,\Omega,T}'(\overrightarrow{\mathbf{u}})$ is the Fr\'echet derivative of $\mathcal{J}_{\ve,\Omega,T}(\overrightarrow{\mathbf{u}})$ and $\mathcal{H}_{\ve,\Omega}^{-1}$ is the dual space of $\mathcal{H}_{\ve,\Omega}$.  $\overrightarrow{\mathbf{u}}\in\mathcal{H}_{\ve,\Omega}$ is called a positive critical point of $\mathcal{J}_{\ve,\Omega,T}(\overrightarrow{\mathbf{u}})$ if $\overrightarrow{\mathbf{u}}$ is a nontrivial critical point and $u_i>0$ for both $i=1,2$.
\end{definition}

By Definition~\ref{def0001}, positive critical points of $\mathcal{J}_{\ve,\Omega,T}(\overrightarrow{\mathbf{u}})$ with $\|\overrightarrow{\mathbf{u}}\|_{\ve,\Omega}^2\leq\ve^4T^2$ are equivalent to the solutions of the system~$(\mathcal{S}_{\ve})$.

Without loss of generality, we assume $0\in\Omega$ and set $\Omega_{\ve}=\{x\in\bbr^4\mid \ve x\in\Omega\}$.  Now, let
\begin{eqnarray*}
\mathcal{N}_{\ve,\Omega,T}=\{\overrightarrow{\mathbf{u}}=(u_1,u_2)\in\widetilde{\mathcal{H}}_{\ve,\Omega}\mid \mathcal{J}_{\ve,\Omega,T}'(\overrightarrow{\mathbf{u}})\overrightarrow{\mathbf{u}}^1=\mathcal{J}_{\ve,\Omega,T}'(\overrightarrow{\mathbf{u}})\overrightarrow{\mathbf{u}}^2=0\},
\end{eqnarray*}
with $\widetilde{\mathcal{H}}_{\ve,\Omega}=(\mathcal{H}_{1,\ve,\Omega}\backslash\{0\})\times(\mathcal{H}_{2,\ve,\Omega}\backslash\{0\})$, $\overrightarrow{\mathbf{u}}^1=(u_1,0)$ and $\overrightarrow{\mathbf{u}}^2=(0, u_2)$.  Clearly, $\mathcal{N}_{\ve,\Omega,T}\not=\emptyset$.  We also define
\begin{eqnarray*}
\mathcal{N}'_{\ve,\Omega,T}=\{\overrightarrow{\mathbf{u}}=(u_1,u_2)\in\mathcal{H}_{\ve,\Omega}\backslash\{\overrightarrow{\mathbf{0}}\}\mid \mathcal{J}_{\ve,\Omega,T}'(\overrightarrow{\mathbf{u}})\overrightarrow{\mathbf{u}}=0\}.
\end{eqnarray*}

\vskip0.13in

\begin{lemma}\label{lemn0001}
We have $\mathcal{N}_{\ve,\Omega,T}\subset \mathcal{N}'_{\ve,\Omega,T}$ and $\mathcal{J}_{\ve,\Omega,T}(\overrightarrow{\mathbf{u}})$ is bounded from below on $\mathcal{N}'_{\ve,\Omega,T}$.
\end{lemma}
\begin{proof}
By the definitions of $\mathcal{N}_{\ve,\Omega,T}$ and $\mathcal{N}'_{\ve,\Omega,T}$, it is easy to see that $\mathcal{N}_{\ve,\Omega,T}\subset \mathcal{N}'_{\ve,\Omega,T}$.  Now let $\overrightarrow{\mathbf{u}}\in\mathcal{N}'_{\ve,\Omega,T}$, then by the construction of $\chi_\beta$ and the H\"older inequality, we have that
\begin{eqnarray}
\mathcal{J}_{\ve,\Omega,T}(\overrightarrow{\mathbf{u}})
&=&\mathcal{J}_{\ve,\Omega,T}(\overrightarrow{\mathbf{u}})-\frac1p\mathcal{J}_{\ve,\Omega,T}'(\overrightarrow{\mathbf{u}})\overrightarrow{\mathbf{u}}\notag\\
&=&\frac{p-2}{2p}\|\overrightarrow{\mathbf{u}}\|_{\ve,\Omega}^2+\frac{4-p}{4p}(\sum_{i=1}^2\mu_i\mathcal{B}_{ u_i,4,\Omega}^{4}+2\beta\mathcal{B}_{|u_1|^{2}|u_2|^{2},1,\Omega})\notag\\
&\geq&\frac{p-2}{2p}\|\overrightarrow{\mathbf{u}}\|_{\ve,\Omega}^2\label{eqn0001}\\
&\geq&0\notag
\end{eqnarray}
for $\beta>-\sqrt{\mu_1\mu_2}$ since $2<p<4$.  In the case $\beta\leq-\sqrt{\mu_1\mu_2}$, we have
\begin{eqnarray}
\mathcal{J}_{\ve,\Omega,T}(\overrightarrow{\mathbf{u}})
&=&\mathcal{J}_{\ve,\Omega,T}(\overrightarrow{\mathbf{u}})-\frac14\mathcal{J}_{\ve,\Omega,T}'(\overrightarrow{\mathbf{u}})\overrightarrow{\mathbf{u}}\notag\\
&=&(\frac{1}{4}+\sum_{i=1}^2\frac{\alpha_i}{2pT^2\ve^4}\mathcal{B}_{ u_i,p,\Omega}^{p}\chi_\beta'\bigg(\frac{\|\overrightarrow{\mathbf{u}}\|_{\ve,\Omega}^2}{T^2\ve^4}\bigg))\|\overrightarrow{\mathbf{u}}\|_{\ve,\Omega}^2\notag\\
&&-\frac{4-p}{4p}\sum_{i=1}^2\alpha_i\mathcal{B}_{ u_i,p,\Omega}^{p}\chi_\beta\bigg(\frac{\|\overrightarrow{\mathbf{u}}\|_{\ve,\Omega}^2}{T^2\ve^4}\bigg).\label{eqn0004}
\end{eqnarray}
It follows from the construction of $\chi_\beta$ that $\mathcal{J}_{\ve,\Omega,T}(\overrightarrow{\mathbf{u}})\geq0$ for $\|\overrightarrow{\mathbf{u}}\|_{\ve,\Omega}^2\geq2\ve^4T^2$.  For $\|\overrightarrow{\mathbf{u}}\|_{\ve,\Omega}^2<2\ve^4T^2$, we have that $\sum_{i=1}^2\|\overline{u}_i\|_{i,\Omega_\ve}^2<2T^2$ and
\begin{eqnarray}\label{eqn0006}
\mathcal{B}_{ u_i,p,\Omega}^{p}=\ve^4\mathcal{B}_{ \overline{u}_i,p,\bbr^4}^{p}\leq\ve^4\mathcal{S}_{p,i}^{-\frac{p}{2}}\|\overline{u}_i\|_{i,\bbr^4}^p
\leq\mathcal{S}_{p,i}^{-\frac{p}{2}}(2T^2)^{\frac{p-2}{2}}\|u_i\|_{i,\ve,\Omega}^2, \quad i=1,2.
\end{eqnarray}
Here, $\overline{u}_i(x)=u_i(\ve x)$ and we regard $\overline{u}_i$ as a function in $H^1(\bbr^4)$ by defining $\overline{u}_i\equiv0$ in $\bbr^4\backslash\Omega_\ve$ and $\mathcal{S}_{p,i}$ is  the best embedding constant from $\mathcal{H}_{i,\bbr^4}\to L^p(\bbr^4)$ defined by
\begin{equation*}
\mathcal{S}_{p,i}=\inf\{\|u_i\|_{i,\bbr^4}^2\mid u\in \mathcal{H}_{i,\bbr^4}, \mathcal{B}_{u_,p, \bbr^4}^2=1\}.
\end{equation*}
Thus, by \eqref{eqn0004}, we can see that
$\mathcal{J}_{\ve,\Omega}(\overrightarrow{\mathbf{u}})\geq-\ve^4C$ for all $\overrightarrow{\mathbf{u}}\in\mathcal{N}'_{\ve,\Omega,T}$ in the case $\beta\leq-\sqrt{\mu_1\mu_2}$.
\end{proof}

Set
\begin{eqnarray*}
c_{\ve,\Omega,T}=\inf_{\overrightarrow{\mathbf{u}}\in\mathcal{N}_{\ve,\Omega,T}}\mathcal{J}_{\ve,\Omega,T}(\overrightarrow{\mathbf{u}})\quad\text{and}\quad
c'_{\ve,\Omega,T}=\inf_{\overrightarrow{\mathbf{u}}\in\mathcal{N}'_{\ve,\Omega,T}}\mathcal{J}_{\ve,\Omega,T}(\overrightarrow{\mathbf{u}})
\end{eqnarray*}
Then by Lemma~\ref{lemn0001}, $c_{\ve,\Omega,T}\geq c'_{\ve,\Omega,T}>-\infty$.
\begin{lemma}\label{lem0002}
Assume  $\beta\not=0$ and let $\ve>0$ be small enough.  Then there exists $\alpha_T>0$ depending  on $T$ only  such that $\ve^4 C\leq c'_{\ve,\Omega,T}\leq c_{\ve,\Omega,T}\leq\ve^4 \sum_{i=1}^2\frac{1}{4\mu_i}\mathcal{S}^2$ in the following two cases:
\begin{enumerate}
\item $\beta>-\sqrt{\mu_1\mu_2}$ and $\alpha_1,\alpha_2>0$,
\item $\beta\leq-\sqrt{\mu_1\mu_2}$ and $\alpha_1,\alpha_2>0$ with $|\overrightarrow{\mathbf{\alpha}}|<\alpha_T$.
\end{enumerate}
\end{lemma}
\begin{proof}
Let $\mathbb{B}_{r_i}(x_i)\subset\Omega$ respectively for $i=1,2$ and $\mathbb{B}_{r_1}(x_1)\cap \mathbb{B}_{r_2}(x_2)=\emptyset$.  Then by Proposition~\ref{propA0001}, there exists $\ve_0>0$ such that for $\ve<\ve_0$, there exists a solution $U_{i,\ve}$ of equation~$(\mathcal{P}_{i,\ve})$ in $\mathbb{B}_{r_i}(x_i)$ such that $\mathcal{E}_{i,\ve,\mathbb{B}_{r_i}(x_i)}(U_{i,\ve})<\frac{\ve^4}{4\mu_i}\mathcal{S}^2$.
It follows from a standard argument that $\|U_{i,\ve}\|_{i,\ve,\mathbb{B}_{r_i}(x_i)}^2\leq\frac{\ve^4p}{2(p-2)\mu_i}\mathcal{S}^2$, $i=1,2$.  Now, we regard $U_{i,\ve}$ as a function in $\h$ by defining $U_{i,\ve}\equiv0$ in $\Omega\backslash\mathbb{B}_{r_i}(x_i)$ and set $\overrightarrow{\mathbf{U}}_\ve=(U_{1,\ve}, U_{2,\ve})$.  By $\mathbb{B}_{r_1}(x_1)\cap \mathbb{B}_{r_2}(x_2)=\emptyset$ and the construction of $\chi_\beta$, we can see from \eqref{eqn0008} that $\overrightarrow{\mathbf{U}}_\ve\in\mathcal{N}_{\ve,\Omega,T}$.  It follows that
\begin{eqnarray}\label{eqn0002}
c_{\ve,\Omega,T}\leq\mathcal{J}_{\ve,\Omega,T}(\overrightarrow{\mathbf{U}}_\ve)=\sum_{i=1}^2\mathcal{E}_{i,\ve}(U_{i,\ve})<\ve^4\sum_{i=1}^2\frac{1}{4\mu_i}\mathcal{S}^2
\end{eqnarray}
for all $\beta$.  Now let $\overrightarrow{\mathbf{u}}\in\mathcal{N}'_{\ve,\Omega,T}$, then by H\"older inequality, we can see that
\begin{eqnarray*}
\|\overline{u}_1\|_{1,\Omega_\ve}^2+\|\overline{u}_2\|_{2,\Omega_\ve}^2\leq C\mathcal{B}_{ \overline{u}_1,4,\Omega_\ve}^{4}+C\mathcal{B}_{ \overline{u}_2,4,\Omega_\ve}^{4}+2|\beta|\mathcal{B}_{ \overline{u}_1,4,\Omega_\ve}^{2}\mathcal{B}_{ \overline{u}_2,4,\Omega_\ve}^{2},
\end{eqnarray*}
where $\overline{u}_i(x)=u_i(\ve x)$.
It follows from the Sobolev inequality that
\begin{eqnarray}\label{eqn0022}
\mathcal{B}_{ \overline{u}_2,4,\Omega_\ve}^{2}+\mathcal{B}_{ \overline{u}_1,4,\Omega_\ve}^{2}\geq C_\beta>0,
\end{eqnarray}
Here, $C_\beta$ depends on  $\beta$ only.  Thanks to \eqref{eqn0001}, we have $c'_{\ve,\Omega,T}\geq \ve^4 C_\beta$ for all $\beta>-\sqrt{\mu_1\mu_2}$.  For $\beta\leq-\sqrt{\mu_1\mu_2}$, by \eqref{eqn0004} and \eqref{eqn0006}, we can see that there exists $\alpha_T>0$ such that $c'_{\ve,\Omega,T}\geq \ve^4 C_\beta$ with $|\overrightarrow{\mathbf{\alpha}}|<\alpha_T$.  Now, the proof is completed by \eqref{eqn0002} and the fact that $c_{\ve,\Omega,T}\geq c'_{\ve,\Omega,T}$.
\end{proof}

By the Ekeland's variational principle, there exists $\{\overrightarrow{\mathbf{u}}_n\}\subset\mathcal{N}_{\ve,\Omega,T}$ such that
\begin{enumerate}
\item[$(1)$]  $\mathcal{J}_{\ve,\Omega,T}(\overrightarrow{\mathbf{u}}_n)=c_{\ve,\Omega,T}+o_n(1)$,
\item[$(2)$]  $\mathcal{J}_{\ve,\Omega,T}(\overrightarrow{\mathbf{v}})\geq \mathcal{J}_{\ve,\Omega,T}(\overrightarrow{\mathbf{u}}_n)-\frac1n\|\overrightarrow{\mathbf{v}}-\overrightarrow{\mathbf{u}}_n\|_{\ve,\Omega}$ for all $\overrightarrow{\mathbf{v}}\in\mathcal{N}_{\ve,\Omega,T}$.
\end{enumerate}
There also exists $\{\overrightarrow{\mathbf{u}}'_n\}\subset\mathcal{N}'_{\ve,\Omega}$ such that
\begin{enumerate}
\item[$(1')$]  $\mathcal{J}_{\ve,\Omega,T}(\overrightarrow{\mathbf{u}}'_n)=c'_{\ve,\Omega,T}+o_n(1)$,
\item[$(2')$]  $\mathcal{J}_{\ve,\Omega,T}(\overrightarrow{\mathbf{v}})\geq \mathcal{J}_{\ve,\Omega,T}(\overrightarrow{\mathbf{u}}'_n)-\frac1n\|\overrightarrow{\mathbf{v}}-\overrightarrow{\mathbf{u}}'_n\|_{\ve,\Omega}$ for all $\overrightarrow{\mathbf{v}}\in\mathcal{N}'_{\ve,\Omega,T}$.
\end{enumerate}
Define
\begin{eqnarray*}
\Phi_n(\overrightarrow{\mathbf{t}})=\mathcal{J}_{\ve,\Omega,T}(\overrightarrow{\mathbf{t}}\circ\overrightarrow{\mathbf{u}}_n)\quad\text{and}\quad\widetilde{\Phi}_n(t)=\mathcal{J}_{\ve,\Omega,T}(t\overrightarrow{\mathbf{u}}'_n),
\end{eqnarray*}
where $\overrightarrow{\mathbf{t}}\circ\overrightarrow{\mathbf{u}}_n=(t_1u_{1}^n, t_2u_2^n)$ and $t\overrightarrow{\mathbf{u}}'_n=(t(u_1^n)', t(u_2^n)')$.
\begin{lemma}\label{lem0001}
Let $\ve>0$ be small enough, $\alpha_1,\alpha_2>0$ and $\alpha_T$ is given in Lemma~\ref{lem0002}.  Then
\begin{enumerate}
\item[$(i)$] $\|\overrightarrow{\mathbf{u}}_n\|_{\ve,\Omega}^2<\ve^4T^2$ and $\|\overrightarrow{\mathbf{u}}_n'\|_{\ve,\Omega}^2<\ve^4T^2$ in one of the following two cases:
    \begin{enumerate}
    \item $\beta>-\sqrt{\mu_1\mu_2}$,
    \item $\beta\leq-\sqrt{\mu_1\mu_2}$ and $|\overrightarrow{\mathbf{\alpha}}|<\alpha_T$.
    \end{enumerate}
\item[$(ii)$] There exists $\beta_0\in(0, \min\{\mu_1, \mu_2\})$ independent of $\ve$ such that $\{\overrightarrow{\mathbf{u}}_n\}$ is a $(PS)_{c_{\ve,\Omega,T}}$ sequence of $\mathcal{J}_{\ve,\Omega,T}(\overrightarrow{\mathbf{u}})$
     in one of the following two cases:
    \begin{enumerate}
    \item $-\sqrt{\mu_1\mu_2}<\beta<\beta_0$,
    \item $\beta\leq-\sqrt{\mu_1\mu_2}$ and $|\overrightarrow{\mathbf{\alpha}}|<\alpha_T$.
    \end{enumerate}
     That is, $$\mathcal{J}_{\ve,\Omega,T}(\overrightarrow{\mathbf{u}}_n)=c_{\ve,\Omega,T}+o_n(1);\quad \mathcal{J}_{\ve,\Omega,T}'(\overrightarrow{\mathbf{u}}_n)=o_n(1)   {\; strongly\; in\;  }  \mathcal{H}_{\ve,\Omega}^{-1}.$$  Moreover, $\Phi_n(\overrightarrow{\mathbf{1}})\geq\Phi_n(\overrightarrow{\mathbf{t}})$ for all $n\in\bbn$ and $\overrightarrow{\mathbf{t}}\in(\bbr^+)^2$.
\item[$(iii)$] $\{\overrightarrow{\mathbf{u}}'_n\}$ is a $(PS)_{c'_{\ve,\Omega,T}}$ sequence of $\mathcal{J}_{\ve,\Omega,T}(\overrightarrow{\mathbf{u}})$ for all $\beta>0$.  Moreover, $\widetilde{\Phi}_n(1)\geq\widetilde{\Phi}_n(t)$ for all $n\in\bbn$ and $t>0$.
\end{enumerate}
\end{lemma}
\begin{proof}
$(i)$\quad
We only give the proof for $\overrightarrow{\mathbf{u}}_n$ since that of $\overrightarrow{\mathbf{u}}_n'$ is similar.  By Lemma~\ref{lem0002}, we can see from \eqref{eqn0001} and $(1)$ that
\begin{eqnarray}\label{eqn0009}
\|\overrightarrow{\mathbf{u}}_n\|_{\ve,\Omega}^2<\ve^4\frac{2p}{p-2}\sum_{i=1}^2\frac{1}{4\mu_i}\mathcal{S}^2+o_n(1)
\end{eqnarray}
in the case $\beta>-\sqrt{\mu_1\mu_2}$.  For $\beta\leq-\sqrt{\mu_1\mu_2}$, by Lemma~\ref{lem0002} and $(1)$, we can see from \eqref{eqn0004} and \eqref{eqn0006} that
\begin{eqnarray*}
&&\ve^4\sum_{i=1}^2\frac{1}{4\mu_i}\mathcal{S}^2+o_n(1)\\
&>&(\frac{1}{4}+\frac{\alpha_i}{2pT^2\ve^4}\mathcal{S}_{p,i}^{-\frac{p}{2}}(2T^2)^{\frac{p-2}{2}}\|\overrightarrow{\mathbf{u}}_n\|_{\ve,\Omega}^2
\chi_\beta'\bigg(\frac{\|\overrightarrow{\mathbf{u}}_n\|_{\ve,\Omega}^2}{T^2\ve^4}\bigg))\|\overrightarrow{\mathbf{u}}_n\|_{\ve,\Omega}^2\\
&&-\frac{4-p}{4p}\sum_{i=1}^2\alpha_i\mathcal{S}_{p,i}^{-\frac{p}{2}}(2T^2)^{\frac{p-2}{2}}\|\overrightarrow{\mathbf{u}}_n\|_{\ve,\Omega}^2.
\end{eqnarray*}
By choosing $\alpha_T$ small enough if necessary, we have from the construction of $\chi_\beta$ that
\begin{eqnarray}\label{eqn0010}
\ve^4\sum_{i=1}^2\frac{1}{4\mu_i}\mathcal{S}^2+o_n(1)>\frac1{8}\|\overrightarrow{\mathbf{u}}_n\|_{\ve,\Omega}^2
\end{eqnarray}
for $|\overrightarrow{\mathbf{\alpha}}|<\alpha_T$ in the case $\beta\leq-\sqrt{\mu_1\mu_2}$.  Combining  \eqref{eqn0009}--\eqref{eqn0010} and \eqref{eqn0008}, we have $\|\overrightarrow{\mathbf{u}}_n\|_{\ve,\Omega}^2<\ve^4T^2$ in the following two cases:
\begin{enumerate}
\item $\beta>-\sqrt{\mu_1\mu_2}$ and $\alpha_1,\alpha_2>0$,
\item $\beta\leq-\sqrt{\mu_1\mu_2}$ and $\alpha_1,\alpha_2>0$ with $|\overrightarrow{\mathbf{\alpha}}|<\alpha_T$.
\end{enumerate}

$(ii)$\quad
Let $\overrightarrow{\mathbf{w}}\in\mathcal{H}_{\ve,\Omega}$.  For every $n\in\bbn$, we consider the system $\overrightarrow{\mathbf{\Psi}}_n(\overrightarrow{\mathbf{t}},l)=\overrightarrow{\mathbf{0}}$, where $\overrightarrow{\mathbf{\Psi}}_n(\overrightarrow{\mathbf{t}},l)=(\Psi_1^n(\overrightarrow{\mathbf{t}},l), \Psi_2^n(\overrightarrow{\mathbf{t}},l))$ with
\begin{eqnarray*}
\Psi_i^n(\overrightarrow{\mathbf{t}},l)&=&\|t_iu_i^n+lw_i\|_{i,\ve,\Omega}^2-\alpha_i\mathcal{B}_{t_iu_i^n+lw_i,p,\Omega}^{p}
-\mu_i\mathcal{B}_{t_iu_i^n+lw_i,4,\Omega}^{4}\\
&&-\beta\mathcal{B}_{|t_1u_1^n+lw_1|^{2}|t_2u_2^n+lw_2|^{2},1,\Omega}.
\end{eqnarray*}
Clearly, $\overrightarrow{\mathbf{\Psi}}_n(\overrightarrow{\mathbf{t}},l)$ is of $C^1$.  Moreover, since $\{\overrightarrow{\mathbf{u}}_n\}\subset\mathcal{N}_{\ve,\Omega,T}$, we also have that $\overrightarrow{\mathbf{\Psi}}_n(\overrightarrow{\mathbf{1}},0)=\overrightarrow{\mathbf{0}}$.  By a direct calculation, we can see from $(i)$ and the construction of $\chi_\beta$ that
\begin{eqnarray}
\frac{\partial\Psi_i^n(\overrightarrow{\mathbf{1}},0)}{\partial t_i}&=&2\|u_i^n\|_{i,\ve,\Omega}^2-p\alpha_i\mathcal{B}_{u_i^n,p,\Omega}^{p}
-4\mu_i\mathcal{B}_{u_i^n,4,\Omega}^{4}-2\beta\mathcal{B}_{|u_1^n|^{2}|u_2^n|^{2},1,\Omega}\notag\\
&=&-(p-2)\alpha_i\mathcal{B}_{u_i^n,p,\Omega}^{p}-2\mu_i\mathcal{B}_{u_i^n,4,\Omega}^{4}\label{eqn0011}
\end{eqnarray}
respectively for $i=1,2$ and
\begin{eqnarray}
\frac{\partial\Psi_1^n(\overrightarrow{\mathbf{1}},0)}{\partial t_2}=\frac{\partial\Psi_2^n(\overrightarrow{\mathbf{1}},0)}{\partial t_1}=-2\beta\mathcal{B}_{|u_1^n|^{2}|u_2^n|^{2},1,\Omega}.\label{eqn0012}
\end{eqnarray}
Set
\begin{eqnarray}\label{eqn0016}
\Theta_n=(\theta_{ij}^n)_{i,j=1,2}
\end{eqnarray}
with $\theta_{ij}^n=\frac{\partial\Psi_i^n(\overrightarrow{\mathbf{1}},0)}{\partial t_j}$.  Then by \eqref{eqn0011} and \eqref{eqn0012}, we can see from $p>2$ and the H\"older inequality that
\begin{eqnarray}\label{eqn0020}
\text{det}(\Theta_n)>(\mu_1\mu_2-\beta^2)\mathcal{B}_{u_1^n,4,\Omega}^{4}\mathcal{B}_{u_2^n,4,\Omega}^{4}.
\end{eqnarray}
If $-\sqrt{\mu_1\mu_2}<\beta<0$, then by the Sobolev and H\"older inequalities, we can see from $\{\overrightarrow{\mathbf{u}}_n\}\subset\mathcal{N}_{\ve,\Omega,T}$ and the construction of $\chi_\beta$ that
\begin{eqnarray}\label{eqn0007}
\ve^4\mathcal{B}_{\nabla \overline{u}_i^n,2,\Omega_\ve}^{2}\leq\|u_i^n\|_{i,\ve,\Omega}^2\leq C\mathcal{B}_{ u_i^n,4,\Omega}^{4}\leq\ve^4 C\mathcal{B}_{\nabla \overline{u}_i^n,2,\Omega_\ve}^{4}, \quad i=1,2,
\end{eqnarray}
 which implies
\begin{eqnarray}\label{eqn0013}
\mathcal{B}_{\overline{u}_i^n,4,\Omega_\ve}^{4}\geq C
\end{eqnarray}
in the case $-\sqrt{\mu_1\mu_2}<\beta<0$.  Here $\overline{u}_i(x)=u_i(\ve x)$.  On the other hand, when  $\beta>0$, by similar arguments as used in \eqref{eqn0007}, we can see that
\begin{eqnarray*}
&\|\overline{u}_1^n\|_{1,\Omega_\ve}^2\leq C\mathcal{B}_{ \overline{u}_1^n,4,\Omega_\ve}^{4}+\beta\mathcal{B}_{ \overline{u}_1^n,4,\Omega_\ve}^{2}\mathcal{B}_{ \overline{u}_2^n,4,\Omega_\ve}^{2},\\
&\|\overline{u}_2^n\|_{2,\Omega_\ve}^2\leq C\mathcal{B}_{ \overline{u}_2^n,4,\Omega_\ve}^{4}+\beta\mathcal{B}_{ \overline{u}_1^n,4,\Omega_\ve}^{2}\mathcal{B}_{ \overline{u}_2^n,4,\Omega_\ve}^{2}.
\end{eqnarray*}
It follows from the Sobolev inequality that
\begin{eqnarray}\label{eqn0003}
&C\leq \mathcal{B}_{ \overline{u}_1^n,4,\Omega_\ve}^{2}+\beta\mathcal{B}_{ \overline{u}_2^n,4,\Omega_\ve}^{2},\\
&C\leq \mathcal{B}_{ \overline{u}_2^n,4,\Omega_\ve}^{2}+\beta\mathcal{B}_{ \overline{u}_1^n,4,\Omega_\ve}^{2}.
\end{eqnarray}
By $(1)$ and \eqref{eqn0001}, we can see from Lemma~\ref{lem0002} that
\begin{eqnarray}
C'\geq\mathcal{B}_{ \overline{u}_1^n,4,\Omega_\ve}^{2}+\mathcal{B}_{ \overline{u}_2^n,4,\Omega_\ve}^{2}\label{eqn0005}
\end{eqnarray}
for $\beta>0$.
Combining \eqref{eqn0003}--\eqref{eqn0005}, we can see that there exists $\beta_0\in(0, \sqrt{\mu_1\mu_2})$ independent of $\ve$ such that
\begin{eqnarray}\label{eqn0014}
\mathcal{B}_{\overline{u}_i^n,4,\Omega_\ve}^{4}\geq C, \quad i=1,2,
\end{eqnarray}
in the case $0<\beta<\beta_0$.  By \eqref{eqn0020}, \eqref{eqn0013} and \eqref{eqn0014}, we have that $\text{det}(\Theta_n)\geq \ve^4C>0$ in the case $-\sqrt{\mu_1\mu_2}<\beta<\beta_0$.  For $\beta\leq-\sqrt{\mu_1\mu_2}$, by $\{\overrightarrow{\mathbf{u}}_n\}\subset\mathcal{N}_{\ve,\Omega,T}$ and \eqref{eqn0006}, we can see from the construction of $\chi_\beta$ that
\begin{eqnarray*}
(1-\alpha_i\mathcal{S}_{p,i}^{-\frac{p}{2}}(2T^2)^{\frac{p-2}{2}})\|u_i^n\|_{i,\ve,\Omega}^2\leq
\mu_i\mathcal{B}_{u_i^n,4,\Omega}^{4}+\beta\mathcal{B}_{|u_1^n|^{2}|u_2^n|^{2},1,\Omega}, \quad i=1,2.
\end{eqnarray*}
By taking $\alpha_T$ small enough if necessary, we can see from \eqref{eqn0013} and the Sobolev inequality that
\begin{eqnarray}\label{eqnew9999}
\mu_i\mathcal{B}_{u_i^n,4,\Omega}^{4}+\beta\mathcal{B}_{|u_1^n|^{2}|u_2^n|^{2},1,\Omega}\geq \ve^4C, \quad i=1,2,
\end{eqnarray}
hence,
\begin{eqnarray*}
\mu_1\mu_2\mathcal{B}_{u_1^n,4,\Omega}^{4}\mathcal{B}_{u_2^n,4,\Omega}^{4}-\beta^2\mathcal{B}^2_{|u_1^n|^{2}|u_2^n|^{2},1,\Omega}\geq \ve^8C
\end{eqnarray*}
for $|\overrightarrow{\mathbf{\alpha}}|<\alpha_T$ in the case $\beta\leq-\sqrt{\mu_1\mu_2}$.  Thus, by $p>2$, we also have that $\text{det}(\Theta_n)\geq \ve^8C>0$ for $|\overrightarrow{\mathbf{\alpha}}|<\alpha_T$ in the case $\beta\leq-\sqrt{\mu_1\mu_2}$.  Since we always have $\text{det}(\Theta_n)\geq \ve^8C>0$, by the implicit function theorem, there exist $\sigma>0$ and $\overrightarrow{\mathbf{t}}_n(l)\in C^1([-\sigma, \sigma], [\frac12, \frac32]^2)$ such that $\{\overrightarrow{\mathbf{t}}_n(l)\circ\overrightarrow{\mathbf{u}}_n+l\overrightarrow{\mathbf{w}}\}\subset\mathcal{N}_{\ve,\Omega,T}$ in the following two cases:
\begin{enumerate}
\item $-\sqrt{\mu_1\mu_2}<\beta<\beta_0$ and $\alpha_1,\alpha_2>0$,
\item $\beta\leq-\sqrt{\mu_1\mu_2}$ and $\alpha_1,\alpha_2>0$ with $|\overrightarrow{\mathbf{\alpha}}|<\alpha_T$.
\end{enumerate}
Since $(2)$ holds and $\text{det}(\Theta_n)\geq \ve^8C>0$ for the above two cases, by using the Taylor expansion to $\mathcal{J}_{\ve,\Omega,T}(\overrightarrow{\mathbf{t}}_n(l)\circ\overrightarrow{\mathbf{u}}_n+l\overrightarrow{\mathbf{w}})$ in a standard way (cf. \cite{CZ121}), we can obtain that $\mathcal{J}_{\ve,\Omega,T}'(\overrightarrow{\mathbf{u}}_n)=o_n(1)$ strongly in $\mathcal{H}_{\ve,\Omega}^{-1}$.  In what follows, we will show that $\Phi_n(\overrightarrow{\mathbf{1}})\geq\Phi_n(\overrightarrow{\mathbf{t}})$ for all $n\in\bbn$ and $\overrightarrow{\mathbf{t}}\in(\bbr^+)^2$.  Consider the system $\overrightarrow{\mathbf{\Gamma}}_n(\overrightarrow{\mathbf{t}}, \tau)=\overrightarrow{\mathbf{0}}$, where $\overrightarrow{\mathbf{\Gamma}}_n(\overrightarrow{\mathbf{t}}, \tau)=(\Gamma_1^n(\overrightarrow{\mathbf{t}},\tau), \Gamma_2^n(\overrightarrow{\mathbf{t}},\tau))$ with
\begin{eqnarray*}
\Gamma_i^n(\overrightarrow{\mathbf{t}},\tau)&=&\|t_iu_i^n\|_{i,\ve,\Omega}^2-\alpha_i\mathcal{B}_{t_iu_i^n,p,\Omega}^{p}
-\mu_i\mathcal{B}_{t_iu_i^n,4,\Omega}^{4}-\tau\mathcal{B}_{|t_1u_1^n|^{2}|t_2u_2^n|^{2},1,\Omega}.
\end{eqnarray*}
Clearly, $\overrightarrow{\mathbf{\Gamma}}_n$ is of $C^1$ with $\overrightarrow{\mathbf{\Gamma}}_n(\overrightarrow{\mathbf{s}}_n, 0)=\overrightarrow{\mathbf{0}}$ for some $\overrightarrow{\mathbf{s}}_n\in(\bbr^+)^2$ and $\overrightarrow{\mathbf{\Gamma}}_n(\overrightarrow{\mathbf{1}}, \beta)=\overrightarrow{\mathbf{0}}$.  Moreover,
\begin{eqnarray}
s_i^n\frac{\partial\Gamma_i^n(\overrightarrow{\mathbf{s}}_n, 0)}{\partial t_i}=-(p-2)\alpha_i\mathcal{B}_{s_i^nu_i^n,p,\Omega}^{p}-2\mu_i\mathcal{B}_{s_i^nu_i^n,4,\Omega}^{4}\label{eqnew1011}, \quad i=1,2,
\end{eqnarray}
and
\begin{eqnarray}
s_1^n\frac{\partial\Gamma_1^n(\overrightarrow{\mathbf{s}}_n, 0)}{\partial t_2}=s_2^n\frac{\partial\Gamma_2^n(\overrightarrow{\mathbf{s}}_n, 0)}{\partial t_1}=0.\label{eqnew1012}
\end{eqnarray}
Clearly, det$(\widetilde{\Theta}_n)>0$, where $\widetilde{\Theta}_n=(\widetilde{\theta}_{ij}^n)_{i,j=1,2}$ with $\widetilde{\theta}_{ij}^n=\frac{\partial\Gamma_i^n(\overrightarrow{\mathbf{s}}_n, 0)}{\partial t_j}$.  By the implicit function theorem, there exists a $C^1$ vector valued function $\overrightarrow{\mathbf{t}}_n(\tau)$ with $\overrightarrow{\mathbf{t}}_n(0)=\overrightarrow{\mathbf{s}}_n$ and $t_i^n(\tau)>0$ such that $\overrightarrow{\mathbf{\Gamma}}_n(\overrightarrow{\mathbf{t}}(\tau), \tau)=\overrightarrow{\mathbf{0}}$ for $\tau$ small enough.  Moreover, by \eqref{eqn0014}, $\tau$ is also uniformly for $n$.  Thus, by taking $\beta_0$ small enough if necessary, we can see that $\overrightarrow{\mathbf{1}}$ is the unique solution of $\overrightarrow{\mathbf{\Gamma}}_n(\overrightarrow{\mathbf{t}}, \beta)=\overrightarrow{\mathbf{0}}$ for $0<\beta<\beta_0$.  For $\beta<0$, we follow the idea in the proof of \cite[Lemma~2.2]{W17} by considering the following set
\begin{eqnarray*}
\mathcal{Z}_n=\{\tau\in[0, \beta]\mid \overrightarrow{\mathbf{\Gamma}}_n(\overrightarrow{\mathbf{t}}, \tau)=\overrightarrow{\mathbf{0}} \text{ is uniquely solvable in }(\bbr^+)^2\}.
\end{eqnarray*}
Clearly, $0\in\mathcal{Z}_n$.  Moreover, we claim that $\{\overrightarrow{\mathbf{t}}_n(\tau)\}_{\tau\in\mathcal{Z}_n}$ is both bounded from above and below away from $0$, where $\overrightarrow{\mathbf{t}}_n(\tau)$ is the unique solution of $\overrightarrow{\mathbf{\Gamma}}_n(\overrightarrow{\mathbf{t}}, \tau)=\overrightarrow{\mathbf{0}}$.  Indeed, by $(i)$ and \eqref{eqn0013}, it is easy to see from $\beta<0$ and $\overrightarrow{\mathbf{\Gamma}}_n(\overrightarrow{\mathbf{t}}_n(\tau), \tau)=\overrightarrow{\mathbf{0}}$ that $t_i^n(\tau)\geq C$ with some $C>0$ uniformly for $n$ and $\tau$.  On the other hand, if $t_1^n(\tau)\to+\infty$ and $t_2^n(\tau)\leq C$, then by $(i)$ and \eqref{eqn0013} once more, we can see that $\Gamma_1^n(\overrightarrow{\mathbf{t}}_n(\tau),\tau)\to-\infty$, which contradicts to $\Gamma_1^n(\overrightarrow{\mathbf{t}}_n(\tau),\tau)=0$.  Similarly, it is also impossible if $t_1^n(\tau)\leq C$ and $t_2^n(\tau)\to+\infty$.  It remains to exclude the case of $t_i^n(\tau)\to+\infty$.  Since $\text{det}(\Theta_n)\geq \ve^8C>0$, we can see that $\Theta_n$ which is given by \eqref{eqn0016} is strictly diagonally dominant for $\beta<0$ and the first eigenvalue of $\Theta_n$ is bounded below away from $0$ uniformly for $n$.  Thus, we must have $\sum_{i=1}^2\Gamma_i^n(\overrightarrow{\mathbf{t}}_n(\tau),\tau)\to-\infty$ in this case, which is also impossible due to the fact that $\overrightarrow{\mathbf{\Gamma}}_n(\overrightarrow{\mathbf{t}}_n(\tau), \tau)=\overrightarrow{\mathbf{0}}$.  Now, for every $\tau\in\mathcal{Z}_n$, by similar calculations for $\Theta_n$, we have
\begin{eqnarray}
t_i^n(\tau)\frac{\partial\Gamma_i^n(\overrightarrow{\mathbf{t}}_n(\tau),\tau)}{\partial t_i}=-(p-2)\alpha_i\mathcal{B}_{t_i^n(\tau)u_i^n,p,\Omega}^{p}-2\mu_i\mathcal{B}_{t_i^n(\tau)u_i^n,4,\Omega}^{4}\label{eqn1011}
\end{eqnarray}
and
\begin{eqnarray}
t_2^n(\tau)\frac{\partial\Gamma_1^n(\overrightarrow{\mathbf{t}}_n(\tau),\tau)}{\partial t_2}=t_1^n(\tau)\frac{\partial\Gamma_2^n(\overrightarrow{\mathbf{t}}_n(\tau),\tau)}{\partial t_1}=-2\beta\mathcal{B}_{|t_1^n(\tau)u_1^n|^{2}|t_2^n(\tau)u_2^n|^{2},1,\Omega}.\label{eqn1012}
\end{eqnarray}
By taking $\alpha_T$ small enough if necessary for $\beta\leq-\sqrt{\mu_1\mu_2}$, we can see from the similar calculations for $\Theta_n$ that
det$(\widehat{\Theta}_n(\tau))>0$ for all $n$ and $\tau$, where $\widehat{\Theta}_n(\tau)=(\widehat{\theta}_{ij}^n(\tau))_{i,j=1,2}$ with $\widehat{\theta}_{ij}^n(\tau)=\frac{\partial\Gamma_i^n(\overrightarrow{\mathbf{t}}_n(\tau),\tau)}{\partial t_j}$.  Now, applying the implicit function theorem, we can extend $\overrightarrow{\mathbf{t}}_n(\tau)$ to $\tau=0$ for all $n$, since $\{\overrightarrow{\mathbf{t}}_n(\tau)\}_{\tau\in\mathcal{Z}_n}$ is both bounded from above and below away from $0$.  Since $\overrightarrow{\mathbf{t}}_n(0)=\overrightarrow{\mathbf{s}}_n$ is unique, by similar arguments as used in the proof of \cite[Lemma~2.2]{W17}, we can see that $\overrightarrow{\mathbf{1}}$ is the unique solution of $\overrightarrow{\mathbf{\Gamma}}_n(\overrightarrow{\mathbf{t}}, \beta)=\overrightarrow{\mathbf{0}}$, which implies that $\overrightarrow{\mathbf{1}}$ is the unique critical point of $\Phi_n(\overrightarrow{\mathbf{t}})$ for all $n\in\bbn$.  Recall that
$$\mu_1\mu_2\mathcal{B}_{u_1^n,4,\Omega}^{4}\mathcal{B}_{u_2^n,4,\Omega}^{4}-\beta^2\mathcal{B}^2_{|u_1^n|^{2}|u_2^n|^{2},1,\Omega}\geq \ve^8C, \quad  \mathcal{B}_{u_i^n,4,\Omega}^{4}\geq \ve^4C, \quad i=1,2,$$
 we can see that $\Phi_n(\overrightarrow{\mathbf{t}})\to-\infty$ as $|\overrightarrow{\mathbf{t}}|\to+\infty$.  It follows that $\Phi_n(\overrightarrow{\mathbf{t}})$ has a maximum point in $(\overline{\bbr^+})^2$ for all $n\in\bbn$.  Note that $p>2$, thus by a standard argument, either $\frac{\partial\Phi_n(\overrightarrow{\mathbf{t}})}{\partial t_1}>0$ or $\frac{\partial\Phi_n(\overrightarrow{\mathbf{t}})}{\partial t_2}>0$ near $\partial\overline{(\bbr^+)^2}$.  Hence, $\overrightarrow{\mathbf{1}}$ is the unique maximum point of $\Phi_n(\overrightarrow{\mathbf{t}})$ in $\overline{(\bbr^+)^2}$ and $\Phi_n(\overrightarrow{\mathbf{1}})\geq\Phi_n(\overrightarrow{\mathbf{t}})$ for all $n\in\bbn$ and $\overrightarrow{\mathbf{t}}\in(\bbr^+)^2$.

\vskip0.3in

$(iii)$\quad Let $\overrightarrow{\mathbf{w}}\in\mathcal{H}_{\ve,\Omega}$.  For every $n\in\bbn$, we consider the equation
\begin{eqnarray*}
\Upsilon_n(t,l)&=&\sum_{i=1}^2(\|tu_i^n+lw_i\|_{i,\ve,\Omega}^2-\alpha_i\mathcal{B}_{tu_i^n+lw_i,p,\Omega}^{p}
-\mu_i\mathcal{B}_{tu_i^n+lw_i,4,\Omega}^{4})\\
&&-2\beta\mathcal{B}_{|tu_1^n+lw_1|^{2}|tu_2^n+lw_2|^{2},1,\Omega}.
\end{eqnarray*}
Clearly, $\Upsilon_n(t,l)$ is of $C^1$.  Moreover, since $\{\overrightarrow{\mathbf{u}}_n'\}\subset\mathcal{N}'_{\ve,\Omega,T}$, we also have from $(i)$ and the construction of $\chi_\beta$ given by \eqref{eqnew8000} that $\Upsilon_n(1,0)=0$.  By a direct calculation, we can see that
\begin{eqnarray*}
\frac{\partial\Upsilon_n(1,0)}{\partial t}=-\sum_{i=1}^2((p-2)\alpha_i\mathcal{B}_{u_i^n,p,\Omega}^{p}+2\mu_i\mathcal{B}_{u_i^n,4,\Omega}^{4})-4\beta\mathcal{B}_{|u_1^n|^{2}|u_2^n|^{2},1,\Omega}.
\end{eqnarray*}
Since $p>2$, by a similar argument as used for \eqref{eqn0022}, we can show that $\frac{\partial\Upsilon_n(1,0)}{\partial t}\leq-C<0$ for all $\beta>0$.  Now, by the implicit function theorem, there exist $\sigma_n>0$ and $t_n(l)\in C^1([-\sigma_n, \sigma_n], [\frac12, \frac32])$ such that $\{t_n(l)\overrightarrow{\mathbf{u}}'_n+l\overrightarrow{\mathbf{w}}\}\subset\mathcal{N}'_{\ve,\Omega,T}$ for all $\beta>0$.  By applying the Taylor expansion to $\mathcal{J}_{\ve,\Omega,T}(t_n(l)\overrightarrow{\mathbf{u}}'_n+l\overrightarrow{\mathbf{w}})$, we can show that $\mathcal{J}'_{\ve,\Omega,T}(\overrightarrow{\mathbf{u}}'_n)=o_n(1)$ strongly in $\mathcal{H}_{\ve,\Omega}^{-1}$. Since $p>2$ and $\beta>0$, we can prove  $\widetilde{\Phi}_n(1)\geq\widetilde{\Phi}_n(t)$ for all $n\in\bbn$ and $t>0$ in a standard way.
\end{proof}

We also need the following important energy estimates.
\begin{lemma}\label{lem0003}
Let $\alpha_T$ and $\beta_0$ are respectively given by Lemmas~\ref{lem0002} and \ref{lem0001}.  Then for $\ve>0$ small enough and $\alpha_1,\alpha_2>0$,
\begin{enumerate}
\item[$(i)$] we have
$$
c_{\ve,\Omega,T}<\min\{d_{1,\ve,\Omega}+\frac{\ve^4}{4\mu_2}\mathcal{S}^2, d_{2,\ve,\Omega}+\frac{\ve^4}{4\mu_1}\mathcal{S}^2, A_\ve\}-\ve^4C_\beta
$$
under one of  the following two cases:
\begin{enumerate}
\item $-\sqrt{\mu_1\mu_2}<\beta<\beta_0$,
\item $\beta\leq-\sqrt{\mu_1\mu_2}$ and $|\overrightarrow{\mathbf{\alpha}}|<\alpha_T$.
\end{enumerate}
\item[$(ii)$] there exists $\beta_1>\max\{\mu_1, \mu_2\}$ independent of $\ve$ such that
$$
c'_{\ve,\Omega,T}<\min\{d_{1,\ve,\Omega},d_{2,\ve,\Omega}, A_\ve\}-\ve^4C_\beta
$$
for $\beta>\beta_1$.
\end{enumerate}
Here, $C_\beta$ depends only on $\beta$.
\end{lemma}
\begin{proof}
$(i)$ By Proposition~\ref{propA0001}, $\widetilde{U}_{i,\ve}$ is the ground state solution of the following equation
$$
\left\{\aligned&-\ve^2\Delta u+\lambda_iu=\mu_iu^3+\alpha_iu^{p-1}\quad&\text{in }\Omega,\\
&u>0\quad\text{in }\Omega,\quad&u=0\quad\text{on }\partial\Omega. \endaligned\right.
$$
Then $\mathcal{E}_{i,\ve,\Omega}(\widetilde{U}_{i,\ve})=d_{i,\ve,\Omega}, i=1,2$.  Since $d_{i,\ve,\Omega}<\frac{\ve^4}{4\mu_i}\mathcal{S}^2$ and $2<p<4$, by a standard argument, we can see that
\begin{eqnarray}\label{eqn0025}
\mathcal{B}_{\widetilde{U}_{i,\ve},4,\Omega}^{4}<\frac{\ve^4p}{(4-p)\mu_i}\mathcal{S}^2\quad\text{and}\quad
\|\widetilde{U}_{i,\ve}\|_{i,\ve,\Omega}^2<\frac{\ve^4p}{2(p-2)\mu_i}\mathcal{S}^2, \quad i=1,2.
\end{eqnarray}
On the other hand, by a similar argument as used for \eqref{eqn0007}, we also have that
\begin{eqnarray}\label{eqn0021}
\|\widetilde{U}_{i,\ve}\|_{i,\ve,\Omega}^2\geq\ve^4 C\quad\text{and}\quad\mathcal{B}_{\widetilde{U}_{i,\ve},4,\Omega}^{4}\geq\ve^4 C.
\end{eqnarray}
Let $x_R\in\Omega_\ve$ and $R>0$ satisfy $\mathbb{B}_{3R}(x_R)\subset\Omega_\ve$. Take $\Psi\in C_0^2(\mathbb{B}_2(0))$ satisfying  $0\leq\Psi(x)\leq1$ and $\Psi(x)\equiv1$ in $\mathbb{B}_1(0)$.  Set $\varphi^*_R(x)=\Psi\bigg(\frac{x-x_R}{R}\bigg)$ for $x\in \mathbb{B}_{2R}(x_R)$ and $\varphi^*_R(x)=0$ for $x\in\Omega_\ve\backslash \mathbb{B}_{2R}(x_R)$.
Then $\varphi^*_R(x)\in C_0^2(\Omega_\ve)$ and
\begin{eqnarray*}
\varphi_R^*(x)=\left\{\aligned &1,\quad&x\in \mathbb{B}_{R}(x_R);\\
&0,\quad&x\in\Omega_\ve\backslash \mathbb{B}_{2R}(x_R),\endaligned\right.
\end{eqnarray*}
and $|\nabla \varphi_R^*(x)|\leq\frac{C}{R}$.  Let
\begin{eqnarray*}
V_{\sigma}(x)=\frac{2\sqrt{2}\sigma}{(\sigma^2+|x|^2)},\quad  v_{\sigma}(x)=\varphi_R^*(x)V_{\sigma}(\frac{x-x_R}{R}).
\end{eqnarray*}
Then it is well known that $v_{\sigma}\rightharpoonup0$ weakly in $H^1_0(\Omega_\ve)\cap L^4(\Omega_\ve)$ and $v_{\sigma}\to0$ strongly in $L^r(\Omega_\ve)$ for all $1\leq r<4$ as $\sigma\to0$. Moreover, by a well known calculation, we also have that
\begin{eqnarray}\label{eq0072}
\mathcal{B}_{\nabla v_\sigma,2,\Omega_\ve}^2=\mathcal{S}^{2}+O(\sigma^{2}),\quad\mathcal{B}_{v_\sigma,4,\Omega_\ve}^{4}=\mathcal{S}^{2}+O(\sigma^{4})
\end{eqnarray}
and
\begin{eqnarray}\label{eq0073}
\mathcal{B}_{v_\sigma,p,\Omega_\ve}^{p}\geq C(\sigma^{4-p}+\sigma^p),\quad\mathcal{B}_{v_\sigma,2,\Omega_\ve}^{2}\leq C'\sigma^2|\text{ln}\sigma|+O(\sigma^2).
\end{eqnarray}
Since $0\in\Omega$, we must have $\Omega_{\ve_1}\subset\Omega_{\ve_2}$ with $\ve_1>\ve_2$, which implies that $O(\sigma^2)$, $O(\sigma^4)$ and the constants $C,C'$ in \eqref{eq0072} and \eqref{eq0073} can be chosen such that they are independent of $\ve<1$.
Let us consider the following system
\begin{eqnarray*}
\left\{\aligned\eta_1(t_1,t_2)&=\|t_1\widetilde{U}_{1,\ve}\|_{1,\ve,\Omega}^2-\alpha_1\mathcal{B}_{t_1\widetilde{U}_{1,\ve},p,\Omega}^{p}
-\mu_1\mathcal{B}_{t_1\widetilde{U}_{1,\ve},4,\Omega}^{4}-\beta\mathcal{B}_{|t_1\widetilde{U}_{1,\ve}|^2|t_2u_\sigma|^2,1,\Omega},\\
\eta_2(t_1,t_2)&=\|t_2u_\sigma\|_{2,\ve,\Omega}^2-\alpha_2\mathcal{B}_{t_2u_\sigma,p,\Omega}^{p}
-\mu_2\mathcal{B}_{t_2u_\sigma,4,\Omega}^{4}-\beta\mathcal{B}_{|t_1\widetilde{U}_{1,\ve}|^2|t_2u_\sigma|^2,1,\Omega},
\endaligned\right.
\end{eqnarray*}
where $u_{\sigma}(x)=v_\sigma(\ve^{-1}x)$.  Set
\begin{eqnarray*}
t_1^*=\bigg(\frac{\|\widetilde{U}_{1,\ve}\|_{1,\ve,\Omega}^2}{\mu_1\mathcal{B}_{\widetilde{U}_{1,\ve},4,\Omega}^{4}}\bigg)^{\frac12}\quad\text{and}\quad
t_2^*=\bigg(\frac{\|u_\sigma\|_{2,\ve,\Omega}^2}{\mu_2\mathcal{B}_{u_\sigma,4,\Omega}^{4}}\bigg)^{\frac12}.
\end{eqnarray*}
By \eqref{eqn0025} and \eqref{eqn0021}, we have $t_1^*\leq C$.  Moreover, by \eqref{eq0072}, we also have $t_2^*\leq C$ for $\sigma$ small enough.  Note that by classical elliptic regularity theorem, we have $\widetilde{U}_{1,\ve}\in C^2$.  Thus, by \eqref{eq0073}, $2<p<4$ and Proposition~\ref{propA0001}, we can see that
\begin{eqnarray*}
&&|\beta|\mathcal{B}_{|t_1^*\widetilde{U}_{1,\ve}|^2|t_2u_\sigma|^2,1,\Omega}-\alpha_1\mathcal{B}_{t_1^*\widetilde{U}_{1,\ve},p,\Omega}^{p}\\
& & \leq \ve^{4}\|\widetilde{U}_{1,\ve}\|_{L^\infty(\Omega)}^2C|\beta|(\sigma^2|\text{ln}\sigma|+O(\sigma^2))-\alpha_1\mathcal{B}_{t_1\widetilde{U}_{1,\ve},p,\Omega}^{p}\\
& &<0
\end{eqnarray*}
for all $t_2\in[0, t_2^*]$ and
\begin{eqnarray*}
&&|\beta|\mathcal{B}_{|t_1\widetilde{U}_{1,\ve}|^2|t_2^*u_\sigma|^2,1,\Omega}-\alpha_2\mathcal{B}_{t_2^*u_\sigma,p,\Omega}^{p}\\
&&\leq \ve^{4}(\|\widetilde{U}_{1,\ve}\|_{L^\infty(\Omega)}^2C|\beta|(\sigma^2|\text{ln}\sigma|+O(\sigma^2))-C'(\sigma^{4-p}+\sigma^p))\\
&&<0
\end{eqnarray*}
for all $t_1\in[0, t_1^*]$ with $\sigma$ small enough.  It follows that $\eta_1(t_1^*,t_2)<0$ for all $t_2\in[0, t_2^*]$ and $\eta_2(t_1,t_2^*)<0$ for all $t_1\in[0, t_1^*]$ with $\sigma$ small enough.  On the other hand, by similar arguments as used in \eqref{eqn0007}, we have
\begin{eqnarray*}
\eta_1(t_1,t_2)\geq t_1^2(\|\widetilde{U}_{1,\ve}\|_{1,\ve,\Omega}^2-Ct_1^2\mathcal{B}_{\widetilde{U}_{1,\ve},4,\Omega}^{4}
-\beta\mathcal{B}_{|\widetilde{U}_{1,\ve}|^2|t_2^*u_\sigma|^2,1,\Omega})
\end{eqnarray*}
for all $t_2\in[0, t_2^*]$ and
\begin{eqnarray*}
\eta_2(t_1,t_2)\geq t_2^2(\|u_\sigma\|_{2,\ve,\Omega}^2-Ct_2^2\mathcal{B}_{u_\sigma,4,\Omega}^{4}
-\beta\mathcal{B}_{|t_1^*\widetilde{U}_{1,\ve}|^2|u_\sigma|^2,1,\Omega})
\end{eqnarray*}
for all $t_1\in[0, t_1^*]$ with $\beta>0$ while
\begin{eqnarray*}
\eta_1(t_1,t_2)\geq t_1^2(\|\widetilde{U}_{1,\ve}\|_{1,\ve,\Omega}^2-Ct_1^2\mathcal{B}_{\widetilde{U}_{1,\ve},4,\Omega}^{4})
\end{eqnarray*}
for all $t_2\in[0, t_2^*]$ and
\begin{eqnarray*}
\eta_2(t_1,t_2)\geq t_2^2(\|u_\sigma\|_{2,\ve,\Omega}^2-Ct_2^2\mathcal{B}_{u_\sigma,4,\Omega}^{4})
\end{eqnarray*}
for all $t_1\in[0, t_1^*]$ with $\beta<0$.  Since \eqref{eqn0025}--\eqref{eq0072} hold, by taking $\beta_0$ small enough if necessary and using a standard argument, we can see that there exist $t_1^{**}, t_2^{**}\geq C'$ such that $\eta_1(t_1^{**},t_2)>0$ for all $t_2\in[t_2^{**}, t_2^*]$ and $\eta_2(t_1,t_2^{**})>0$ for all $t_1\in[t_1^{**}, t_1^*]$ with $\beta<\beta_0$ and $\sigma$ small enough.  Now, applying Miranda's theorem (cf. \cite[Lemma~3.1]{CLZ141}), we can see that there exists $(\widetilde{t}_{1,\ve}, \widetilde{t}_{2,\ve})\in[t_1^{**}, t_1^*]\times[t_2^{**}, t_2^*]$ such that $\eta_1(\widetilde{t}_{1,\ve},\widetilde{t}_{2,\ve})=\eta_2(\widetilde{t}_{1,\ve},\widetilde{t}_{2,\ve})=0$.  It follows from \eqref{eqn0008}, \eqref{eqn0025} and \eqref{eqn0021} that $\overrightarrow{\mathbf{t}}_\ve\circ\overrightarrow{\widetilde{\mathbf{U}}}_\ve\in\mathcal{N}_{\ve,\Omega,T}$ for $\sigma$ small enough, where $\overrightarrow{\mathbf{t}}_\ve=(\widetilde{t}_{1,\ve}, \widetilde{t}_{2,\ve})$ and $\overrightarrow{\widetilde{\mathbf{U}}}_\ve=(\widetilde{U}_{1,\ve}, u_\sigma)$.   {By the choice of $u_\sigma$ and Proposition~\ref{propA0001} in the Appendix, } we have
\begin{eqnarray*}
|\beta|\mathcal{B}_{|\widetilde{t}_{1,\ve}\widetilde{U}_{1,\ve}|^2|\widetilde{t}_{1,\ve}u_\sigma|^2,1,\Omega}
\leq|\beta|C\mathcal{B}_{u_\sigma,2,\Omega}^2.
\end{eqnarray*}
Since $2<p<4$, by a standard argument we can see that
\begin{eqnarray*}
c_{\ve,\Omega,T}&\leq&\mathcal{J}_{\ve,\Omega,T}(\overrightarrow{\mathbf{t}}_\ve\circ\overrightarrow{\widetilde{\mathbf{U}}}_\ve)\notag\\
&\leq&\mathcal{E}_{1,\ve}(\widetilde{U}_{1,\ve})+\frac{\ve^4}{4\mu_2}\mathcal{S}^2+\ve^4(C|\beta|\sigma^2|\text{ln}\sigma|+|\beta|O(\sigma^2)-C'\sigma^{4-p})\notag\\
&<&d_{1,\ve,\Omega}+\frac{\ve^4}{4\mu_2}\mathcal{S}^2-\ve^4C_\beta\label{eqn0099}
\end{eqnarray*}
with $\sigma$ small enough, where $C_\beta$ is only dependent on $\beta$.  Similarly, we can also show that
$$
c_{\ve,\Omega,T}<d_{2,\ve,\Omega}+\frac{\ve^4}{4\mu_1}\mathcal{S}^2-\ve^4C_\beta.
$$
It remains to show $c_{\ve,\Omega,T}<A_\ve-\ve^4C$.  For $\beta<0$, by Proposition~\ref{propAnew0001}, we have $A_\ve=\frac{\ve^4}{4\mu_1}\mathcal{S}^2+\frac{\ve^4}{4\mu_2}\mathcal{S}^2$.  Note that $d_{i,\ve,\Omega}<\frac{\ve^4}{4\mu_i}\mathcal{S}^2-\ve^4C$, thus, we must have  $c_{\ve,\Omega,T}<A_\ve-\ve^4C$ for $\beta<0$.  For $\beta>0$, let
$$
\overrightarrow{\overline{\mathbf{U}}}_{\sigma}=(\sqrt{k}_1u_{\sigma}, \sqrt{k}_2u_{\sigma})
$$
and consider the following system
\begin{eqnarray*}
\eta_i(t_1,t_2)&=&\|t_i\sqrt{k}_iu_{\sigma}\|_{i,\ve,\Omega}^2-\alpha_i\mathcal{B}_{t_i\sqrt{k}_iu_{\sigma},p,\Omega}^{p}-\mu_i\mathcal{B}_{t_i\sqrt{k}_iu_{\sigma},4,\Omega}^{4}\\
&&-\beta\mathcal{B}_{|t_1\sqrt{k}_1u_{\sigma}|^2|t_2\sqrt{k}_2u_{\sigma}|^2,1,\Omega},
\end{eqnarray*}
where $k_1,k_2$ satisfy \eqref{eqnew0003}.
Since \eqref{eqn0008}, \eqref{eq0072}--\eqref{eq0073} hold, taking into account $k_i\to\frac{1}{\mu_i}$ as $\beta\to0$, we can apply Miranda's theorem similar to that for $\eta_i(t_1,t_2)$ to show that there exists $(\overline{t}_{1,\ve}, \overline{t}_{2,\ve})\in[C', C]\times[C', C]$ such that $\overrightarrow{\overline{\mathbf{t}}}_\ve\circ\overrightarrow{\overline{\mathbf{U}}}_\sigma\in\mathcal{N}_{\ve,\Omega,T}$ for $\beta>0$ small enough.  By taking $\beta_0$ small enough if necessary, we can see that $\overrightarrow{\overline{\mathbf{t}}}_\ve\circ\overrightarrow{\overline{\mathbf{U}}}_\sigma\in\mathcal{N}_{\ve,\Omega,T}$ for $0<\beta<\beta_0$.  Thus, we have from \eqref{eq0072} and \eqref{eq0073} once more that
\begin{eqnarray}
c_{\ve,\Omega,T}&\leq&\mathcal{J}_{\ve,\Omega,T}(\overrightarrow{\overline{\mathbf{t}}}_\ve\circ\overrightarrow{\overline{\mathbf{U}}}_\sigma)\notag\\
&=&\ve^4\mathcal{S}^2\bigg(\sum_{i=1}^2(\frac{(\overline{t}_{i,\ve}\sqrt{k_i})^2}{2}-\frac{\mu_i(\overline{t}_{i,\ve}\sqrt{k_i})^4}{4})-\frac{\beta(\overline{t}_{1,\ve}\sqrt{k_1})^2(\overline{t}_{2,\ve}\sqrt{k_2})^2}{2}\bigg)\notag\\
&&+\ve^4(C\sigma^2|\text{ln}\sigma|+O(\sigma^2)-C(\sigma^{4-p}+\sigma^p)).\label{eqn0023}
\end{eqnarray}
Since $\beta_0<\sqrt{\mu_1\mu_2}$, we can see that $k_1,k_2$ is the unique one satisfying \eqref{eqnew0003} for $0<\beta<\beta_0$.  Now by $2<p<4$ and a similar argument as used in the proof of \cite[Lemma~3.1]{WWZ17}, we have from \eqref{eqn0023} and \eqref{eqnew0003} once more that
\begin{eqnarray*}
c_{\ve,\Omega,T}&\leq&\ve^4\mathcal{S}^2\bigg(\sum_{i=1}^2(\frac{k_i}{2}-\frac{\mu_ik_i^2}{4})-\frac{\beta k_1k_2}{2}\bigg)\notag\\
&&+\ve^4(C\sigma^2|\text{ln}\sigma|+O(\sigma^2)-C(\sigma^{4-p}+\sigma^p))\notag\\
&<&\ve^4\frac{k_1+k_2}{4}\mathcal{S}^2-\ve^4C\label{eqn0100}
\end{eqnarray*}
for $\sigma$ small enough.  It follows from Proposition~\ref{propAnew0001} that $c_{\ve,\Omega}<A_\ve-\ve^4C$ for $0<\beta<\beta_0$.  In a word, we finally have that
$$
c_{\ve,\Omega,T}<\min\{d_{1,\ve,\Omega}+\frac{\ve^4}{4\mu_2}\mathcal{S}^2, d_{2,\ve,\Omega}+\frac{\ve^4}{4\mu_1}\mathcal{S}^2, A_\ve\}-\ve^4C_\beta
$$
in the following two cases:
\begin{enumerate}
\item $-\sqrt{\mu_1\mu_2}<\beta<\beta_0$ and $\alpha_1,\alpha_2>0$,
\item $\beta\leq-\sqrt{\mu_1\mu_2}$ and $\alpha_1,\alpha_2>0$ with $|\overrightarrow{\mathbf{\alpha}}|<\alpha_T$,
\end{enumerate}
where $C_\beta$ is dependent on $\beta$ only.

\vskip0.13in

$(ii)$\quad Let $\overrightarrow{\mathbf{U}^*}_\ve=(\widetilde{U}_{1,\ve}, \widetilde{U}_{1,\ve})$.  Since $p>2$, by a standard argument, we can see from \eqref{eqn0008} and \eqref{eqn0025} that there exists $t_\beta>0$ such that $t_\beta\overrightarrow{\mathbf{U}^*}_\ve\in\mathcal{N}'_{\ve,\Omega,T}$.  We claim that $t_\beta\to0$ as $\beta\to+\infty$.  Indeed, suppose the contrary, then without loss of generality, we may assume that $t_\beta\geq C$ for $\beta$ large enough.  It follows from \eqref{eqn0025}--\eqref{eqn0021} that
\begin{eqnarray*}
c'_{\ve,\Omega,T}\leq\mathcal{J}_{\ve,\Omega,T}(t_\beta\overrightarrow{\mathbf{U}^*}_\ve)\leq\ve^4t_\beta^2(C-\beta C')\to-\infty
\end{eqnarray*}
as $\beta\to+\infty$, which contradicts to Lemma~\ref{lem0002}.  Now, since $t_\beta\to0$ as $\beta\to+\infty$ and $p>2$, we can see from $t_\beta\overrightarrow{\mathbf{U}^*}_\ve\in\mathcal{N}'_{\ve,\Omega,T}$ and \eqref{eqn0025} and the H\"older inequality that
\begin{eqnarray*}
\sum_{i=1}^2\|\widetilde{U}_{1,\ve}\|_{i,\ve,\Omega}^2=2t_\beta^2\beta\mathcal{B}_{\widetilde{U}_{1,\ve},4,\Omega}^{4}+o(1)\ve^4,
\end{eqnarray*}
where $o(1)\to0$ uniformly for $\ve<1$ as $\beta\to+\infty$.  This implies
\begin{eqnarray*}
c'_{\ve,\Omega,T}\leq\mathcal{J}_{\ve,\Omega,T}(t_\beta\overrightarrow{\mathbf{U}^*}_\ve)\leq \ve^4t_\beta^2(\frac14\sum_{i=1}^2\|\overline{U}_{1,\ve}\|_{i,\Omega_\ve}^2+o(1)),
\end{eqnarray*}
where $\overline{U}_{1,\ve}(x)=\widetilde{U}_{1,\ve}(\ve x)$.
Thus, by \eqref{eqn0025} and Propositions~\ref{propA0001}--\ref{propAnew0001}, there exists $\beta_1>\beta_0$ independent of $\ve$ such that
$$
c'_{\ve,\Omega,T}<\min\{d_{1,\ve,\Omega},d_{2,\ve,\Omega}, A_\ve\}-\ve^4C_\beta
$$
for all $\beta>\beta_1$.
\end{proof}

By Lemma~\ref{lem0001} and the construction of $\chi_\beta$, we have $\overrightarrow{\mathbf{u}}_n\rightharpoonup\overrightarrow{\mathbf{u}}_0$ and $\overrightarrow{\mathbf{u}}'_n\rightharpoonup\overrightarrow{\mathbf{u}}'_0$ weakly in $\mathcal{H}_{\ve,\Omega}$ as $n\to\infty$ and $\mathcal{J}_{\ve,\Omega,T}'(\overrightarrow{\mathbf{u}}_0)=\mathcal{J}_{\ve,\Omega,T}'(\overrightarrow{\mathbf{u}}'_0)=0$ in $\mathcal{H}_{\ve,\Omega}^{-1}$.
\begin{proposition}\label{prop0001}
Assume $\alpha_1,\alpha_2>0$. Let $\alpha_T$, $\beta_0$ and $\beta_1$ be respectively given by Lemmas~\ref{lem0002}-\ref{lem0003}.  Then $(\mathcal{S}_{\ve})$ has a nontrivial solution $\overrightarrow{\mathbf{u}}_\ve$ with $\ve>0$ small enough in the following  two cases:
\begin{enumerate}
\item either \; $-\sqrt{\mu_1\mu_2}<\beta<\beta_0$ or $\beta>\beta_1$,
\item $\beta\leq-\sqrt{\mu_1\mu_2}$   with $|\overrightarrow{\mathbf{\alpha}}|<\alpha_T$.
\end{enumerate}
Moreover, $\mathcal{J}_{\ve,\Omega,T}(\overrightarrow{\mathbf{u}}_\ve)=c_{\ve,\Omega,T}$ for $\beta<\beta_0$ and $\mathcal{J}_{\ve,\Omega,T}(\overrightarrow{\mathbf{u}}_\ve)=c'_{\ve,\Omega,T}$ for $\beta>\beta_1$.
\end{proposition}
\begin{proof}
Let $\overrightarrow{\mathbf{v}}_n=\overrightarrow{\mathbf{u}}_n-\overrightarrow{\mathbf{u}}_0$ and $\overrightarrow{\mathbf{v}}'_n=\overrightarrow{\mathbf{u}}'_n-\overrightarrow{\mathbf{u}}'_0$.  Since $\overrightarrow{\mathbf{u}}_n\rightharpoonup\overrightarrow{\mathbf{u}}_0$ and $\overrightarrow{\mathbf{u}}'_n\rightharpoonup\overrightarrow{\mathbf{u}}'_0$ weakly in $\mathcal{H}_{\ve,\Omega}$ as $n\to\infty$, we have $\overrightarrow{\mathbf{v}}_n,\overrightarrow{\mathbf{v}}'_n\rightharpoonup0$ weakly in $\mathcal{H}_{\ve,\Omega}$ as $n\to\infty$.

{\bf Claim.~1}\quad $\overrightarrow{\mathbf{u}}'_0$ is nontrivial for $\beta>\beta_1$.

Indeed, suppose $\overrightarrow{\mathbf{u}}'_0=\overrightarrow{\mathbf{0}}$, then  $\overrightarrow{\mathbf{u}}'_n\rightharpoonup\overrightarrow{\mathbf{0}}$ weakly in $\mathcal{H}_{\ve,\Omega}$ as $n\to\infty$.  It follows from $\overrightarrow{\mathbf{u}}'_n\in\mathcal{N}'_{\ve,\Omega,T}$ and the Sobolev embedding theorem that
\begin{eqnarray}\label{eqn0026}
\sum_{i=1}^2\|(u_{i}^n)'\|_{i,\ve,\Omega}^2=\sum_{i=1}^2\mathcal{B}_{(u_i^n)',4,\Omega}^4+2\beta\mathcal{B}_{|(u_1^n)'|^2|(u_2^n)'|^2,1,\Omega}+o_n(1).
\end{eqnarray}
Let   $(u_i^n)'$ be a function in $D^{1,2}(\bbr^4)$ by setting $(u_i^n)'\equiv0$ outside $\Omega$.  Then by a standard argument and \eqref{eqn0026}, we have  that
\begin{eqnarray*}
c'_{\ve,\Omega,T}&=&\frac14\sum_{i=1}^2\|(u_i^n)'\|_{i,\ve,\Omega}^2+o_n(1)\\
&\geq&\min_{\overrightarrow{\mathbf{u}}\in\widetilde{\mathcal{D}}}
\frac14\bigg(\frac{\sum_{i=1}^2\|u_{i}\|_{i,\ve,\bbr^4}^2}
{(\sum_{i=1}^2\mathcal{B}_{u_i,4,\bbr^4}^4+2\beta\mathcal{B}_{|u_1|^2|u_2|^2,1,\bbr^4})^{\frac12}}\bigg)^2+o_n(1),
\end{eqnarray*}
which together with \cite[(5.46)]{CZ121} and a standard scaling technique, implies $c'_{\ve,\Omega,T}\geq A_\ve$ in the case $\beta>\beta_1$.  This contradicts to Lemma~\ref{lem0003}.  Thus, $\overrightarrow{\mathbf{u}}'_0\not=\overrightarrow{\mathbf{0}}$ for $\beta>\beta_1$.  Thanks to Lemma~\ref{lem0003} once more, $\overrightarrow{\mathbf{u}}'_0$ is  also not   semi-trivial.  Hence, $\overrightarrow{\mathbf{u}}'_0$ must be nontrivial.

{\bf Claim.~2}\quad $\overrightarrow{\mathbf{u}}_0$ is nontrivial for $\beta<\beta_0$.

We first prove that $\overrightarrow{\mathbf{u}}_0\not=\overrightarrow{\mathbf{0}}$ for $\beta<\beta_0$.  Indeed, suppose the contrary, then $\overrightarrow{\mathbf{u}}_n\rightharpoonup\overrightarrow{\mathbf{0}}$ weakly in $\mathcal{H}_{\ve,\Omega}$ as $n\to\infty$.  It follows from $\overrightarrow{\mathbf{u}}_n\in\mathcal{N}_{\ve,\Omega,T}$ and the Sobolev embedding theorem that
\begin{eqnarray}\label{eqn0027}
\|u_{i}^n\|_{i,\ve,\Omega}^2=\mu_i\mathcal{B}_{u_i^n,4,\Omega}^4+\beta\mathcal{B}_{|u_1^n|^2|u_2^n|^2,1,\Omega}+o_n(1),\;\; i=1,2.
\end{eqnarray}
By setting $u_i^n\equiv0$ outside $\Omega$, we also regard $u_i^n$ as  a function in $D^{1,2}(\bbr^4)$.  Due to \eqref{eqn0027}, it is easy to show that there exists $\overrightarrow{\mathbf{t}}_n\to\overrightarrow{\mathbf{1}}$ as $n\to\infty$ such that $\overrightarrow{\mathbf{t}}_n\circ\overrightarrow{\mathbf{u}}_n\in\mathcal{V}_{\ve}$ (cf. \cite{WWZ17}), where $\mathcal{V}_{\ve}$ is given by \eqref{eqn0034}.  It follows from Lemma~\ref{lem0002} and the Sobolev embedding theorem that
\begin{eqnarray*}
c_{\ve,\Omega,T}&=&\mathcal{J}_{\ve,\Omega,T}(\overrightarrow{\mathbf{u}}_n)+o_n(1)\\
&\geq&\mathcal{J}_{\ve,\Omega,T}(\overrightarrow{\mathbf{t}}_n\circ\overrightarrow{\mathbf{u}}_n)+o_n(1)\\
&=&\mathcal{I}_{\ve}(\overrightarrow{\mathbf{t}}_n\circ\overrightarrow{\mathbf{u}}_n)+o_n(1)\\
&\geq&A_\ve+o_n(1)
\end{eqnarray*}
for $\beta<\beta_0$, which contradicts Lemma~\ref{lem0003}.  Here, $\mathcal{I}_{\ve}(\overrightarrow{\mathbf{u}})$ is given by \eqref{eqnew8001}.  We next prove that $\overrightarrow{\mathbf{u}}_0$ is also not semi-trivial for $\beta<\beta_0$.  If not, then without loss of generality, we may assume $\overrightarrow{\mathbf{u}}_0=(u_1^0, 0)$ with $u_1^0\not=0$.  Let $v_n=u_1^n-u_1^0$.  Then $\overrightarrow{\mathbf{v}}_n=(v_n, u_2^n)\rightharpoonup\overrightarrow{\mathbf{0}}$ weakly in $\mathcal{H}_{\ve,\Omega}$ as $n\to\infty$.  By the Brez\'is-Lieb lemma and \cite[Lemma~2.3]{CZ15}, we  obtain that
\begin{eqnarray}\label{eqn0029}
\mathcal{J}_{\ve,\Omega,T}(\overrightarrow{\mathbf{u}}_n)=\mathcal{E}_{1,\ve,\Omega}(u_1^0)+\mathcal{J}_{\ve,\Omega,T}(\overrightarrow{\mathbf{v}}_n)+o_n(1),
\end{eqnarray}
where $\mathcal{E}_{1,\ve,\Omega}(u)$ is given by \eqref{eqn0015}.
Since $\overrightarrow{\mathbf{u}}_n\rightharpoonup\overrightarrow{\mathbf{u}}_0$ weakly in $\mathcal{H}_{\ve,\Omega}$ as $n\to\infty$ and $u_2^0\equiv0$, it is easy to show that $\mathcal{E}'_{1,\ve,\Omega}(u_1^0)=0$ in $\mathcal{H}^{-1}_{1,\ve,\Omega}$, where $\mathcal{H}^{-1}_{1,\ve,\Omega}$ is the dual space of $\mathcal{H}_{1,\ve,\Omega}$.  Now, by similar arguments as used above, we can see from \eqref{eqn0029} that
\begin{eqnarray*}
c_{\ve,\Omega,T}\geq d_{1,\ve,\Omega}+A_\ve,
\end{eqnarray*}
which also contradicts to Lemma~\ref{lem0003}.

\vskip0.16in

Now, since $\mathcal{J}_{\ve,\Omega,T}'(\overrightarrow{\mathbf{u}}_0)=\mathcal{J}_{\ve,\Omega,T}'(\overrightarrow{\mathbf{u}}'_0)=0$ in $\mathcal{H}_{\ve,\Omega}^{-1}$, by the fact that $\overrightarrow{\mathbf{u}}_0$ and $\overrightarrow{\mathbf{u}}'_0$ are both nontrivial, we can see that $c_{\ve,\Omega,T}$ and $c'_{\ve,\Omega,T}$ are both attained by $\overrightarrow{\mathbf{u}}_\ve$ and $\overrightarrow{\mathbf{u}}'_\ve$ respectively for $\beta<\beta_0$ and $\beta>\beta_1$.  Here, $\overrightarrow{\mathbf{u}}_\ve=(|u_1^0|, |u_2^0|)$ and $\overrightarrow{\mathbf{u}}'_\ve=(|(u_1^0)'|, |(u_2^0)'|)$.  Moreover, we also have that $\overrightarrow{\mathbf{u}}_n\to\overrightarrow{\mathbf{u}}_0$ and $\overrightarrow{\mathbf{u}}_n'\to\overrightarrow{\mathbf{u}}_0'$ strongly in $\mathcal{H}_{\ve,\Omega}$ as $n\to\infty$.  Now, by the method of Lagrange multipliers, there exists $\rho_1$, $\rho_2$ and $\rho_3$ such that
\begin{eqnarray}\label{eqn0030}
\mathcal{J}_{\ve,\Omega,T}'(\overrightarrow{\mathbf{u}}_\ve)
-\rho_1(\mathcal{J}_{\ve,\Omega,T}'(\overrightarrow{\mathbf{u}}_\ve)\overrightarrow{\mathbf{u}}_\ve^1)'
-\rho_2(\mathcal{J}_{\ve,\Omega,T}'(\overrightarrow{\mathbf{u}}_\ve)\overrightarrow{\mathbf{u}}_\ve^2)'=0
\end{eqnarray}
and
\begin{eqnarray}\label{eqn0031}
\mathcal{J}_{\ve,\Omega,T}'(\overrightarrow{\mathbf{u}}'_\ve)
-\rho_3(\mathcal{J}_{\ve,\Omega,T}'(\overrightarrow{\mathbf{u}}'_\ve)\overrightarrow{\mathbf{u}}'_\ve)'=0
\end{eqnarray}
in $\mathcal{H}^{-1}_{\ve,\Omega}$.  Multiplying \eqref{eqn0030} and \eqref{eqn0031} with $\overrightarrow{\mathbf{u}}_\ve^i$ and $\overrightarrow{\mathbf{u}}'_\ve$ respectively, we can see that
\begin{eqnarray}\label{eqn9999}
\rho_1(\mathcal{J}_{\ve,\Omega,T}'(\overrightarrow{\mathbf{u}}_\ve)\overrightarrow{\mathbf{u}}_\ve^1)'\overrightarrow{\mathbf{u}}_\ve^i
+\rho_2(\mathcal{J}_{\ve,\Omega,T}'(\overrightarrow{\mathbf{u}}_\ve)\overrightarrow{\mathbf{u}}_\ve^2)'\overrightarrow{\mathbf{u}}_\ve^i=0, i=1,2
\end{eqnarray}
 and
\begin{eqnarray}\label{eqn9998}
\rho_3(\mathcal{J}_{\ve,\Omega,T}'(\overrightarrow{\mathbf{u}}'_\ve)\overrightarrow{\mathbf{u}}'_\ve)'\overrightarrow{\mathbf{u}}'_\ve=0.
\end{eqnarray}
Since $\text{det}(\Theta_n)\geq C$ which is given by \eqref{eqn0016}, by $\overrightarrow{\mathbf{u}}_n\to\overrightarrow{\mathbf{u}}_0$ strongly in $\mathcal{H}_{\ve,\Omega}$ as $n\to\infty$, we can see that $\overrightarrow{\mathbf{0}}$ is the unique solution of the System~\eqref{eqn9999}.  On the other hand, due to $p>2$, it is also easy to see that $0$ is only solution of the Equation~\eqref{eqn9998}.
Thus, by \eqref{eqn0030} and \eqref{eqn0031}, we have $\mathcal{J}_{\ve,\Omega,T}'(\overrightarrow{\mathbf{u}}_\ve)=0$ in the following two cases:
\begin{enumerate}
\item $-\sqrt{\mu_1\mu_2}<\beta<\beta_0$ and $\alpha_1,\alpha_2>0$,
\item $\beta\leq-\sqrt{\mu_1\mu_2}$ and $\alpha_1,\alpha_2>0$ with $|\overrightarrow{\mathbf{\alpha}}|<\alpha_T$,
\end{enumerate}
and $\mathcal{J}_{\ve,\Omega,T}'(\overrightarrow{\mathbf{u}}'_\ve)=0$ for $\beta>\beta_1$ and $\alpha_1,\alpha_2>0$ in $\mathcal{H}^{-1}_{\ve,\Omega}$.  Since Lemma~\ref{lem0001} holds, by standard elliptic estimates and the maximum principle, $\overrightarrow{\mathbf{u}}_\ve$ and $\overrightarrow{\mathbf{u}}_\ve'$ are also solutions of $(\mathcal{S}_\ve)$.
\end{proof}

\vskip0.1in

\noindent\textbf{Proof of Theorem~\ref{thm0001}:}\quad  It follows immediately from \eqref{eqn0008}, Lemma~\ref{lem0001} and Proposition~\ref{prop0001}.
\hfill$\Box$


\vskip0.3in
\section{The limiting problem in $\bbr^4$}
In this section, we mainly consider the following system
\begin{equation*}
\left\{\aligned&-\Delta u_1+\lambda_1u_1=\mu_1u_1^3+\alpha_1u_1^{p-1}+\beta u_2^2u_1\quad&\text{in }\bbr^4,\\
&-\Delta u_2+\lambda_2u_2=\mu_2u_2^3+\alpha_2u_2^{p-1}+\beta u_1^2u_2\quad&\text{in }\bbr^4,\\
&u_1,u_2>0\text{ in }\bbr^4,\quad u_1,u_2\to0\text{ as }|x|\to+\infty.\endaligned\right.\eqno{(\mathcal{S}_{*})}
\end{equation*}

Let $\mathcal{H}_{\bbr^4}=\mathcal{H}_{1,\bbr^4}\times\mathcal{H}_{2,\bbr^4}$, where $\mathcal{H}_{i,\bbr^4}$ are given in the Appendix.  Then $\mathcal{H}_{\bbr^4}$ is a Hilbert space equipped with the inner product
$$
\langle \overrightarrow{\mathbf{u}},\overrightarrow{\mathbf{v}}\rangle_{\bbr^4}=\sum_{i=1}^2\langle u_i,v_i\rangle_{i,\bbr^4}.
$$
The corresponding norm is given by $\|\overrightarrow{\mathbf{u}}\|_{\bbr^4}=\langle \overrightarrow{\mathbf{u}},\overrightarrow{\mathbf{u}}\rangle_{\bbr^4}^{\frac12}$.  Here, $u_i,v_i$ are the $i$th component of $\overrightarrow{\mathbf{u}}, \overrightarrow{\mathbf{v}}$ respectively.  Set
\begin{eqnarray}
\mathcal{J}_{\bbr^4,T}(\overrightarrow{\mathbf{u}})&=&\sum_{i=1}^2(\frac12\|u_i\|_{i,\bbr^4}^2-\frac{\alpha_i}{p}\mathcal{B}_{ u_i,p,\bbr^4}^{p}\chi_\beta\bigg(\frac{\|\overrightarrow{\mathbf{u}}\|_{\bbr^4}}{T}\bigg)-\frac{\mu_i}{4}\mathcal{B}_{ u_i,4,\bbr^4}^{4})\notag\\
&&-\frac{\beta}{2}\mathcal{B}_{|u_1|^{2}|u_2|^{2},1,\bbr^4},\label{eqnew8002}
\end{eqnarray}
where $\chi_\beta(s)$ is given by \eqref{eqnew8000}.
Clearly, $\mathcal{J}_{\bbr^4,T}(\overrightarrow{\mathbf{u}})$ is of $C^2$.  Define
\begin{eqnarray*}
B=\inf_{\mathcal{N}_{\bbr^4,T}}\mathcal{J}_{\bbr^4,T}(\overrightarrow{\mathbf{u}})\quad\text{and}\quad B'=\inf_{\mathcal{N}'_{\bbr^4,T}}\mathcal{J}_{\bbr^4,T}(\overrightarrow{\mathbf{u}}),
\end{eqnarray*}
where
\begin{eqnarray}
\mathcal{N}_{\bbr^4,T}=\{\overrightarrow{\mathbf{u}}=(u_1,u_2)\in\widetilde{\mathcal{H}}_{\bbr^4}\mid \mathcal{J}_{\bbr^4,T}'(\overrightarrow{\mathbf{u}})\overrightarrow{\mathbf{u}}^1=\mathcal{J}_{\bbr^4,T}'(\overrightarrow{\mathbf{u}})\overrightarrow{\mathbf{u}}^2=0\},\label{eqnew8009}
\end{eqnarray}
with $\widetilde{\mathcal{H}}_{\bbr^4}=(\mathcal{H}_{1,\bbr^4}\backslash\{0\})\times(\mathcal{H}_{2,\bbr^4}\backslash\{0\})$, $\overrightarrow{\mathbf{u}}^1=(u_1,0)$ and $\overrightarrow{\mathbf{u}}^2=(0, u_2)$ and
\begin{eqnarray}\label{eqnew8010}
\mathcal{N}'_{\bbr^4,T}=\{\overrightarrow{\mathbf{u}}=(u_1,u_2)\in\mathcal{H}_{\bbr^4}\backslash\{\overrightarrow{\mathbf{0}}\}\mid \mathcal{J}_{\bbr^4,T}'(\overrightarrow{\mathbf{u}})\overrightarrow{\mathbf{u}}=0\}.
\end{eqnarray}
Now, our results for $(\mathcal{S}_*)$ can be stated as follows.

\vskip0.13in

\begin{proposition}\label{prop0002}
Let $\alpha_T$, $\beta_0$ and $\beta_1$ be respectively given by Lemmas~\ref{lem0002}--\ref{lem0003}. Then
\begin{enumerate}
\item[$(i)$]  $B=\sum_{i=1}^2d_{i,\bbr^4}$ and $B$ can not be attained in one of the following two cases
\begin{enumerate}
\item $-\sqrt{\mu_1\mu_2}<\beta<0$;
\item $\beta\leq-\sqrt{\mu_1\mu_2}$ and $|\overrightarrow{\alpha}|<\alpha_T$,
\end{enumerate}
where $d_{i,\bbr^4}$ are given by \eqref{eqn0066}.
\item[$(ii)$]  there exists $\overrightarrow{\mathbf{U}}_*$ with $U_i^*$ radial symmetric such that $\overrightarrow{\mathbf{U}}_*$ is a solution of $(\mathcal{S}_*)$ and $\overrightarrow{\mathbf{U}}_*$ attains $B$ for $0<\beta<\beta_0$ or $\beta>\beta_1$.  Moreover, $B<\sum_{i=1}^2d_{i,\bbr^4}$ for $0<\beta<\beta_0$ and $B=B'$ for $\beta>\beta_1$.
\end{enumerate}
\end{proposition}
\begin{proof}
$(i)$\quad Let $U_{2,\bbr^4}^R(x)=U_{2,\bbr^4}^R(x-Re_1)$, where $e_1=(1,0,0,0)$.  Then it is easy to see that $\mathcal{B}_{|U_{1,\bbr^4}|^{2}|U_{2,\bbr^4}^R|^{2},1,\bbr^4}\to0$ as $R\to+\infty$.  It follows that
$$
\mu_1\mu_2\mathcal{B}_{U_{1,\bbr^4},4,\bbr^4}^{4}\mathcal{B}_{U_{2,\bbr^4}^R,4,\bbr^4}^{4}-\beta^2\mathcal{B}_{|U_{1,\bbr^4}|^{2}|U_{2,\bbr^4}^R|^{2},1,\bbr^4}\geq C
$$
and
$$
\mathcal{B}_{U_{1,\bbr^4},4,\bbr^4}^{4}\geq C,\quad \mathcal{B}_{U_{2,\bbr^4}^R,4,\bbr^4}^{4}\geq C
$$
for $R$ large enough.  Since \eqref{eqn0008} holds and $d_{i,\bbr^4}\in(0, \frac{1}{4\mu_i}\mathcal{S}^2)$, we can apply Miranda's theorem similar to that for $\eta_i(t_1,t_2)$ in the proof of Lemma~\ref{lem0003} to show that there exists $\overrightarrow{\mathbf{t}}_R\in(\bbr^+)^2$ with $\overrightarrow{\mathbf{t}}_R\to\overrightarrow{\mathbf{1}}$ as $R\to+\infty$ such that $\overrightarrow{\mathbf{t}}_R\circ\overrightarrow{\mathbf{U}}_R\in\mathcal{N}_{\bbr^4,T}$ and $\mathcal{J}_{\bbr^4,T}(\overrightarrow{\mathbf{t}}_R\circ\overrightarrow{\mathbf{U}}_R)=\max_{\overrightarrow{\mathbf{t}}\in\bbr_+^2}\mathcal{J}_{\bbr^4,T}(\overrightarrow{\mathbf{t}}\circ\overrightarrow{\mathbf{U}}_R)$ for $R$ large enough, where $\overrightarrow{\mathbf{U}}_R=(U_{1,\bbr^4}, U_{2,\bbr^4}^R)$.  Moreover, taking into account \eqref{eqn0008}, we can use similar arguments as used for $\Phi_n(\overrightarrow{\mathbf{t}})$ in the proof of Lemma~\ref{lem0001} to show that
$$
\mathcal{J}_{\bbr^4,T}(\overrightarrow{\mathbf{u}})=\max_{\overrightarrow{\mathbf{s}}\in\bbr_+^2}\mathcal{J}_{\bbr^4,T}(\overrightarrow{\mathbf{s}}\circ\overrightarrow{\mathbf{u}})
$$
for all $\overrightarrow{\mathbf{u}}\in\mathcal{N}_{\bbr^4,T}$ with $\mathcal{J}_{\bbr^4,T}(\overrightarrow{\mathbf{u}})<\frac{1}{4\mu_1}\mathcal{S}^2+\frac{1}{4\mu_2}\mathcal{S}^2$.  Now, since $d_{i,\bbr^4}$ is attained by $U_{i,\bbr^4}$,  $i=1,2$, by $\beta<0$ and a similar argument as adopted for \cite[Theorem~1.5]{CZ121}, we can show that $B=\sum_{i=1}^2d_{i,\bbr^4}$ and it can not be attained for $\beta<0$.

\vskip0.2in
$(ii)$\quad By setting $u\equiv0$ outside $\Omega$, we can regard any $u\in\h$ as a function in $H^1(\bbr^4)$.  Now, applying similar argument for $\overrightarrow{\widetilde{\mathbf{U}}}_\ve=(\widetilde{U}_{1,\ve}, u_\sigma)$ in the proof of Lemma~\ref{lem0003} to $(U_{1,\bbr^4}, v_\sigma)$, we have  that
$B<d_{1,\bbr^4}+\frac{1}{4\mu_2}\mathcal{S}^2$ for all $0<\beta<\beta_0$.  Similarly, we can also show that $B<d_{2,\bbr^4}+\frac{1}{4\mu_1}\mathcal{S}^2$ for all $0<\beta<\beta_0$.  By following the similar argument as used for $\overrightarrow{\overline{\mathbf{U}}}_{\sigma}=(\sqrt{k}_1u_{\sigma}, \sqrt{k}_2u_{\sigma})$ in the proof of  Lemma~\ref{lem0003}, we can prove  that $B<A_1$  for all $0<\beta<\beta_0$.  Thus, we actually have
\begin{eqnarray}\label{eqn0040}
B<\min\{d_{1,\bbr^4}+\frac{1}{4\mu_2}\mathcal{S}^2, d_{2,\bbr^4}+\frac{1}{4\mu_1}\mathcal{S}^2, A_1\}
\end{eqnarray}
 for all $0<\beta<\beta_0$.  Moreover, applying  the argument as used for $\overrightarrow{\mathbf{U}^*}_\ve=(\widetilde{U}_{1,\ve}, \widetilde{U}_{1,\ve})$ in the proof of Lemma~\ref{lem0003}  to $(U_{1,\bbr^4}, U_{1,\bbr^4})$, we can show that
\begin{eqnarray}\label{eqn0041}
B'<\min\{d_{1,\bbr^4},d_{2,\bbr^4}, A_1\}
\end{eqnarray}
for all $\beta>\beta_1$.  On the other hand, let $\{\overrightarrow{\mathbf{u}}_n\}\subset\mathcal{N}_{\bbr^4,T}$ and $\{\overrightarrow{\mathbf{u}}'_n\}\subset\mathcal{N}'_{\bbr^4,T}$ be the sequences obtained by the Ekeland's principle, then by $\beta>0$ and the Schwartz's symmetrization (cf. \cite{LW05}), we can see that $\overrightarrow{\mathbf{u}}_n$ and $\overrightarrow{\mathbf{u}}'_n$ can be chosen to be radial symmetric.  Now, due to the compactness of the embedding map $\mathcal{H}_{i,\bbr^4}\to L^r_{Rad}(\bbr^4)$ for all $2<r<4$, by \eqref{eqn0040} and \eqref{eqn0041}, we can follow the similar arguments as used in the proof of Proposition~\ref{prop0001} to prove that there exists $\overrightarrow{\mathbf{U}}_*$ with $U_i^*$ radial symmetric such that $\overrightarrow{\mathbf{U}}_*$ is a solution of $(\mathcal{S}_*)$ and $\overrightarrow{\mathbf{U}}_*$ attains $B$ for $0<\beta<\beta_0$ or $\beta>\beta_1$ with $B=B'$ for $\beta>\beta_1$, where $L^r_{Rad}(\bbr^4)=\{u\in L^r(\bbr^4)\mid u\text{ is radial symmetric}\}$.  It remains to show that $B<\sum_{i=1}^2d_{i,\bbr^4}$ for $0<\beta<\beta_0$.  Indeed, by considering $\overrightarrow{\mathbf{U}}_{\bbr^4}=(U_{1,\bbr^4}, U_{2,\bbr^4})$ and taking into account \eqref{eqn0008}, we can apply Miranda's theorem similar to that for $\eta_i(t_1,t_2)$ in the proof of Lemma~\ref{lem0003} to show that $\overrightarrow{\mathbf{t}}\circ\overrightarrow{\mathbf{U}}_{\bbr^4}\in\mathcal{N}_{\bbr^4,T}$ for $\beta>0$ small enough with some $\overrightarrow{\mathbf{t}}\in(\bbr^+)^2$.  By taking $\beta_0$ small enough if necessary, we can see that $\overrightarrow{\mathbf{t}}\circ\overrightarrow{\mathbf{U}}_{\bbr^4}\in\mathcal{N}_{\bbr^4,T}$ for $0<\beta<\beta_0$ with some $\overrightarrow{\mathbf{t}}\in(\bbr^+)^2$.  Since $p>2$, by a standard argument, we can see from $\beta>0$ that
\begin{eqnarray*}
\sum_{i=1}^2d_{i,\bbr^4}=\sum_{i=1}^2\mathcal{E}_{i,\bbr^4}(U_{i,\bbr^4})\geq\sum_{i=1}^2\mathcal{E}_{i,\bbr^4}(t_iU_{i,\bbr^4})>\mathcal{J}_{\bbr^4,T}(\overrightarrow{\mathbf{t}}\circ\overrightarrow{\mathbf{U}}_{\bbr^4})\geq B,
\end{eqnarray*}
where $\mathcal{E}_{i,\bbr^4}(u)$ are given by \eqref{eqn0095}.
\end{proof}

\vskip0.3in
Next, we also establish a  result for the following system
\begin{equation*}
\left\{\aligned&-\Delta u_1+\lambda_1u_1=\mu_1u_1^3+\alpha_1u_1^{p-1}+\beta u_2^2u_1\quad&\text{in }\bbr^4_{+},\\
&-\Delta u_2+\lambda_2u_2=\mu_2u_2^3+\alpha_2u_2^{p-1}+\beta u_1^2u_2\quad&\text{in }\bbr^4_{+},\\
&u_1,u_2\geq0\text{ in }\bbr^4_{+},\quad u_1,u_2\to0\text{ as }|x|\to+\infty,\\
&u_1=u_2=0\text{ on }\partial\bbr^4_{+},\endaligned\right.\eqno{(\mathcal{S}_{**})}
\end{equation*}
where $\bbr^4_{+}=\{x=(x_1,x_2,x_3,x_4)\in\bbr^4\mid x_4>0\}$.  Our result can be read as follows.
\begin{proposition}\label{prop0003}
$(\mathcal{S}_{**})$ has no solution unless $u_1\equiv0$ and $u_2\equiv0$.
\end{proposition}
\begin{proof}
Suppose the contrary and let $(u_1, u_2)$ be a solution of $(\mathcal{S}_{**})$ such that $\sum_{i=1}^2|u_i|>0$.  Then by the classical elliptic regularity theorem, we can see that $u_i$ is of $C^2$.  Thus, by the strong maximum principle, we have $\frac{\partial u_i}{\partial x_4}>0$ on $\partial\bbr^4_{+}$ for all $i=1,2$.  Now, multiplying $(\mathcal{S}_{**})$ with $(\frac{\partial u_1}{\partial x_4}, \frac{\partial u_2}{\partial x_4})$ and integrating by parts, we have
\begin{eqnarray*}
0=\int_{\bbr^4_{+}}\sum_{i=1}^2\Delta u_i\frac{\partial u_i}{\partial x_4}dx=\frac12\int_{\partial\bbr^4_{+}}\sum_{i=1}^2|\frac{\partial u_1}{\partial x_4}|^2ds>0,
\end{eqnarray*}
which is a contradiction.
\end{proof}

\section{The concentration behavior as $\ve\to0^+$}
By Proposition~\ref{prop0001}, $(\mathcal{S}_{\ve})$ has a nontrivial solution $\overrightarrow{\mathbf{u}}_\ve$ with $\ve$ small enough and $\alpha_1,\alpha_2>0$ in the following two cases:
\begin{enumerate}
\item either \; $-\sqrt{\mu_1\mu_2}<\beta<\beta_0$ or $\beta>\beta_1$,
\item $\beta\leq-\sqrt{\mu_1\mu_2}$ with $|\overrightarrow{\mathbf{\alpha}}|<\alpha_T$.
\end{enumerate}
In this section, we will obtain some results to $\overrightarrow{\mathbf{u}}_\ve$ as $\ve\to0^+$.  We first give an estimate for $\ve^{-4}c_{\ve,\Omega,T}$ as $\ve\to0^+$.
\begin{lemma}\label{lem0005}
Let $\alpha_T$, $\beta_0$ and $\beta_1$ be respectively given by Lemmas~\ref{lem0002}--\ref{lem0003}.  Then we have the following.
\begin{enumerate}
\item[$(i)$]  $\ve^{-4}c_{\ve,\Omega,T}=B+o(1)$  in the following two cases:
\begin{enumerate}
\item $-\sqrt{\mu_1\mu_2}<\beta<\beta_0$ and $\alpha_1,\alpha_2>0$,
\item $\beta\leq-\sqrt{\mu_1\mu_2}$ and $\alpha_1,\alpha_2>0$ with $|\overrightarrow{\mathbf{\alpha}}|<\alpha_T$.
\end{enumerate}
\item[$(ii)$]  $\ve^{-4}c'_{\ve,\Omega,T}=B'+o(1)$ for $\beta>\beta_1$ and $\alpha_1,\alpha_2>0$.
\end{enumerate}
\end{lemma}
\begin{proof}
$(1)$\quad Let $\{\overrightarrow{\mathbf{U}}_n\}\in\mathcal{N}_{\bbr^4,T}$ and $\mathcal{J}_{\bbr^4,T}(\overrightarrow{\mathbf{U}}_n)<B+\frac1n$, where $\mathcal{J}_{\bbr^4,T}(\overrightarrow{\mathbf{U}})$ is given by \eqref{eqnew8002}.
Since $\Omega_{\ve}\to\bbr^4$ as $\ve\to0^+$, where $\Omega_{\ve}=\{y\in\bbr^4\mid\ve y\in\Omega\}$, there exists $\{\overrightarrow{\mathbf{U}}_{n,\ve}\}\subset\mathcal{H}_{\Omega_\ve}$ such that $\overrightarrow{\mathbf{U}}_{n,\ve}\to\overrightarrow{\mathbf{U}}_{n}$ strongly in $\mathcal{H}_{\bbr^4}$ as $\ve\to0^+$.  It follows from \eqref{eqn0008}, the construction of $\chi_\beta$ given by \eqref{eqnew8000} and Propositions~\ref{prop0002} and \ref{propA0001} that
\begin{eqnarray*}
\|U_{1}^{n,\ve}\|_{1,\ve,\Omega}^2-\alpha_1\mathcal{B}_{U_{1}^{n,\ve},p,\Omega}^{p}
-\mu_1\mathcal{B}_{U_{1}^{n,\ve},4,\Omega}^{4}-\beta\mathcal{B}_{|U_{1}^{n,\ve}|^2|U_{2}^{n,\ve}|^2,1,\Omega}&=o(1),\\
\|U_{2}^{n,\ve}\|_{1,\ve,\Omega}^2-\alpha_2\mathcal{B}_{U_{2}^{n,\ve},p,\Omega}^{p}
-\mu_2\mathcal{B}_{U_{2}^{n,\ve},4,\Omega}^{4}-\beta\mathcal{B}_{|U_{1}^{n,\ve}|^2|U_{2}^{n,\ve}|^2,1,\Omega}&=o(1).
\end{eqnarray*}
Since $\{\overrightarrow{\mathbf{U}}_{n}\}$ is bounded in $\mathcal{H}$, we can apply Miranda's theorem similar to that for $\eta_i(t_1,t_2)$ in the proof of Lemma~\ref{lem0003} to show that there exists $\overrightarrow{\mathbf{t}}_{n,\ve}$ with $\overrightarrow{\mathbf{t}}_{n,\ve}\to\overrightarrow{\mathbf{1}}$ as $\ve\to0^+$ such that $\overrightarrow{\mathbf{t}}_{n,\ve}\circ\overrightarrow{\mathbf{U}}_{n,\ve}\in\mathcal{N}_{\ve,\Omega,T}$ for $\beta>0$ small enough.  By taking $\beta_0$ small enough if necessary, we can see that $\overrightarrow{\mathbf{t}}_{n,\ve}\circ\overrightarrow{\mathbf{U}}_{n,\ve}\in\mathcal{N}_{\ve,\Omega,T}$ for $0<\beta<\beta_0$.  Taking into account \eqref{eqn0008}, we can use the implicit function theorem similar to that for $\Gamma_i^n(\overrightarrow{\mathbf{t}},\tau)$ in the proof of Lemma~\ref{lem0001} to show that there exists $\overrightarrow{\mathbf{t}}_{n,\ve}$ with $\overrightarrow{\mathbf{t}}_{n,\ve}\to\overrightarrow{\mathbf{1}}$ as $\ve\to0^+$ such that $\overrightarrow{\mathbf{t}}_{n,\ve}\circ\overrightarrow{\mathbf{U}}_{n,\ve}\in\mathcal{N}_{\ve,\Omega,T}$ for $\beta<0$.
Since $\overrightarrow{\mathbf{U}}_n\in\mathcal{N}_{\bbr^4,T}$ with $\mathcal{J}_{\bbr^4,T}(\overrightarrow{\mathbf{U}}_n)<B+\frac1n$, we can use similar arguments as used for $\Phi_n(\overrightarrow{\mathbf{t}})$ in the proof of Lemma~\ref{lem0001} further to show that
\begin{eqnarray*}
\mathcal{J}_{\bbr^4,T}(\overrightarrow{\mathbf{U}}_n)\geq\mathcal{J}_{\bbr^4,T}(\overrightarrow{\mathbf{t}}\circ\overrightarrow{\mathbf{U}}_n)\quad\text{for all }\overrightarrow{\mathbf{t}}\in(\bbr_+)^2.
\end{eqnarray*}
Thus, by a standard argument, we can see that $\lim_{\ve\to0^+}\ve^{-4}c_{\ve,\Omega,T}\leq B+\frac1n$ under the conditions $(a1)$ or $(a2)$.
Let $n\to\infty$, we have $\lim_{\ve\to0^+}\ve^{-4}c_{\ve,\Omega,T}\leq B$.
 On the other hand, by setting $u\equiv0$ outside $\Omega_\ve$, we can regard $\overrightarrow{\mathbf{u}}\in\mathcal{H}_{\Omega_\ve}$ as in $\mathcal{H}_{\bbr^4}$.  It follows that $\mathcal{N}_{\Omega_\ve,T}\subset\mathcal{N}_{\bbr^4,T}$, which together with the standard scaling technique, implies $\lim_{\ve\to0^+}\ve^{-4}c_{\ve,\Omega,T}\geq B$.

\vskip0.12in
$(ii)$ The proof is similar but more simple  than that of  $(i)$.
\end{proof}

\vskip0.32in

Let $p_i^\ve$ be the maximum point of $u_i^\ve $($i=1,2$) and define
\begin{eqnarray*}
\Omega_{i,\ve}=\{x\in\bbr^4\mid\ve x+p_i^\ve\in\Omega\}.
\end{eqnarray*}
\begin{lemma}\label{lem0004}
Assume $\alpha_1,\alpha_2>0$.  Let $\alpha_T$, $\beta_0$ and $\beta_1$ be respectively given by Lemmas~\ref{lem0002}-\ref{lem0003}.  Then we have $\Omega_{i,\ve}\to\bbr^4$ as $\ve\to0^+$($i=1,2$) under one of the  following two cases:
\begin{enumerate}
\item either \; $-\sqrt{\mu_1\mu_2}<\beta<\beta_0$ or $\beta>\beta_1$,
\item $\beta\leq-\sqrt{\mu_1\mu_2}$ and $|\overrightarrow{\mathbf{\alpha}}|<\alpha_T$.
\end{enumerate}
\end{lemma}
\begin{proof}
We only give the proof of $\Omega_{1,\ve}$ since that of $\Omega_{2,\ve}$ is similar.  Suppose the contrary, then as in \cite{LW05}, we may assume $\Omega_{1,\ve}\to\bbr^4_+$ as $\ve\to0^+$ without loss of generality, where $\bbr^4_{+}=\{x=(x_1,x_2,x_3,x_4)\in\bbr^4\mid x_4>0\}$.  On the other hand, since $\overrightarrow{\mathbf{u}}_\ve$ is a solution of $(\mathcal{S}_{\ve})$, by a standard scaling technique, we can see that $\overrightarrow{\mathbf{v}}_\ve$ satisfies
\begin{equation*}
\left\{\aligned&-\Delta v_1^\ve+\lambda_1v_1^\ve=\mu_1(v_1^\ve)^3+\alpha_1(v_1^\ve)^{p-1}+\beta (v_2^\ve)^2v_1^\ve\quad&\text{in }\Omega_{1,\ve},\\
&-\Delta v_2^\ve+\lambda_2v_2^\ve=\mu_2(v_2^\ve)^3+\alpha_2(v_2^\ve)^{p-1}+\beta (v_1^\ve)^2v_2^\ve\quad&\text{in }\Omega_{1,\ve},\\
&v_1^\ve,v_2^\ve>0\quad\text{in }\Omega_{1,\ve},\quad v_1^\ve=v_2^\ve=0\quad\text{on }\partial\Omega_{1,\ve},\endaligned\right.\eqno{(\mathcal{S}_{\ve}^*)}
\end{equation*}
where
$$
v_i^\ve(x)=u_i^\ve(p_1^\ve+\ve x),\;\; i=1,2.
$$
Moreover, since $|\overrightarrow{\mathbf{\alpha}}|<\alpha_T$ for $\beta<-\sqrt{\mu_1\mu_2}$, by Lemma~\ref{lem0003}, Proposition~\ref{prop0001} and Proposition~\ref{propA0001}, we can see from a standard scaling technique and a similar argument as used in the proof of Lemma~\ref{lemn0001} that $\{v_i^\ve\}$ are bounded in $\mathcal{H}_{i,\bbr^4}$.  Here, we regard $v_i^\ve$ as  a function in $\mathcal{H}_{i,\bbr^4}$ by setting $v_i^\ve\equiv0$ outside $\Omega_{1,\ve}$.  Without loss of generality, we may assume that $\overrightarrow{\mathbf{v}}_\ve\rightharpoonup\overrightarrow{\mathbf{v}}_0$ weakly in $\mathcal{H}_{\bbr^4}$ as $\ve\to0^+$.  Since $\overrightarrow{\mathbf{v}}_\ve$ satisfies $(\mathcal{S}_{\ve}^*)$, it is easy to see that $\overrightarrow{\mathbf{v}}_0$ is a solution of $(\mathcal{S}_{**})$ due to $\Omega_{1,\ve}\to\bbr^4_+$ as $\ve\to0^+$.  By Proposition~\ref{prop0003}, we must have $\overrightarrow{\mathbf{v}}_0=\overrightarrow{\mathbf{0}}$.  Thus, $\overrightarrow{\mathbf{v}}_\ve\rightharpoonup\overrightarrow{\mathbf{0}}$ weakly in $\mathcal{H}_{\bbr^4}$ as $\ve\to0^+$.  If
\begin{eqnarray*}
\lim_{\ve\to0^+}\sup_{y\in\bbr^4}\int_{\mathbb{B}_\rho(y)}\sum_{i=1}^2(v_i^\ve)^2dx=0
\end{eqnarray*}
for all $\rho>0$, then by the Lions' lemma, we have $\overrightarrow{\mathbf{v}}_\ve\to\overrightarrow{\mathbf{0}}$ strongly in $\mathcal{L}^r(\bbr^4)$ for all $2< r<4$ as $\ve\to0^+$, where $\mathcal{L}^r(\bbr^4)=(L^r(\bbr^4))^2$.  It follows from the fact that $\overrightarrow{\mathbf{v}}_\ve$ satisfies $(\mathcal{S}_{\ve}^*)$ that
\begin{eqnarray}
&\mathcal{B}_{\nabla v_1^\ve,2,\bbr^4}^{2}+\lambda_1\mathcal{B}_{v_1^\ve,2,\bbr^4}^{2}-\mu_1\mathcal{B}_{v_1^\ve,4,\bbr^4}^{4}-\beta\mathcal{B}_{|v_1^\ve|^{2}|v_2^\ve|^{2},1,\bbr^4}=o(1),\label{eqn0032}\\
&\mathcal{B}_{\nabla v_2^\ve,2,\bbr^4}^{2}+\lambda_2\mathcal{B}_{v_2^\ve,2,\bbr^4}^{2}-\mu_2\mathcal{B}_{v_2^\ve,4,\bbr^4}^{4}-\beta\mathcal{B}_{|v_1^\ve|^{2}|v_2^\ve|^{2},1,\bbr^4}=o(1).\label{eqn0033}
\end{eqnarray}
Since $\lambda_1,\lambda_2>0$, by similar arguments as used in the proofs of \cite[Lemma~3.1]{WWZ17} and \cite[Lemma~5.1]{HWW15}, we can see from \eqref{eqn0032} and \eqref{eqn0033} that $\lim_{\ve\to0^+}(c_{\ve,\Omega,T}-\ve^{-4}A_\ve)\geq0$ in the following two cases:
\begin{enumerate}
\item $-\sqrt{\mu_1\mu_2}<\beta<\beta_0$ and $\alpha_1,\alpha_2>0$,
\item $\beta\leq-\sqrt{\mu_1\mu_2}$ and $\alpha_1,\alpha_2>0$ with $|\overrightarrow{\mathbf{\alpha}}|<\alpha_T$.
\end{enumerate}
In the case $\beta>\beta_1$, since $p>2$,   we also have from a standard argument that $\lim_{\ve\to0^+}(c'_{\ve,\Omega,T}-\ve^{-4}A_\ve)\geq0$.  But it is impossible since Lemma~\ref{lem0003} holds.  Therefore, there exist $\rho>0$ and $y^\ve\in\bbr^4$ such that
\begin{eqnarray}\label{eqnew8003}
\int_{\mathbb{B}_\rho(y^\ve)}\sum_{i=1}^2(v_i^\ve)^2dx\geq C>0.
\end{eqnarray}
Without loss of generality, we may assume that $\int_{\mathbb{B}_\rho(y^\ve)}(v_1^\ve)^2dx\geq C>0$.  Since $\overrightarrow{\mathbf{v}}_\ve\rightharpoonup\overrightarrow{\mathbf{0}}$ weakly in $\mathcal{H}_{\bbr^4}$ as $\ve\to0^+$, by the Sobolev embedding theorem, we must have $|y^\ve|\to+\infty$.  Let $v_{i,1}^\ve(x)=v_i^\ve(x+y^\ve)$.  Then  $\overrightarrow{\widetilde{\mathbf{v}}}_\ve=(v_{1,1}^\ve, v_{2,1}^\ve)\rightharpoonup \overrightarrow{\widetilde{\mathbf{v}}}_0=(v_{1,1}^0, v_{2,1}^0)$ weakly in $\mathcal{H}_{\bbr^4}$ as $\ve\to0^+$ with $v_{i,1}^0\geq0$ and $v_{1,1}^0\not\equiv0$.  Let $\Omega_{1,\ve}^*=\Omega_{1,\ve}-y^\ve$ and assume $\Omega_{1,\ve}^*\to \Omega^*$ as $\ve\to0^+$, where $\Omega^*$ is either the whole space $\bbr^4$ or the half space $\bbr^4_+$ under rotations and translations.  Since $\overrightarrow{\mathbf{v}}_\ve$ satisfies $(\mathcal{S}_{\ve}^*)$, it is easy to see that $\overrightarrow{\widetilde{\mathbf{v}}}_0$ is a solution of the following system
\begin{equation*}
\left\{\aligned&-\Delta u_1+\lambda_1u_1=\mu_1u_1^3+\alpha_1u_1^{p-1}+\beta u_2^2u_1\quad&\text{in }\Omega^*,\\
&-\Delta u_2+\lambda_2u_2=\mu_2u_2^3+\alpha_2u_2^{p-1}+\beta u_1^2u_2\quad&\text{in }\Omega^*,\\
&u_1,u_2\geq0\text{ in }\Omega^*,\quad u_1,u_2\to0\text{ as }|x|\to+\infty,\\
&u_1=u_2=0\text{ on }\partial\Omega^*,\endaligned\right.\eqno{(\overline{\mathcal{S}}_{**})}
\end{equation*}
Since $v_{1,1}^0\not\equiv0$, by Proposition~\ref{prop0003}, we must have that $\Omega^*=\bbr^4$.

{\bf Case.~1} $\beta<0$

If $v_{2,1}^0\not=0$, then by Proposition~\ref{prop0002}, we have
\begin{eqnarray*}
\mathcal{J}_{\bbr^4,T}(\overrightarrow{\widetilde{\mathbf{v}}}_0)>B=\sum_{i=1}^2d_{i,\bbr^4}.
\end{eqnarray*}
Let $\overrightarrow{\mathbf{w}}_\ve=\overrightarrow{\widetilde{\mathbf{v}}}_\ve-\overrightarrow{\widetilde{\mathbf{v}}}_0$. Then $\overrightarrow{\mathbf{w}}_\ve\rightharpoonup\overrightarrow{\mathbf{0}}$ weakly in $\mathcal{H}_{\bbr^4}$ as $\ve\to0^+$.  By applying the Brez\'is-Lieb lemma and \cite[Lemma~2.3]{CZ15}, we also have that
\begin{eqnarray}
\ve^{-4}\mathcal{J}_{\ve,\Omega,T}(\overrightarrow{\mathbf{u}}_\ve)=\mathcal{J}_{\Omega_{1,\ve}^*,T}(\overrightarrow{\widetilde{\mathbf{v}}}_\ve)=\mathcal{J}_{\bbr^4,T}(\overrightarrow{\widetilde{\mathbf{v}}}_0)
+\mathcal{J}_{\bbr^4,T}(\overrightarrow{\mathbf{w}}_\ve)+o(1)\label{eqnew8004}
\end{eqnarray}
and
\begin{eqnarray}
\mathcal{J}_{\bbr^4,T}'(\overrightarrow{\mathbf{w}}_\ve)\overrightarrow{\mathbf{w}}_{\ve}^1=\mathcal{J}_{\bbr^4,T}'(\overrightarrow{\mathbf{w}}_\ve)\overrightarrow{\mathbf{w}}_{\ve}^1=o(1),\label{eqnew8005}
\end{eqnarray}
where $\overrightarrow{\mathbf{w}}_{\ve}^1=(w_{1}^\ve, 0)$ and $\overrightarrow{\mathbf{w}}_{\ve}^2=(0, w_{2}^\ve)$.
Since $|\overrightarrow{\mathbf{\alpha}}|<\alpha_T$ for $\beta<-\sqrt{\mu_1\mu_2}$, by a similar argument as used for \eqref{eqn0001} and \eqref{eqn0004} in the proof of Lemma~\ref{lem0002} respectively for $-\sqrt{\mu_1\mu_2}<\beta<0$ and $\beta\leq-\sqrt{\mu_1\mu_2}$, we can see that $\mathcal{J}_{\bbr^4,T}(\overrightarrow{\mathbf{w}}_\ve)\geq0$, which contradicts to Lemma~\ref{lem0005}.  Thus, we must have $v_{2,1}^0=0$.  Then similar to \eqref{eqnew8004} and \eqref{eqnew8005}, we have
\begin{eqnarray}
\ve^{-4}\mathcal{J}_{\ve,\Omega,T}(\overrightarrow{\mathbf{u}}_\ve)=\mathcal{J}_{\Omega_{1,\ve}^*,T}(\overrightarrow{\widetilde{\mathbf{v}}}_\ve)=\mathcal{E}_{1,\bbr^4}(v_{1,1}^0)
+\mathcal{J}_{\bbr^4,T}(\overrightarrow{\mathbf{w}}_\ve)+o(1)\label{eqnew8007}
\end{eqnarray}
and
\begin{eqnarray}
\mathcal{J}_{\bbr^4,T}'(\overrightarrow{\mathbf{w}}_\ve)\overrightarrow{\mathbf{w}}_{\ve}^1=\mathcal{J}_{\bbr^4,T}'(\overrightarrow{\mathbf{w}}_\ve)\overrightarrow{\mathbf{w}}_{\ve}^1=o(1),\label{eqnew8008}
\end{eqnarray}
where $\mathcal{E}_{1,\bbr^4}(u)$ is given by \eqref{eqn0095}.  Clearly, $\mathcal{E}_{1,\bbr^4}'(v_{1,1}^0)=0$.  Moreover, in this situation, we also have $w_{2}^\ve=v_{2,1}^\ve$.  Thus, by a similar argument as use for \eqref{eqn0007}, we can see that $\mathcal{B}_{\nabla w_{2}^\ve,2,\bbr^4}^{2}+\lambda_2\mathcal{B}_{w_{2}^\ve,2,\bbr^4}^{2}\geq C+o(1)$.  It follows from $\beta<0$ and \eqref{eqnew8008} that there exists $0<t_\ve\leq1+o(1)$ such that $t_\ve w_{2}^\ve\in\mathcal{M}_{2,\bbr^4}$, which is given by \eqref{eqnew8006}.  Recall that $|\overrightarrow{\mathbf{\alpha}}|<\alpha_T$ for $\beta<-\sqrt{\mu_1\mu_2}$, thus by a similar argument as used for \eqref{eqn0001} and \eqref{eqn0004} in the proof of Lemma~\ref{lem0002} respectively for $-\sqrt{\mu_1\mu_2}<\beta<0$ and $\beta\leq-\sqrt{\mu_1\mu_2}$, we can see from \eqref{eqn0008} and \eqref{eqnew8007} and Lemma~\ref{lem0005} and Proposition~\ref{prop0002} that $w_1^\ve\to0$ strongly in $\mathcal{H}_{1,\bbr^4}$ as $\ve\to0^+$, which implies $v_{1,1}^\ve\to v_{1,1}^0$ strongly in $\mathcal{H}_{1,\bbr^4}$ as $\ve\to0^+$.  Note that $\overrightarrow{\mathbf{v}}_\ve$ satisfies the system~$(\mathcal{S}_*)$, by $\beta<0$, we can apply the Moser's iteration as in\cite{BZZ13} to show that $v_{1,1}^\ve$ is uniformly bounded in $L^q(\bbr^4)$ for all $q\geq2$.  Since $\overrightarrow{\widetilde{\mathbf{v}}}_0$ is a nontrivial solution of $(\mathcal{S}_{*})$, by the classical elliptic regularity, we can see that $v_{1,1}^0\in L^\infty(\bbr^4)$.  It follows from the Taylor expansion that $w^\ve_1=v_{1,1}^\ve-v_{1,1}^0$ satisfies the following equation
\begin{eqnarray*}
-\Delta w^\ve_1+\lambda_1 w^\ve_1\leq\alpha_1(p-1)(v_{1,1}^\ve+v_{1,1}^0)^{p-2}w^\ve_1+3\mu_1(v_{1,1}^\ve+v_{1,1}^0)^{2}w^\ve_1
\end{eqnarray*}
in $\bbr^4$.  Since $w^\ve_1\to0$ strongly in $\mathcal{H}_{1,\bbr^4}$, by applying the Moser's iteration as in \cite{BZZ13} once more, we can obtain that $w^\ve_1\to0$ strongly in $L^q(\bbr^4)$ for all $q\geq2$.  It follows from \cite[Theorem~8.17]{GT98} (see also \cite[Lemma~4.3]{BZZ13}) that $v_{1,1}^\ve\to0$ as $|x|\to+\infty$ uniformly for $\ve$, which contradicts to the fact that $v_{1,1}^\ve(-y_\ve)=v_1^\ve(0)$ is the maximum value and $|y^\ve|\to+\infty$.  Thus, we must have that $\Omega_{1,\ve}\to\bbr^4$ as $\ve\to0^+$ for $\beta<0$.

{\bf Case.~2} $\beta>0$

If $v_{2,1}^0=0$, then similar to the case $\beta<0$, we also have that \eqref{eqnew8007}--\eqref{eqnew8008} hold and $\mathcal{B}_{\nabla w_{2}^\ve,2,\bbr^4}^{2}+\lambda_2\mathcal{B}_{w_{2}^\ve,2,\bbr^4}^{2}\geq C+o(1)$.  If we also have $\mathcal{B}_{\nabla w_{1}^\ve,2,\bbr^4}^{2}+\lambda_1\mathcal{B}_{w_{1}^\ve,2,\bbr^4}^{2}\geq C+o(1)$ in this situation, then by \eqref{eqnew8008}, we can apply Miranda's theorem similar to that for $\eta_i(t_1,t_2)$ in the proof of Lemma~\ref{lem0003} to show that there exists $\overrightarrow{\mathbf{t}}_{\ve}\in(\bbr^+)^2$ with $\overrightarrow{\mathbf{t}}_{\ve}\to\overrightarrow{\mathbf{1}}$ as $\ve\to0$ such that $\overrightarrow{\mathbf{t}}_{\ve}\circ\overrightarrow{\mathbf{w}}_{\ve}\in\mathcal{N}_{\bbr^4,T}$ for $\beta>0$ small enough.  Here, $\mathcal{N}_{\bbr^4,T}$ is given by \eqref{eqnew8009}.  By taking $\beta_0$ small enough if necessary, we have that $\overrightarrow{\mathbf{t}}_{\ve}\circ\overrightarrow{\mathbf{w}}_{\ve}\in\mathcal{N}_{\bbr^4,T}$ for $0<\beta<\beta_0$.  It leads that \eqref{eqnew8007} contradicts to Lemma~\ref{lem0005} since $\{\overrightarrow{\mathbf{w}}_{\ve}\}$ is bounded in $\mathcal{H}_{\bbr^4}$ and $v_{1,1}^0\not=0$ satisfies $\mathcal{E}_{1,\bbr^4}'(v_{1,1}^0)=0$.  Thus, we must have $w_{1}^\ve\to0$ strongly in $\mathcal{H}_{1,\bbr^4}$ as $\ve\to0^+$.  It follows from \eqref{eqnew8008}, the H\"older inequality and $\mathcal{B}_{\nabla w_{2}^\ve,2,\bbr^4}^{2}+\lambda_2\mathcal{B}_{w_{2}^\ve,2,\bbr^4}^{2}\geq C+o(1)$ that there exists $t_\ve\to1$ as $\ve\to0$ such that $t_\ve w_2^\ve\in\mathcal{M}_{2,\bbr^4}$.  Thus, by \eqref{eqnew8007}, we have
$$
\ve^{-4}\mathcal{J}_{\ve,\Omega,T}(\overrightarrow{\mathbf{u}}_\ve)\geq\sum_{i=1}^2d_{i,\bbr^4}+o(1),
$$
which contradicts to Proposition~\ref{prop0002} and Lemma~\ref{lem0005} for $0<\beta<\beta_0$.  For $\beta>\beta_1$, we have from $\mathcal{B}_{\nabla w_{2}^\ve,2,\bbr^4}^{2}+\lambda_2\mathcal{B}_{w_{2}^\ve,2,\bbr^4}^{2}\geq C+o(1)$ that there exists $s_\ve>0$ such that $s_\ve\overrightarrow{\mathbf{w}}_{\ve}\in\mathcal{N}_{\bbr^4,T}'$, which is given by \eqref{eqnew8010}.  It leads that \eqref{eqnew8007} contradicts to Lemma~\ref{lem0005} since $v_{1,1}^0\not=0$ satisfies $\mathcal{E}_{1,\bbr^4}'(v_{1,1}^0)=0$.  Thus, we must have that $v_{2,1}^0\not=0$.  It follows from Proposition~\ref{prop0002} and Lemma~\ref{lem0005} that $\overrightarrow{\widetilde{\mathbf{v}}}_0$ is a ground state solution of $(\mathcal{S}_{*})$ and $\overrightarrow{\widetilde{\mathbf{v}}}_\ve\to\overrightarrow{\widetilde{\mathbf{v}}}_0$ strongly in $\mathcal{H}_{\bbr^4}$ as $\ve\to0^+$.  Note that for $\beta>0$, $w^\ve_1=v_{1,1}^\ve-v_{1,1}^0$ satisfies the following equation by the Taylor expansion
\begin{eqnarray*}
-\Delta w^\ve_1+\lambda_1 w^\ve_1\leq\alpha_1(v_{1,1}^\ve+v_{1,1}^0)^{p-2}w^\ve_1+\mu_1(v_{1,1}^\ve+v_{1,1}^0)^{2}w^\ve_1+\beta(v_{1,2}^\ve)w_1^\ve+o(1)
\end{eqnarray*}
in $\bbr^4$.  Thus, we can obtain a contradiction similar to the case $\beta<0$.  Hence, we also have that $\Omega_{1,\ve}\to\bbr^4$ as $\ve\to0^+$ for $0<\beta<\beta_0$ or $\beta>\beta_1$.
\end{proof}

Let $v_i^\ve(x)=u_i^\ve(p_1^\ve+\ve x)$ and $\widetilde{v}_i^\ve(x)=u_i^\ve(p_2^\ve+\ve x)$ respectively for $i=1,2$.  Then we have the following.
\begin{proposition}\label{prop0004}
Let $\alpha_T$, $\beta_0$ and $\beta_1$ be respectively given by Lemmas~\ref{lem0002}--\ref{lem0003}.
\begin{enumerate}
\item[$(i)$]  $\overrightarrow{\mathbf{V}}_\ve^*=(v_1^\ve, \widetilde{v}_2^\ve)\to\overrightarrow{\mathbf{v}}_0$ strongly in $\mathcal{H}_{\bbr^4}$ as $\ve\to0^+$ in the following two cases:
\begin{enumerate}
\item $-\sqrt{\mu_1\mu_2}<\beta<0$ and $\alpha_1,\alpha_2>0$,
\item $\beta\leq-\sqrt{\mu_1\mu_2}$ and $\alpha_1,\alpha_2>0$ with $|\overrightarrow{\mathbf{\alpha}}|<\alpha_T$,
\end{enumerate}
where $v_i^0$ is a ground state solution of $(\mathcal{P}_i)$.  Moreover, $\frac{|p_1^\ve-p_2^\ve|}{\ve}\to+\infty$ as $\ve\to0^+$.
\item[$(ii)$]  For $\beta\in(0, \beta_0)\cup(\beta_1, +\infty)$, $\overrightarrow{\mathbf{V}}_\ve^*\to\overrightarrow{\mathbf{v}}_*$ strongly in $\mathcal{H}_{\bbr^4}$ as $\ve\to0^+$, where $\overrightarrow{\mathbf{v}}_*$ is the ground state solution of $(\mathcal{S}_*)$.  Moreover, $\frac{|p_1^\ve-p_2^\ve|}{\ve}\to0$ as $\ve\to0^+$.
\end{enumerate}
\end{proposition}
\begin{proof}
$(i)$\quad By a similar argument as used in the proof of Lemma~\ref{lem0004} for $\beta<0$, we can show that either $v_1^0\equiv0$ or $v_2^0\equiv0$.  By the scaling technique, we can see from Lemma~\ref{lem0003} and a similar argument as used for \eqref{eqn0013} that $\mathcal{B}_{v_i^\ve,4,\bbr^4}^{4}\geq C$ both for $i=1,2$.  If $\mathcal{B}_{v_i^\ve,p,\bbr^4}^{p}=o(1)$, then by the Sobolev inequality, $\beta<0$ and the fact that $\overrightarrow{\mathbf{v}}_\ve$ satisfies the system~$(\mathcal{S}_*)$, we obtain  that
$$
\mathcal{J}_{\Omega_{1,\ve},T}(\overrightarrow{\mathbf{v}}_\ve)\geq\sum_{i=1}^2\frac{1}{4\mu_i}\mathcal{S}^2+o(1),
$$
which contradicts to Lemma~\ref{lem0003} and  Proposition~\ref{propA0001}.  Thus, we must have $$\sum_{i=1}^2\mathcal{B}_{v_i^\ve,p,\bbr^4}^{p}\geq C.$$  If $\mathcal{B}_{v_1^\ve,p,\bbr^4}^{p}=o(1)$ then we must have that $\mathcal{B}_{v_2^\ve,p,\bbr^4}^{p}\geq C$.  Thanks to the H\"older and Sobolev inequalities, we can see that $\mathcal{B}_{v_2^\ve,2,\bbr^4}^{2}\geq C$.  Since $\overrightarrow{\mathbf{v}}_\ve$ satisfies the system~$(\mathcal{S}_*)$ and $\beta<0$, we see  that
\begin{eqnarray*}
\|v_{1}^{\ve}\|_{1,\bbr^4}^2\leq\mu_1\mathcal{B}_{v_{1}^{\ve},4,\bbr^4}^{4}+o(1)
\end{eqnarray*}
and
\begin{eqnarray*}
\|v_{2}^{\ve}\|_{2,\bbr^4}^2\leq\alpha_2\mathcal{B}_{v_{2}^{\ve},p,\bbr^4}^{p}+\mu_2\mathcal{B}_{v_{2}^{\ve},4,\bbr^4}^{4}+o(1).
\end{eqnarray*}
By a standard argument, we deduce that  either $v_1^\ve\to0$ strongly in $\mathcal{H}_{1,\bbr^4}$ as $\ve\to0^+$ or there exists $0<s_i^\ve\leq1+o(1)$ such that $$\frac12\|s_1^\ve v_{1}^{\ve}\|_{1,\bbr^4}^2-\frac{\mu_1}4\mathcal{B}_{s_1^\ve v_{1}^{\ve},4,\bbr^4}^{4}\geq\frac{1}{4\mu_1}\mathcal{S}^2; \quad s_2^\ve v_2^\ve\in\mathcal{M}_{2,\bbr^4}.$$  If  $0<s_i^\ve\leq1+o(1)$ such that $\frac12\|s_1^\ve v_{1}^{\ve}\|_{1,\bbr^4}^2-\frac{\mu_1}4\mathcal{B}_{s_1^\ve v_{1}^{\ve},4,\bbr^4}^{4}\geq\frac{1}{4\mu_1}\mathcal{S}^2$ and $s_2^\ve v_2^\ve\in\mathcal{M}_{2,\bbr^4}$, then by a similar argument as used in  Lemma~\ref{lem0001} and $\beta<0$, we observe    that
\begin{eqnarray*}
\ve^{-4}\mathcal{J}_{\ve,\Omega,T}(\overrightarrow{\mathbf{u}}_\ve)&\geq&\ve^{-4}\mathcal{J}_{\ve,\Omega,T}(\overrightarrow{\mathbf{s}}_\ve\circ\overrightarrow{\mathbf{u}}_\ve)\\
&=&\mathcal{J}_{\Omega_{1,\ve},T}(\overrightarrow{\mathbf{s}}_\ve\circ\overrightarrow{\mathbf{v}}_\ve)\\
&\geq&\frac12\|s_1^\ve v_{1}^{\ve}\|_{1,\bbr^4}^2-\frac{\mu_1}4\mathcal{B}_{s_1^\ve v_{1}^{\ve},4,\bbr^4}^{4}+\mathcal{E}_{2,\bbr^4}(s_2^\ve v_2^\ve)\\
&\geq&\frac{1}{4\mu_1}\mathcal{S}^2+d_{2,\bbr^4},
\end{eqnarray*}
which also contradicts to  Lemma~\ref{lem0003}.  Thus, we must have $v_1^\ve\to0$ strongly in $\mathcal{H}_{1,\bbr^4}$ as $\ve\to0^+$ in this case.  Note that $\beta<0$ and $\overrightarrow{\mathbf{v}}_\ve$ satisfies the system~$(\mathcal{S}_*)$, by applying the Moser's iteration as in \cite{BZZ13}, we can show that $v_{1}^\ve\to0$ strongly in $L^q(\bbr^4)$ as $\ve\to0^+$ for all $q\geq2$.  Thanks to the classical $L^p$ estimate for elliptic equations and the Sobolev embedding theorem, we also have that $v_{1}^\ve\to0$ strongly in $C^1(\bbr^4)$.  Therefore, $v_{1}^\ve(0)\to0$ as $\ve\to0^+$.  On the other hand, since $p_1^\ve$ is the maximum point of $u_1^\ve$, $0$ is the maximum point of $v_1^0$.  By $p>2$, we can see from $\beta<0$ that
\begin{eqnarray}\label{eqn9994}
\lambda_1 v_1^\ve(0)\leq\alpha_1 (v_1^\ve(0))^{p-1}+\mu_1(v_1^\ve(0))^{2},
\end{eqnarray}
which implies $v_1^\ve(0)\geq C$.  It is impossible too.  Thus, we must have $\mathcal{B}_{v_1^\ve,p,\bbr^4}^{p}\geq C$.  If we also have $\mathcal{B}_{v_2^\ve,p,\bbr^4}^{p}\geq C$,
then by the H\"older and Sobolev inequalities, we  see that $\mathcal{B}_{v_i^\ve,2,\bbr^4}^{2}\geq C$, $i=1,2$.  Thus, by $\beta<0$, there exists $0<s_i^\ve\leq1$ such that $s_i^\ve v_i^\ve\in\mathcal{M}_{i,\bbr^4}$ both for $i=1,2$.  By a similar argument as used in Lemma~\ref{lem0001} and note $\beta<0$, we obtain   that
\begin{eqnarray}
\ve^{-4}\mathcal{J}_{\ve,\Omega,T}(\overrightarrow{\mathbf{u}}_\ve)&\geq&\ve^{-4}\mathcal{J}_{\ve,\Omega,T}(\overrightarrow{\mathbf{s}}_\ve\circ\overrightarrow{\mathbf{u}}_\ve)\notag\\
&=&\mathcal{J}_{\Omega_{1,\ve},T}(\overrightarrow{\mathbf{s}}_\ve\circ\overrightarrow{\mathbf{v}}_\ve)\notag\\
&\geq&\sum_{i=1}^2\mathcal{E}_{i,\bbr^4}(s_i^\ve v_i^\ve)\notag\\
&\geq&\sum_{i=1}^2d_{i,\bbr^4},\label{eqn9990}
\end{eqnarray}
combining this  with Lemma~\ref{lem0005}, it implies that $\overrightarrow{\mathbf{s}}_\ve=\overrightarrow{\mathbf{1}}+o(1)$ and $\mathcal{B}_{|v_1^\ve|^{2}|v_2^\ve|^{2},1,\bbr^4}=o(1)$.  On the other hand, since $\mathcal{B}_{v_i^\ve,p,\bbr^4}^{p}\geq C$ in this situation, by the Lions' Lemma, there exist $\{y_j^\ve\}\subset\bbr^4$ such that $\overrightarrow{\overline{\mathbf{v}}}_\ve=(v_{1,j}^\ve, v_{2,j}^\ve)\rightharpoonup \overrightarrow{\overline{\mathbf{v}}}_0=(v_{1,j}^0, v_{2,j}^0)$ weakly in $\mathcal{H}_{\bbr^4}$ as $\ve\to0^+$, where $v_{i,j}^\ve(x)=v_i(x+y_j^\ve)$ for $i,j=1,2$.
Moreover, $v_{i,i}^0\not\equiv0$  for $i=1,2$.  Let $\overrightarrow{\mathbf{w}}_\ve=\overrightarrow{\overline{\mathbf{v}}}_\ve-\overrightarrow{\overline{\mathbf{v}}}_0$. Then $\overrightarrow{\mathbf{w}}_\ve\rightharpoonup\overrightarrow{\mathbf{0}}$ weakly in $\mathcal{H}_{\bbr^4}$ as $\ve\to0^+$.  Since $v_{1,1}^0\not=0$, by Proposition~\ref{prop0003}, we must have that $\overrightarrow{\overline{\mathbf{v}}}_0$ is a solution of $(\mathcal{S}_*)$.  Thus, by the Brez\'is-Lieb lemma and \cite[Lemma~2.3]{CZ15}, we have
\begin{eqnarray}\label{eqn9989}
\ve^{-4}\mathcal{J}_{\ve,\Omega,T}(\overrightarrow{\mathbf{u}}_\ve)=\mathcal{J}_{\Omega_{1,\ve},T}(\overrightarrow{\mathbf{v}}_\ve)=\mathcal{J}_{\bbr^4,T}(\overrightarrow{\overline{\mathbf{v}}}_0)
+\mathcal{J}_{\bbr^4,T}(\overrightarrow{\mathbf{w}}_\ve)+o(1)
\end{eqnarray}
and
\begin{eqnarray*}
\mathcal{J}_{\bbr^4,T}'(\overrightarrow{\mathbf{w}}_\ve)\overrightarrow{\mathbf{w}}_{\ve}^1=\mathcal{J}_{\bbr^4,T}'(\overrightarrow{\mathbf{w}}_\ve)\overrightarrow{\mathbf{w}}_{\ve}^2=o(1),
\end{eqnarray*}
where $\overrightarrow{\mathbf{w}}_{\ve}^1=(w_{1}^\ve, 0)$ and $\overrightarrow{\mathbf{w}}_{\ve}^2=(0, w_{2}^\ve)$.  As above, we also have $\mathcal{J}_{\bbr^4,T}(\overrightarrow{\mathbf{w}}_\ve)\geq0$.  Hence, by Proposition~\ref{prop0002} and Lemma~\ref{lem0004}, we can see from $\beta<0$ and Lemma~\ref{lem0005} that $v_{2,1}^0\equiv0$.  It follows that $v_{1,1}^0$ is a solution of $(\mathcal{P}_1)$.  Since $v_{1,1}^0\geq0$ and $v_{1,1}^0\not\equiv0$, by the maximum principle and a standard argument, we must have  $\mathcal{E}_{1,\bbr^4}(v_{1,1}^0)\geq d_{1,\bbr^4}$.  Similarly, we also have $v_{1,2}^0\equiv0$ and $v_{2,2}^0\not\equiv0$ with $\mathcal{E}_{2,\bbr^4}(v_{2,2}^0)\geq d_{2,\bbr^4}$.  Since $\mathcal{B}_{|v_1^\ve|^{2}|v_2^\ve|^{2},1,\bbr^4}=o(1)$, by translation and a similar calculation in \eqref{eqn9990}, we can see that $v_{i,i}^\ve\to v_{i,i}^0$ strongly in $\mathcal{H}_{1,\bbr^4}$ as $\ve\to0^+$ for $i=1,2$.  Note that $\beta<0$ and $\overrightarrow{\mathbf{v}}_\ve$ satisfies the system~$(\mathcal{S}_*)$, by applying a modified Moser's iteration (cf. \cite{BZZ13}), we can show that $v_{1,1}^\ve$ is uniformly bounded in $L^q(\bbr^4)$ for all $q\geq2$.  Since $v_{2,1}^0\equiv0$, $v_{1,1}^0$ is also a nontrivial solution of $(\mathcal{P}_1)$.  Thanks to the classical elliptic regularity, we can see that $v_{1,1}^0\in L^\infty(\bbr^4)$.  By Taylor expansion and $\beta<0$, we see  that $w^\ve_1=v_{1,1}^\ve-v_{1,1}^0$ satisfies the following equation:
\begin{eqnarray*}
-\Delta w^\ve_1+\lambda_1 w^\ve_1\leq\alpha_1(p-1)(v_{1,1}^\ve+v_{1,1}^0)^{p-2}w^\ve_1+3\mu_1(v_{1,1}^\ve+v_{1,1}^0)^{2}w^\ve_1+o(1)
\end{eqnarray*}
in $\bbr^4$.
Since $w^\ve_1\to0$ strongly in $\mathcal{H}_{1,\bbr^4}$, by applying the Moser's iteration as in \cite{BZZ13}, we can obtain that $w^\ve_1\to0$ strongly in $L^q(\bbr^4)$ for all $q\geq2$.  It follows from \cite[Theorem~8.17]{GT98} (see also \cite[Lemma~4.3]{BZZ13}) that $v_{1,1}^\ve\to0$ as $|x|\to+\infty$ uniformly for $\ve$.  Note that $v_{1,1}^\ve(-y_\ve)=v_1^\ve(0)$, we get a contradiction.  Thus, we must have that $\mathcal{B}_{v_2^\ve,p,\bbr^4}^{p}=o(1)$, which implies $v_1^0\not\equiv0$ and $v_2^0\equiv0$.  Similarly, we can also show that $\widetilde{v}_1^0\equiv0$ and $\widetilde{v}_2^0\not\equiv0$, where $\overrightarrow{\widetilde{\mathbf{v}}}_\ve\rightharpoonup\overrightarrow{\widetilde{\mathbf{v}}}_0$ weakly in $\mathcal{H}_{\bbr^4}$ as $\ve\to0^+$.  Now, by Proposition~\ref{prop0002} and Lemma~\ref{lem0005}, we can easy to show that $\overrightarrow{\mathbf{V}}_\ve^*=(v_1^\ve, \widetilde{v}_2^\ve)\to\overrightarrow{\mathbf{v}}_0$ strongly in $\mathcal{H}_{\bbr^4}$ as $\ve\to0^+$, where $v_i^0$ is a ground state solution of $(\mathcal{P}_i)$, $i=1,2$.  On the other hand, since $\beta<0$ and $\overrightarrow{\mathbf{v}}_\ve$ satisfies the system~$(\mathcal{S}_*)$, by applying the Moser's iteration as in \cite{BZZ13}, we can show that $\widetilde{v}_{2}^\ve\to0$ strongly in $L^q(\bbr^4)$ as $\ve\to0$ for all $q\geq2$.  Thanks to the classical elliptic regularity, we can see that $\widetilde{v}_{2}^\ve\to0$ strongly in $C^1(\bbr^4)$.  It follows from $\widetilde{v}_{2}^\ve(x)=v_{2}^\ve(x+\frac{p_\ve^2-p_1^\ve}{\ve})$ that $\frac{|p_1^\ve-p_2^\ve|}{\ve}\to+\infty$ since $v_2^\ve\rightharpoonup0$ in $\mathcal{H}_{2,\bbr^4}$ as $\ve\to0^+$.


\vskip0.2in
$(ii)$\quad The proof is similar to that of $(i)$, so we only point out the differences.  Suppose that $\mathcal{B}_{|v_1^\ve|^{2}|v_2^\ve|^{2},1,\bbr^4}=o(1)$.  Then by a similar argument as used in $(i)$ for \eqref{eqn9990}, we can see from Lemma~\ref{lem0005} that $B\geq\sum_{i=1}^2d_{i,\bbr^4}$, which is impossible owing to Proposition~\ref{prop0002}.  Thus, we must have $\mathcal{B}_{|v_1^\ve|^{2}|v_2^\ve|^{2},1,\bbr^4}\geq C$ for $\beta\in(0, \beta_0)\cup(\beta_1, +\infty)$.  By applying the Lion's lemma as for \eqref{eqnew8003}, we can show that there exists $\{y^\ve\}\subset\bbr^4$ such that $v_{i,1}^\ve(x)=v_i^\ve(x+y^\ve)\rightharpoonup v_{i,1}^0$ weakly in $\mathcal{H}_{i,\bbr^4}$ as $\ve\to0^+$ and $v_{1,1}^0\not\equiv0$.  Now, by a similar argument as used in the proof of Lemma~\ref{lem0004} for $\beta>0$, we can see that $v_{2,1}^0\not=0$ and $v_{i,1}^\ve\to v_{i,1}^0$ strongly in $\mathcal{H}_{i,\bbr^4}$ as $\ve\to0^+$.  By applying a similar regularity argument as used in $(i)$, we can obtain that
$v_{i,1}^\ve\to v_{i,1}^0$ strongly $\mathcal{L}^\infty(\bbr^4)\cap \mathcal{C}^1(\bbr^4)$ as $\ve\to0^+$, where $\mathcal{L}^\infty(\bbr^4)=L^\infty(\bbr^4)\times L^\infty(\bbr^4)$ and $\mathcal{C}^1(\bbr^4)=(C^1(\bbr^4))^2$.  Since $0$ is the maximum point of $v_1^\ve$, we must have that $v_1^0\not\equiv0$, which implies $|y^\ve|\leq C$ and $v_2^0\not\equiv0$ too.  Still by a similar argument as used for \eqref{eqn9989}, we can see that $v_{i}^\ve\to v_{i}^0$ strongly in $\mathcal{H}_{i,\bbr^4}$ as $\ve\to0^+$.  Also by applying a similar regularity argument as used in $(i)$, we can obtain that $v_{i}^\ve\to v_{i}^0$ strongly in $\mathcal{L}^\infty(\bbr^4)\cap \mathcal{C}^1(\bbr^4)$ as $\ve\to0^+$.  Similarly, we also have that $\overrightarrow{\widetilde{\mathbf{v}}}_\ve\to\overrightarrow{\widetilde{\mathbf{v}}}_0$ strongly in $\mathcal{L}^\infty(\bbr^4)\cap \mathcal{C}^1(\bbr^4)$ as $\ve\to0^+$, where both $\overrightarrow{\mathbf{v}}_0$ and $\overrightarrow{\widetilde{\mathbf{v}}}_0$ are the ground state solution of $(\mathcal{S}_*)$.  By the results in \cite{BJM09}, we can see from $\beta>0$ that $\overrightarrow{\mathbf{v}}_0$ and $\overrightarrow{\widetilde{\mathbf{v}}}_0$ are both radial symmetric and monotonicity.  Recall that $p_i^\ve$ is the maximum point of $u_i^\ve$,$i=1,2$,  Thus, by $\overrightarrow{\mathbf{v}}_\ve\to\overrightarrow{\mathbf{v}}_0$ and $\overrightarrow{\widetilde{\mathbf{v}}}_\ve\to\overrightarrow{\widetilde{\mathbf{v}}}_0$ strongly in $\mathcal{C}^1(\bbr^4)$ as $\ve\to0^+$, we must have that $0$ is the maximum point of $v_i^0$ and $\widetilde{v}_i^0$ for all $i=1,2$.  Suppose that $\frac{|p_1^\ve-p_2^\ve|}{\ve}\to+\infty$, then by the fact that $\widetilde{v}_2^\ve(x)=v_2^\ve(x+\frac{p_1^\ve-p_2^\ve}{\ve})$, we actually have from $\overrightarrow{\mathbf{v}}_\ve\to\overrightarrow{\mathbf{v}}_0$ and $\overrightarrow{\widetilde{\mathbf{v}}}_\ve\to\overrightarrow{\widetilde{\mathbf{v}}}_0$ strongly in $\mathcal{L}^\infty(\bbr^4)$ as $\ve\to0^+$ that $v_2^0(0)=0$, which is impossible since $v_2^0>0$ in $\bbr^4$ by the maximum principle.  Thus, we must have $|\frac{p_1^\ve-p_2^\ve}{\ve}|\leq C$.  Without loss of generality, we assume that $\frac{|p_1^\ve-p_2^\ve|}{\ve}\to p_0$ as $\ve\to0^+$.  We claim that $p_0=0$.  Indeed, by the fact that $\widetilde{v}_2^\ve(x)=v_2^\ve(x+\frac{p_1^\ve-p_2^\ve}{\ve})$, we can see that $v_2^0(x+p_0)=\widetilde{v}_2^0(x)$ for all $x\in\bbr^4$, in particular, $v_2^0(p_0)=\widetilde{v}_2^0(0)$ and $v_2^0(0)=\widetilde{v}_2^0(-p_0)$.  Note that $0$ is the maximum point of $v_2^0$ and $\widetilde{v}_2^0$, thus we must have $v_2^0(0)=\widetilde{v}_2^0(0)$.  It follows that $v_2^0(p_0)=v_2^0(0)$ and $\widetilde{v}_2^0(-p_0)=\widetilde{v}_2^0(0)$.  Since $\widetilde{v}_2^0$ is radial symmetric, we also have that $\widetilde{v}_2^0(p_0)=\widetilde{v}_2^0(0)$.  Now, by iteration, we can see that $v_2^0(0)=v_2^0(kp_0)$ for all $k\in\bbn$.  It is impossible if $p_0\not=0$ since $v_2^0(kp_0)\to0$ as $k\to\infty$ in this situation.  Now, set $\overrightarrow{\mathbf{v}}_*=(v_1^0, \widetilde{v}_2^0)$.  Since $\frac{|p_1^\ve-p_2^\ve|}{\ve}\to0$ as $\ve\to0^+$ and $\overrightarrow{\mathbf{v}}_\ve\to\overrightarrow{\mathbf{v}}_0$ and $\overrightarrow{\widetilde{\mathbf{v}}}_\ve\to\overrightarrow{\widetilde{\mathbf{v}}}_0$ strongly in $\mathcal{H}_{\bbr^4}\cap\mathcal{C}^1(\bbr^4)$ as $\ve\to0^+$, we have that $\overrightarrow{\mathbf{V}}_\ve^*\to\overrightarrow{\mathbf{v}}_*$ strongly in $\mathcal{H}_{\bbr^4}\cap\mathcal{C}^1(\bbr^4)$ as $\ve\to0^+$.  Moreover, by the fact that $\frac{|p_1^\ve-p_2^\ve|}{\ve}\to0$ once more and Proposition~\ref{prop0002}, we also have that $\overrightarrow{\mathbf{v}}_*$ is the ground state solution of $(\mathcal{S}_*)$.
\end{proof}

\noindent\textbf{Proof of Theorem~\ref{thm0002}:}\quad  It follows immediately from Proposition~\ref{prop0004}.
\hfill$\Box$


\section{The locations of the spikes as $\ve\to0^+$}
We first study the location of the spikes for $\beta>0$.  Without loss of generality, we assume that $$\text{dist}(0, \partial\Omega)=\mathcal{D}=\max_{p\in\Omega}\text{dist}(p, \partial\Omega).$$  Then it is easy to see that $\mathcal{D}>0$ and $\mathbb{B}_{\mathcal{D}}\subset\Omega$.  Let us consider the following system
\begin{equation*}
\left\{\aligned&-\ve^2\Delta u_1+\lambda_1u_1=\mu_1u_1^3+\alpha_1u_1^{p-1}+\beta u_2^2u_1\quad&\text{in }\mathbb{B}_{\mathcal{D}},\\
&-\ve^2\Delta u_2+\lambda_2u_2=\mu_2u_2^3+\alpha_2u_2^{p-1}+\beta u_1^2u_2\quad&\text{in }\mathbb{B}_{\mathcal{D}},\\
&u_1,u_2>0\quad\text{in }\mathbb{B}_{\mathcal{D}},\quad u_1=u_2=0\quad\text{on }\partial\mathbb{B}_{\mathcal{D}}.\endaligned\right.\eqno{(\mathcal{S}^0_{\ve})}
\end{equation*}
Then by a similar argument as used for Proposition~\ref{prop0001},  we can show that $(\mathcal{S}^0_{\ve})$ has a ground state solution $\overrightarrow{\widetilde{\mathbf{u}}}_\ve$ for $0<\beta<\beta_0$ or $\beta>\beta_1$ with $\mathcal{J}_{\ve,\mathbb{B}_{\mathcal{D}},T}(\overrightarrow{\widetilde{\mathbf{u}}}_\ve)=c_{\ve,\mathbb{B}_{\mathcal{D}},T}$.  Since $\mathbb{B}_{\mathcal{D}}$ is radial symmetric and $\beta>0$, by the Schwartz's symmetrization, $\overrightarrow{\widetilde{\mathbf{u}}}_\ve$ is also radial symmetric.

\begin{lemma}\label{lem0006}
For every $\sigma>0$ small enough, we have $\widetilde{u}_{i}^\ve+\ve|\nabla\widetilde{u}_{i}^\ve|\leq Ce^{-\frac{(1-\sigma)\sqrt{\lambda_i}|x|}{\ve}}$ with $\ve>0$ small enough.
\end{lemma}

\begin{proof}
Since $\overrightarrow{\widetilde{\mathbf{u}}}_\ve$ is a solution of $(\mathcal{S}^0_{\ve})$, by the classical elliptic regularity theory, we can see that $\overrightarrow{\widetilde{\mathbf{u}}}_\ve\in\mathcal{C}^2(\mathbb{B}_{\mathcal{D}})=(C^2(\mathbb{B}_{\mathcal{D}}))^2$.  It follows that $\overrightarrow{\overline{\mathbf{u}}}_\ve=(\overline{u}_1^\ve, \overline{u}_2^\ve)\in\mathcal{C}^2(\mathbb{B}_{\frac{\mathcal{D}}{\ve}})$ with $\overline{u}_i^\ve(x)=\widetilde{u}_{i}^\ve(\ve x)$.  Moreover, $\overrightarrow{\overline{\mathbf{u}}}_\ve$ also satisfies the following system
\begin{equation*}
\left\{\aligned&-\Delta u_1+\lambda_1u_1=\mu_1u_1^3+\alpha_1u_1^{p-1}+\beta u_2^2u_1\quad&\text{in }\mathbb{B}_{\frac{\mathcal{D}}{\ve}},\\
&-\Delta u_2+\lambda_2u_2=\mu_2u_2^3+\alpha_2u_2^{p-1}+\beta u_1^2u_2\quad&\text{in }\mathbb{B}_{\frac{\mathcal{D}}{\ve}},\\
&u_1,u_2>0\quad\text{in }\mathbb{B}_{\frac{\mathcal{D}}{\ve}},\quad u_1=u_2=0\quad\text{on }\partial\mathbb{B}_{\frac{\mathcal{D}}{\ve}}.\endaligned\right.\eqno{(\mathcal{S}^{00}_{\ve})}
\end{equation*}
By a similar argument as used for Proposition~\ref{prop0004}, we can see that $\overrightarrow{\overline{\mathbf{u}}}_\ve\to\overrightarrow{\mathbf{v}}_*$ strongly in $\mathcal{H}_{\bbr^4}\cap \mathcal{L}^\infty(\bbr^4)$ as $\ve\to0^+$, where $\overrightarrow{\mathbf{v}}_*$ is the ground state solution of $(\mathcal{S}_*)$.  It follows from the Moser's iteration as used in \cite{BZZ13} and \cite[Theorem~8.17]{GT98} (see also \cite[Lemma~4.3]{BZZ13}) that
\begin{eqnarray*}
\lim_{|x|\to+\infty}\overline{u}_i^\ve=0\quad\text{uniformly for }\ve>0\text{ small enough}.
\end{eqnarray*}
Thus, since $\lambda_i>0$, by using the maximum principle in a standard way, we can see that $\overline{u}_i^\ve(x)\leq Ce^{-(1-\sigma)\sqrt{\lambda_i}|x|}$ with $\ve>0$ small enough for every $\sigma>0$ small enough.  Therefore, $\widetilde{u}_{i}^\ve\leq Ce^{-\frac{(1-\sigma)\sqrt{\lambda_i}|x|}{\ve}}$.  On the other hand, since $\overrightarrow{\widetilde{\mathbf{u}}}_\ve$ is radial symmetric, by the fact that $\overrightarrow{\overline{\mathbf{u}}}_\ve$ satisfies $(\mathcal{S}^{00}_{\ve})$, we can see that
\begin{eqnarray*}
\frac{\lambda_i}{2}\overline{u}_{i}^\ve(r)\leq(\overline{u}_{i}^\ve(r))''+\frac{3}{r}(\overline{u}_{i}^\ve(r))'\leq\lambda_i\overline{u}_{i}^\ve(r)
\end{eqnarray*}
for $\frac{\mathcal{D}}{\ve}-1<r<\frac{\mathcal{D}}{\ve}$ with $\ve>0$ small enough.  By the results in \cite{BJM09}, $(\overline{u}_{i}^\ve(r))'\leq0$.  Thus, we must have $(\overline{u}_{i}^\ve(r))''>0$.  If $(\overline{u}_{i}^\ve(\frac{\mathcal{D}}{\ve}))'\leq-C$, then by the Harnack inequality, we have $(\overline{u}_{i}^\ve(r))'\leq C\overline{u}_{i}^\ve(r)$ for $\frac{\mathcal{D}}{\ve}-1<r<\frac{\mathcal{D}}{\ve}$.  Otherwise, by integrating the above inequality on the interval $[r, \frac{\mathcal{D}}{\ve}]$, we have $(\overline{u}_{i}^\ve(r))'\leq C\overline{u}_{i}^\ve(r)+Ce^{-(1-\sigma)\sqrt{\lambda_i}r}$.  It follows that $|\nabla \overline{u}_i^\ve|\leq Ce^{-(1-\sigma)\sqrt{\lambda_i}|x|}$ with $\ve>0$ small enough for every $\sigma>0$ small enough, which implies $\ve|\nabla\widetilde{u}_{i}^\ve|\leq Ce^{-\frac{(1-\sigma)\sqrt{\lambda_i}|x|}{\ve}}$.
\end{proof}

Recall that $\text{dist}(0, \partial\Omega)=\mathcal{D}=\max_{p\in\Omega}\text{dist}(p, \partial\Omega)$, then we have the following upper bound of $c_{\ve,\Omega,T}$ with $\beta>0$.

\vskip0.23in

\begin{lemma}\label{lem0007}
Let $\beta>0$ and $\ve>0$ be small enough.  Then for every $\sigma>0$ small enough, we have
\begin{eqnarray*}
c_{\ve,\Omega,T}\leq\ve^4(B+C\sum_{i=1}^2e^{-\frac{2(1-\sigma)\sqrt{\lambda_i}\mathcal{D}}{\ve}})
\end{eqnarray*}
with $0<\beta<\beta_0$ while
\begin{eqnarray*}
c_{\ve,\Omega,T}\leq\ve^4(B'+C\sum_{i=1}^2e^{-\frac{2(1-\sigma)\sqrt{\lambda_i}\mathcal{D}}{\ve}})
\end{eqnarray*}
with $\beta>\beta_1$.
\end{lemma}
\begin{proof}
Let $\phi_\ve$ be a smooth radial symmetric function such that $0\leq\phi_\ve\leq1$ and
\begin{eqnarray*}
\phi_\ve(x)=\left\{\aligned &1,\quad&x\in \mathbb{B}_{\frac{\mathcal{D}}{\ve}-1};\\
&0,\quad&x\in\bbr^4\backslash\mathbb{B}_{\frac{\mathcal{D}}{\ve}}.\endaligned\right.
\end{eqnarray*}
Set $\overrightarrow{\overline{\mathbf{U}}}_{*,\ve}=(U_1^*\phi_\ve, U_2^*\phi_\ve)$, where $\overrightarrow{\mathbf{U}}_*$ is the ground state solution of $(\mathcal{S}_*)$ given by Proposition~\ref{prop0002}.
Then $\overrightarrow{\overline{\mathbf{U}}}_{*,\ve}\in\mathcal{H}_{\Omega_\ve}$.  Note that $\overrightarrow{\overline{\mathbf{U}}}_{*,\ve}\to\overrightarrow{\mathbf{U}}_*$ strongly in $\mathcal{H}_{\bbr^4}$ as $\ve\to0$, by a similar argument as used in the proof of Lemma~\ref{lem0005}, we can show that there exits $\overrightarrow{\mathbf{t}}_{\ve}$ with $\overrightarrow{\mathbf{t}}_{\ve}\to\overrightarrow{\mathbf{1}}$ as $\ve\to0^+$ such that $\overrightarrow{\mathbf{t}}_{\ve}\circ\overrightarrow{\overline{\mathbf{U}}}_{*,\ve}\in\mathcal{N}_{\Omega_\ve,T}$.  It follows from Lemma~\ref{lem0006} and a similar argument as used for Lemma~\ref{lem0001} that
\begin{eqnarray*}
B&=&\mathcal{J}_{\bbr^4,T}(\overrightarrow{\mathbf{U}}_*)\\
&\geq&\mathcal{J}_{\bbr^4,T}(\overrightarrow{\mathbf{t}}_{\ve}\circ\overrightarrow{\mathbf{U}}_*)\\
&\geq&\mathcal{J}_{\Omega_\ve,T}(\overrightarrow{\mathbf{t}}_{\ve}\circ\overrightarrow{\overline{\mathbf{U}}}_*)-C\sum_{i=1}^2e^{-\frac{2(1-\sigma)\sqrt{\lambda_i}\mathcal{D}}{\ve}}.
\end{eqnarray*}
By the standard scaling technique, we actually have that
$$
c_{\ve,\Omega,T}\leq\ve^4(B+C\sum_{i=1}^2e^{-\frac{2(1-\sigma)\sqrt{\lambda_i}\mathcal{D}}{\ve}})
$$
for $0<\beta<\beta_0$.  Similarly, we also have that
$$
c_{\ve,\Omega,T}\leq\ve^4(B'+C\sum_{i=1}^2e^{-\frac{2(1-\sigma)\sqrt{\lambda_i}\mathcal{D}}{\ve}})
$$
for $\beta>\beta_1$.
\end{proof}

Recall that $\overrightarrow{\mathbf{u}}_\ve$ is the ground state solution of $(\mathcal{S}_\ve)$ for $0<\beta<\beta_0$ or $\beta>\beta_1$.
\begin{lemma}\label{lem0008}
Let $0<\beta<\beta_0$ or $\beta>\beta_1$.  Then for every $\sigma>0$ small enough, we have $u_{i}^\ve+\ve|\nabla u_{i}^\ve|\leq Ce^{-\frac{(1-\sigma)\sqrt{\lambda_i}|x-p_i^\ve|}{\ve}}$
with $\ve>0$ small enough,  $i=1,2$.
\end{lemma}
\begin{proof}
Recall that $v_1^\ve(x)=u_1^\ve(p_1^\ve+\ve x)$ and $\widetilde{v}_2^\ve(x)=u_2^\ve(p_2^\ve+\ve x)$.  By Proposition~\ref{prop0004}, we can see that $(v_1^\ve, \widetilde{v}_2^\ve)\to\overrightarrow{\mathbf{v}}_*$ strongly in $\mathcal{H}_{\bbr^4}\cap \mathcal{L}^\infty(\bbr^4)$ as $\ve\to0^+$.  It follows from the Moser's iteration as used in \cite{BZZ13} and \cite[Theorem~8.17]{GT98} (see also \cite[Lemma~4.3]{BZZ13}) that
\begin{eqnarray*}
\lim_{|x|\to+\infty}v_1^\ve=\lim_{|x|\to+\infty}\widetilde{v}_2^\ve=0\quad\text{uniformly for }\ve>0\text{ small enough}.
\end{eqnarray*}
Since $\lambda_i>0$, by applying the maximum principle in a standard way, we can show that
\begin{eqnarray*}
v_1^\ve\leq Ce^{-(1-\sigma)\sqrt{\lambda_1}|x|}\quad\text{and}\quad \widetilde{v}_2^\ve\leq Ce^{-(1-\sigma)\sqrt{\lambda_2}|x|}
\end{eqnarray*}
for all $\sigma>0$ small enough with $\ve>0$ small enough.  Hence, by the standard scaling technique, we also have that $u_{i}^\ve\leq Ce^{-\frac{(1-\sigma)\sqrt{\lambda_i}|x-p_i^\ve|}{\ve}}$ respectively for $i=1,2$.  By the the Harnack inequality, we also have $u_{i}^\ve+\ve|\nabla u_{i}^\ve|\leq Ce^{-\frac{(1-\sigma)\sqrt{\lambda_i}|x-p_i^\ve|}{\ve}}$
with $\ve>0$ small enough respectively for $i=1,2$.
\end{proof}

Set $\mathcal{D}_i^\ve=\text{dist}(p_i^\ve, \partial\Omega)$ respectively for $i=1,2$, then $\mathbb{B}_{\mathcal{D}_i^\ve}(p_i^\ve)\subset\Omega$ both for $i=1,2$.  Without loss of generality, by Proposition~\ref{prop0004}, we may assume that $\mathcal{D}_i^\ve\to\mathcal{D}_0=\text{dist}(p_0, \partial\Omega)$ as $\ve\to0^+$ with $p_i^\ve\to p_0$ as $\ve\to0^+$ for $0<\beta<\beta_0$ or $\beta>\beta_1$.  As in \cite{LW05}, for $\delta>0$, we choose $\mathcal{D}_0'<\mathcal{D}_0+\delta$ such that
\begin{eqnarray*}
\text{meas}(\mathbb{B}_{\mathcal{D}_0'}(p_0))=\text{meas}(\mathbb{B}_{\mathcal{D}_0+\delta}(p_0)\cap\Omega).
\end{eqnarray*}
We also choose $\delta'<\delta$ such that $\mathcal{D}_0'<\mathcal{D}_0+\delta'$.  Now, consider the following smooth cut-off function
\begin{eqnarray*}
\eta_i^\ve(s)=\left\{\aligned&1,\quad\text{for }0\leq s\leq\mathcal{D}_i^\ve+\delta',\\
&0,\quad\text{for }s>\mathcal{D}_i^\ve+\delta\endaligned\right.
\end{eqnarray*}
with $0\leq\eta_i^\ve\leq1$ and $|(\eta_i^\ve)'|\leq C$.  Let $\overrightarrow{\widehat{\mathbf{u}}}_\ve=(\widehat{u}_1^\ve, \widehat{u}_2^\ve)$ with $\widehat{u}_i^\ve=u_i^\ve(x)\eta_i^\ve(|x-p_i^\ve|)$.  Then by Lemma~\ref{lem0001}, Proposition~\ref{prop0004} and Lemma~\ref{lem0008}, we have
\begin{eqnarray}\label{eqn0060}
\mathcal{J}_{\ve,\Omega,T}(\overrightarrow{\mathbf{u}}_\ve)\geq\mathcal{J}_{\ve,\Omega,T}(\overrightarrow{\mathbf{t}}\circ\overrightarrow{\widehat{\mathbf{u}}}_\ve)
-C(\overrightarrow{\mathbf{t}})\sum_{i=1}^2e^{-\frac{2(1-\sigma)\sqrt{\lambda_i}(\mathcal{D}_i^\ve+\delta')}{\ve}}
\end{eqnarray}
for $0<\beta<\beta_0$ or $\beta>\beta_1$.  Here, $C(\overrightarrow{\mathbf{t}})$ is bounded from above if $\overrightarrow{\mathbf{t}}$ is bounded from above.  Let $R_\ve=\frac{\widetilde{\mathcal{D}}_1^\ve}{\ve}$ where $\widetilde{\mathcal{D}}_1^\ve$ is chosen such that
\begin{eqnarray*}
\text{meas}(\mathbb{B}_{\widetilde{\mathcal{D}}_1^\ve}(p_1^\ve))=\text{meas}(\mathbb{B}_{\mathcal{D}_1^\ve+\delta}(p_1^\ve)\cap\Omega)\quad\text{and}\quad
\widetilde{\mathcal{D}}_1^\ve>\mathcal{D}_1^\ve+\frac12\delta.
\end{eqnarray*}
Moreover, by the Schwartz's symmetrization and $\beta>0$, we have
\begin{eqnarray}\label{eqn0061}
\mathcal{J}_{\ve,\Omega,T}(\overrightarrow{\mathbf{t}}\circ\overrightarrow{\widehat{\mathbf{u}}}_\ve)\geq
\ve^4\mathcal{J}_{\mathbb{B}_{R_\ve},T}(\overrightarrow{\mathbf{t}}\circ\overrightarrow{\widehat{\mathbf{u}}}_{\ve,*}).
\end{eqnarray}
Here, $\widehat{u}_i^{\ve,*}$ is the Schwartz's symmetrization of $\widehat{u}_i^\ve$ in $\mathbb{B}_{\widetilde{\mathcal{D}}_1^\ve}$ and $\mathcal{J}_{\mathbb{B}_{R_\ve},T}(\overrightarrow{\mathbf{u}})=\mathcal{J}_{1,\mathbb{B}_{R_\ve},T}(\overrightarrow{\mathbf{u}})$.
\begin{lemma}\label{lem0009}
Let $\beta>0$.  Then for every $\delta>0$ small enough, we have
\begin{eqnarray*}
c_{\ve,\Omega,T}\geq\ve^4(B+C\sum_{i=1}^2e^{-\frac{2(1+\sigma)\sqrt{\lambda_i}(\mathcal{D}_1^\ve+\frac12\delta)}{\ve}})
\end{eqnarray*}
with $0<\beta<\beta_0$ and $\ve>0$ small enough and
\begin{eqnarray*}
c_{\ve,\Omega,T}\geq\ve^4(B'+C\sum_{i=1}^2e^{-\frac{2(1+\sigma)\sqrt{\lambda_i}(\mathcal{D}_1^\ve+\frac12\delta)}{\ve}})
\end{eqnarray*}
with $\beta>\beta_1$ and $\ve>0$ small enough.
\end{lemma}
\begin{proof}
The proof is similar to that of \cite[Theorem~4.1]{LW05}, so we only sketch it and point out the differences.  By a similar argument as used for Lemma~\ref{lem0006}, we can see that the system~$(\mathcal{S}^{00}_{\ve})$ has a radial ground state solution $\overrightarrow{\widehat{\mathbf{v}}}_{\ve,*}$ in $\mathbb{B}_{R_\ve}$ for $\beta\in(0, \beta_0)\cup(\beta_1, +\infty)$ with $\ve>0$ small enough.  Thanks to Lemma~\ref{lem0004}, $\mathbb{B}_{R_\ve}(p_1^\ve)\to\bbr^4$ as $\ve\to0^+$.  Thus, by a similar argument as used for Proposition~\ref{prop0004}, we can see that $\overrightarrow{\widehat{\mathbf{v}}}_{\ve,*}\to\overrightarrow{\mathbf{v}}_*$ strongly in $\mathcal{H}_{\bbr^4}\cap\mathcal{L}^\infty(\bbr^4)$ as $\ve\to0^+$, where $\overrightarrow{\mathbf{v}}_*$ is the ground state solution of $(\mathcal{S}_*)$.  It follows from the Moser's iteration as used in \cite{BZZ13} and \cite[Theorem~8.17]{GT98} (see also \cite[Lemma~4.3]{BZZ13}) that
\begin{eqnarray*}
\lim_{|x|\to+\infty}\widehat{v}_i^{\ve,*}=0\quad\text{uniformly for }\ve>0\text{ small enough}.
\end{eqnarray*}
Now, since $\lambda_i>0$, by using the maximum principle in a standard way, we can see that for every $\delta>0$ small enough, we have
\begin{eqnarray}\label{eqn0062}
Ce^{-(1+\sigma)\sqrt{\lambda_i}(\frac{\widetilde{\mathcal{D}}_1^\ve}{\ve}-1)}\leq\widehat{v}_i^{\ve,*}\bigg(\frac{\widetilde{\mathcal{D}}_1^\ve}{\ve}-1\bigg)\leq Ce^{-(1-\sigma)\sqrt{\lambda_i}(\frac{\widetilde{\mathcal{D}}_1^\ve}{\ve}-1)}
\end{eqnarray}
with $\ve>0$ small enough.  Let us extend $\overrightarrow{\widehat{\mathbf{v}}}_{\ve,*}$ to the whole space $\bbr^4$ as in the proof of \cite[Theorem~4.1]{LW05} and denote it by $\overrightarrow{\widehat{\mathbf{v}}}_{\ve,**}$, then $\overrightarrow{\widehat{\mathbf{v}}}_{\ve,**}\in\mathcal{H}_{\bbr^4}$.  Since $p>2$, by a similar argument as used in the proof of \cite[Theorem~4.1]{LW05}, we have
\begin{eqnarray}\label{eqn0063}
\mathcal{J}_{\bbr^4,T}(\overrightarrow{\mathbf{t}}\circ\overrightarrow{\widehat{\mathbf{v}}}_{\ve,**})\leq
\mathcal{J}_{\mathbb{B}_{\widetilde{\mathcal{D}}_1^\ve},T}(\overrightarrow{\mathbf{t}}\circ\overrightarrow{\widehat{\mathbf{v}}}_{\ve,*})
-C(\overrightarrow{\mathbf{t}})\sum_{i=1}^2\widehat{v}_i^{\ve,*}\bigg(\frac{\widetilde{\mathcal{D}}_1^\ve}{\ve}-1\bigg)
\end{eqnarray}
for all $\overrightarrow{\mathbf{t}}\in(\bbr^+)^2$ and $C(\overrightarrow{\mathbf{t}})$ is bounded away from $0$ if $t_i$ bounded away from $0$ for all $i=1,2$.  Note that $\overrightarrow{\widehat{\mathbf{v}}}_{\ve,*}\to\overrightarrow{\mathbf{v}}_*$ strongly in $\mathcal{H}_{\bbr^4}$ as $\ve\to0^+$ and $\overrightarrow{\mathbf{v}}_*$ is the ground state solution of $(\mathcal{S}_*)$, thus, by a similar argument as used in the proof of Lemma~\ref{lem0005}, there exists $\overrightarrow{\mathbf{t}}_\ve\in(\bbr^+)^2$ with $\overrightarrow{\mathbf{t}}_\ve\to\overrightarrow{\mathbf{1}}$ as $\ve\to0^+$ such that $\overrightarrow{\mathbf{t}}_\ve\circ\overrightarrow{\widehat{\mathbf{v}}}_{\ve,**}\in\mathcal{N}_{\bbr^4,T}$ for $0<\beta<\beta_0$.  In the case $\beta>\beta_1$, we also have that $t_1^\ve=t_2^\ve$ and $\overrightarrow{\mathbf{t}}_\ve\circ\overrightarrow{\widehat{\mathbf{v}}}_{\ve,**}\in\mathcal{N}'_{\bbr^4,T}$.  Therefore, by \eqref{eqn0063} and a similar argument as used in the proof of Lemma~\ref{lem0001}, we have
\begin{eqnarray*}
B\leq\mathcal{J}_{\mathbb{B}_{\widetilde{\mathcal{D}}_1^\ve},T}(\overrightarrow{\widehat{\mathbf{v}}}_{\ve,*})
-C\sum_{i=1}^2\widehat{v}_i^{\ve,*}\bigg(\frac{\widetilde{\mathcal{D}}_1^\ve}{\ve}-1\bigg)
\end{eqnarray*}
for $0<\beta<\beta_0$ and
\begin{eqnarray*}
B'\leq\mathcal{J}_{\mathbb{B}_{\widetilde{\mathcal{D}}_1^\ve},T}(\overrightarrow{\widehat{\mathbf{v}}}_{\ve,*})
-C\sum_{i=1}^2\widehat{v}_i^{\ve,*}\bigg(\frac{\widetilde{\mathcal{D}}_1^\ve}{\ve}-1\bigg)
\end{eqnarray*}
for $\beta>\beta_1$.  On the other hand, also by a similar argument as used in the proof of Lemma~\ref{lem0001}, we can see that there exists $\overrightarrow{\mathbf{t}}'_\ve\in(\bbr^+)^2$ such that $\overrightarrow{\mathbf{t}}'_\ve\circ\overrightarrow{\widehat{\mathbf{u}}}_{\ve,*}\in\mathcal{N}_{\mathbb{B}_{R_\ve},T}$ for $0<\beta<\beta_0$.  In the case $\beta>\beta_1$, we also have that $(t_1^\ve)'=(t_2^\ve)'$ and $\overrightarrow{\mathbf{t}}'_\ve\circ\overrightarrow{\widehat{\mathbf{u}}}_{\ve,*}\in\mathcal{N}'_{\mathbb{B}_{R_\ve},T}$.  By Proposition~\ref{prop0004}, we can see from the definition of $\overrightarrow{\widehat{\mathbf{u}}}_{\ve,*}$ that $\overrightarrow{\mathbf{t}}'_\ve\to\overrightarrow{\mathbf{t}}_0'$ as $\ve\to0^+$ with $(t_i^0)'\leq1$ for both $i=1,2$.  Therefore, the conclusion follows from , \eqref{eqn0060}--\eqref{eqn0062}.
\end{proof}

Now, we can obtain the following
\begin{proposition}\label{prop0005}
Let $\beta\in(0, \beta_0)\cup(\beta_1, +\infty)$.  Then we have $\mathcal{D}_i^\ve\to\mathcal{D}$ as $\ve\to0^+$ for all $i=1,2$.
\end{proposition}
\begin{proof}
By Lemmas~\ref{lem0007} and \ref{lem0009}, we can see from Proposition~\ref{prop0004} that
\begin{eqnarray*}
C'\sum_{i=1}^2e^{-\frac{2(1+\sigma')\sqrt{\lambda_i}(\mathcal{D}_1^\ve+\frac12\delta)}{\ve}}\leq C\sum_{i=1}^2e^{-\frac{2(1-\sigma)\sqrt{\lambda_i}\mathcal{D}}{\ve}}.
\end{eqnarray*}
Since $\sigma,\sigma'$ and $\delta$ are all arbitrary, we have that $\liminf_{\ve\to0^+}\mathcal{D}_1^\ve\geq\mathcal{D}$, which   implies that $\lim_{\ve\to0^+}\mathcal{D}_1^\ve=\mathcal{D}$.  By Proposition~\ref{prop0004}, we also have that $\lim_{\ve\to0^+}\mathcal{D}_2^\ve=\mathcal{D}$.
\end{proof}

\vskip0.12in

In what follows, let us study the the location of the spikes for $\beta<0$.  We first follow the ideas in \cite{NW95} to establish an upper-bound of $c_{\ve,\Omega,T}$ for $\ve>0$ small enough.  Fix $P\in\Omega$ and let $u_i^{\ve,P}$ be the unique solution of the following equation
\begin{equation*}
\left\{\aligned&-\Delta u+\lambda_iu=\mu_i(v_i^0)^3+\alpha_i(v_i^0)^{p-1}\quad&\text{in }\Omega_{\ve,P},\\
&u>0\quad\text{in }\Omega_{\ve,P},\quad u=0\quad\text{on }\partial\Omega_{\ve,P}.\endaligned\right.\eqno{(\mathcal{P}_{i}^{\ve,P})}
\end{equation*}
where $\Omega_{\ve,P}=\{y\in\bbr^4\mid \ve y+P\in\Omega\}$ and $v_i^0$ is the ground state solution of $(\mathcal{P}_{i})$,  $i=1,2$, which is given by Proposition~\ref{prop0004}.  Since $\Omega_{\ve,P}\to\bbr^4$ as $\ve\to0^+$, it is easy to see that $u_i^{\ve,P}\to v_i^0$ strongly in $\mathcal{H}_{i}$ as $\ve\to0^+$.  Let
\begin{eqnarray}\label{eqn0075}
\psi_i^{\ve,P}(x)=-\ve\log\bigg(\varphi_{i}^{\ve,P}\bigg(\frac{x-P}{\ve}\bigg)\bigg)\quad\text{and}\quad V_i^{\ve,P}=e^{\frac{\psi_i^\ve(P)}{\ve}}\varphi_i^{\ve,P}, \quad i=1,2
\end{eqnarray}
where $\varphi_{i}^{\ve,P}=v_i^0-u_i^{\ve,P}$.  Then as in \cite{NW95}, by the fact that $u_i^{\ve,P}$ and $v_i^0$ respectively satisfy $(\mathcal{P}_{i}^{\ve,P})$ and $(\mathcal{P}_{i})$ for $i=1,2$, we can see that $\psi_i^{\ve,P}$ and $V_i^{\ve,P}$ respectively satisfy
\begin{equation*}
\left\{\aligned&\ve\Delta\psi_i^{\ve,P}-|\nabla\psi_i^{\ve,P}|^2+\lambda_i=0\quad&\text{in }\Omega,\\
&\psi_i^{\ve,P}=-\ve\log\bigg(v_{i}^0\bigg(\frac{x-P}{\ve}\bigg)\bigg)\quad\text{on }\partial\Omega\endaligned\right.\eqno{(\widetilde{\mathcal{P}}_{i}^{\ve,P})}
\end{equation*}
and
\begin{equation*}
\left\{\aligned&\Delta V_i^{\ve,P}-\lambda_iV_i^{\ve,P}=0\quad&\text{in }\Omega_{\ve,P},\\
&V_i^{\ve,P}(0)=1.\endaligned\right.\eqno{(\widehat{\mathcal{P}}_{i}^{\ve,P})}
\end{equation*}
Set $\widetilde{u}_i^{\ve,P}(x)=u_i^{\ve,P}\bigg(\frac{x-P}{\ve}\bigg)$.  Then by $(\mathcal{P}_{i}^{\ve,P})$, we can see that $\widetilde{u}_i^{\ve,P}$ satisfies
\begin{equation*}
\left\{\aligned&-\ve^2\Delta \widetilde{u}_i^{\ve,P}+\lambda_i\widetilde{u}_i^{\ve,P}=\mu_i\bigg(v_{i,P}^{0,\ve}\bigg)^3+\alpha_i\bigg(v_{i,P}^{0,\ve}\bigg)^{p-1}\quad&\text{in }\Omega,\\
&\widetilde{u}_i^{\ve,P}>0\quad\text{in }\Omega,\quad \widetilde{u}_i^{\ve,P}=0\quad\text{on }\partial\Omega,\endaligned\right.\eqno{(\overline{\mathcal{P}}_{i}^{\ve,P})}
\end{equation*}
where $v_{i,P}^{0,\ve}=v_i^0\bigg(\frac{x-P}{\ve}\bigg)$, $i=1,2$.  We also recall that
$$
\text{dist}(0, \partial\Omega)=\mathcal{D}=\max_{p\in\Omega}\text{dist}(p, \partial\Omega).
$$
\begin{lemma}\label{lem0011}
For $\ve$ sufficiently small, we have the following
\begin{enumerate}
\item[$(1)$]  $\|\widetilde{u}_i^{\ve,P}\|_{i,\ve,\Omega}^2=\ve^4\bigg(\|v_i^0\|_{i,\bbr^4}^2-(\gamma_i+o(1))\varphi_{i}^{\ve,P}(0)\bigg)$,
\item[$(2)$]  $\alpha_i\mathcal{B}_{\widetilde{u}_i^{\ve,P},p,\Omega}^{p}=\ve^4\bigg(\alpha_i\mathcal{B}_{v_i^0,p,\Omega}^{p}-(\gamma_{i,p}+o(1))\varphi_{i}^{\ve,P}(0)\bigg)$,
\item[$(3)$]  $\mu_i\mathcal{B}_{\widetilde{u}_i^{\ve,P},4,\Omega}^{4}=\ve^4\bigg(\mu_i\mathcal{B}_{v_i^0,4,\Omega}^{4}-(\gamma_{i,4}+o(1))\varphi_{i}^{\ve,P}(0)\bigg)$,
\end{enumerate}
where $\gamma_i,\gamma_{i,p}$ and $\gamma_{i,4}$ are positive constants  independent of $\ve$.
Moreover, we also have that
\begin{eqnarray*}
d_{i,\ve,\Omega}\leq \ve^4(d_{i,\bbr^4}+(a_0+o(1))\varphi_{i}^{\ve,P}(0)),
\end{eqnarray*}
where $a_0>0$ is a constant independent of $\ve$.
\end{lemma}
\begin{proof}
Since the proof is similar to that of \cite[Lemma~5.2 and Proposition~5.1]{NW95}, we only sketch it and point out the differences.  By the well-known Gidas-Ni-Nirenberg's Theorem \cite{GNN81} (see also in \cite[Theorem~A]{LN01}), we can see that $v_i^0$ are radial symmetric for $i=1,2$, it follows from $p>2$ that
\begin{equation}\label{eqn0072}
\aligned
&\int_{\bbr^4}(\alpha_i (v_i^0)^{p-1}+\mu_i(v_i^0)^3)V_i^{*,0}=\gamma_i>0,\\
&\int_{\bbr^4}\alpha_i p(v_i^0)^{p-1}V_i^{*,0}=\gamma_{i,p}>0,\\
&\int_{\bbr^4}4\mu_i (v_i^0)^{3}V_i^{*,0}=\gamma_{i,4}>0,
\endaligned
\end{equation}
where $V_i^{*,0}$ is the unique positive radial solution of the following equation
\begin{equation*}
\left\{\aligned&\Delta V-\lambda_iV=0\quad&\text{in }\bbr^4,\\
&V(0)=1.\endaligned\right.\eqno{(\widehat{\mathcal{P}}_{i}^{*,0})}
\end{equation*}
Note that all the integrals in the proof of \cite[Lemma~4.7]{NW95} also make sense in our case, thus, by the similar argument with some trivial modifications, we can see that
\begin{eqnarray*}
&\int_{\bbr^4}(\alpha_i (v_i^0)^{p-1}+\mu_i(v_i^0)^3)\widetilde{V}_i=\gamma_i,\\
&\int_{\bbr^4}\alpha_i (v_i^0)^{p-1}\widetilde{V}_i=\gamma_{i,p},\\
&\int_{\bbr^4}\mu_i (v_i^0)^{3}\widetilde{V}_i=\gamma_{i,4},
\end{eqnarray*}
where $\widetilde{V}_i$ is arbitrary solution of $(\widehat{\mathcal{P}}_{i}^{*,0})$.  On the other hand, by \cite[Lemma~5.1]{LW05}, we have $V_i^{\ve,0}\to V_i^{0}$ strongly in $\mathcal{H}_i$ as $\ve\to0^+$ respectively for $i=1,2$ and
\begin{eqnarray*}
\sup_{y\in\Omega_\ve}|e^{-\sqrt{\lambda_i}(1+\sigma)|y|}V_i^{\ve,0}(y)|\leq C
\end{eqnarray*}
for any $0<\sigma<1$ and uniformly for $\ve$ sufficiently small. Since all the integrals in the proofs of \cite[Lemmas~5.2, 5.3 and Proposition~5.1]{NW95} also make sense in our case, by the similar arguments with some trivial modifications and the fact that $(2-p)\alpha_i(v_i^0)^{p}-2\mu_i(v_i^0)^3<0$ in $\bbr^4$ (since  $p>2$), we can obtain the conclusion.
\end{proof}

We also need the following observation.
\begin{lemma}\label{lem0010}
For $\ve>0$ sufficiently small, we have
\begin{eqnarray*}
\int_{\Omega}|\widetilde{u}_1^{\ve,P_1}|^2|\widetilde{u}_2^{\ve,P_2}|^2dx
=\ve^4(1+o(1))\int_{\bbr^4}\bigg|v_1^0\bigg(x-\frac{P_1}{\ve}\bigg)\bigg|^2\bigg|v_2^0\bigg(x-\frac{P_2}{\ve}\bigg)\bigg|^2dx.
\end{eqnarray*}
\end{lemma}
\begin{proof}
With some trivial modifications, the conclusion follows from the similar argument as used in the proof of $(2)$ of \cite[Lemma~5.2]{LW05}.
\end{proof}

Now, as in \cite{LW05}, for every $(P_1, P_2)\in\Omega^2$, we denote
\begin{eqnarray*}
I_\ve[P_1, P_2]=\int_{\bbr^4}\bigg|v_1^0\bigg(x-\frac{P_1}{\ve}\bigg)\bigg|^2\bigg|v_2^0\bigg(x-\frac{P_2}{\ve}\bigg)\bigg|^2dx
\end{eqnarray*}
and
\begin{eqnarray*}\label{eqn0076}
\delta_\ve[P_1, P_2]=\sum_{i=1}^2\varphi_{i}^{\ve,P_i}(0)+I_\ve[P_1, P_2].
\end{eqnarray*}
\begin{lemma}\label{lem0012}
For $\ve>0$ sufficiently small, we have
\begin{eqnarray*}
C'e^{-\frac{2(1+\sigma)\varphi(P_1,P_2)}{\ve}}\leq\delta_\ve[P_1, P_2]\leq C e^{-\frac{2(1-\sigma)\varphi(P_1,P_2)}{\ve}}
\end{eqnarray*}
for $\sigma>0$ small enough, where
\begin{eqnarray}\label{eqn9960}
\varphi(P_1,P_2)=\min_{i=1,2}\{\min\{\sqrt{\lambda_i}|P_1-P_2|, \sqrt{\lambda_i}dist(P_i, \partial\Omega)\}\}.
\end{eqnarray}
\end{lemma}
\begin{proof}
By \cite[Lemma~4.4]{NW95}, we have
\begin{eqnarray}\label{eqn0069}
C'e^{-\frac{2(1+\sigma)\sqrt{\lambda_i}dist(P_i, \partial\Omega)}{\ve}}\leq\varphi_{i}^{\ve,P_i}(0)\leq Ce^{-\frac{2(1-\sigma)\sqrt{\lambda_i}dist(P_i, \partial\Omega)}{\ve}}
\end{eqnarray}
for any $0<\sigma<1$.  On the other hand, by scaling technique and \cite[Proposition~2.1]{BZZ13}, for any $0<\delta<1$, there exists $C_\delta>0$ such that
\begin{eqnarray}\label{eqn0067}
v_i^0\leq C_\delta e^{-(1-\delta)\sqrt{\lambda_i}|x|}.
\end{eqnarray}
Moreover, since $v_i^0$ satisfies $(\mathcal{P}_{i})$, by \eqref{eqn0067}, we can apply the maximum principle in a standard way to show that for any $0<\delta<1$, there exists $C_\delta'>0$ such that
\begin{eqnarray}\label{eqn0068}
v_i^0\geq C_\delta' e^{-(1+\delta)\sqrt{\lambda_i}|x|}
\end{eqnarray}
for $|x|$ large enough.
Now, by the translation  and \eqref{eqn0067} and \eqref{eqn0068}, we have  that
\begin{eqnarray}\label{eqn0070}
C'e^{-\frac{2(1+\delta)\min_{i=1,2}\{\sqrt{\lambda_i}|P_1-P_2|\}}{\ve}}\leq I_\ve[P_1, P_2]\leq Ce^{-\frac{2(1-\delta)\min_{i=1,2}\{\sqrt{\lambda_i}|P_1-P_2|\}}{\ve}}.
\end{eqnarray}
The conclusion follows immediately from \eqref{eqn0069} and \eqref{eqn0070}.
\end{proof}

Now, we can obtain an upper bound for $c_{\ve,\Omega,T}$ in the case $\beta<0$.
\begin{lemma}\label{lem0013}
Let $\alpha_T$ be given in Lemma~\ref{lem0002} and $\ve>0$ be small enough.  Then for
\begin{enumerate}
\item $-\sqrt{\mu_1\mu_2}<\beta<0$ and $\alpha_1,\alpha_2>0$,
\item $\beta\leq-\sqrt{\mu_1\mu_2}$ and $\alpha_1,\alpha_2>0$ with $|\overrightarrow{\mathbf{\alpha}}|<\alpha_T$,
\end{enumerate}
we have
\begin{eqnarray*}
c_{\ve,\Omega,T}\leq\ve^4(\sum_{i=1}^2d_{i,\bbr^4}+a_1e^{-\frac{2(1-\sigma)\varphi(P_1^*,P_2^*)}{\ve}}),
\end{eqnarray*}
where $P_i^*$ satisfy $\varphi(P_1^*,P_2^*)=\max_{(P_1,P_2)\in\Omega^2}\varphi(P_1,P_2)$.
\end{lemma}
\begin{proof}
Clearly $\varphi(P_1^*,P_2^*)>0$ and $P_1^*\not=P_2^*$.  Now, consider $\overrightarrow{\widetilde{\mathbf{u}}}_{\ve,\overrightarrow{\mathbf{P}}_*}=(\widetilde{u}_1^{\ve,P_1^*}, \widetilde{u}_2^{\ve,P_2^*})$.  Then we must have $\overrightarrow{\widetilde{\mathbf{u}}}_{\ve,\overrightarrow{\mathbf{P}}_*}\in\widetilde{\mathcal{H}}_{\ve,\Omega}$.  By Lemma~\ref{lem0010} and \eqref{eqn0070}, we also have that
\begin{eqnarray}\label{eqn0071}
\mathcal{B}_{|\widetilde{u}_1^{\ve,P_1^*}|^2|\widetilde{u}_2^{\ve,P_2^*}|^2,1,\Omega}=o(1).
\end{eqnarray}
By  \eqref{eqn0008} and the fact that $u_i^{\ve,P_i^*}\to v_i^0$ strongly in $\mathcal{H}_{i,\bbr^4}$ as $\ve\to0^+$, we can apply Miranda's theorem similar to that for $\eta_i(t_1,t_2)$ in the proof of Lemma~\ref{lem0003} to show that there exists $\overrightarrow{\widetilde{\mathbf{t}}}_{\ve}\in(\bbr^+)^2$ with $0<\widetilde{t}_i^\ve<t_i^{\ve,*}$ such that $\overrightarrow{\widetilde{\mathbf{t}}}_{\ve}\circ\overrightarrow{\widetilde{\mathbf{u}}}_{\ve,\overrightarrow{\mathbf{P}}_*}\in \mathcal{N}_{\ve,\Omega,T}$, where
\begin{eqnarray*}
t_i^{\ve,*}=\bigg(\frac{\|\widetilde{u}_i^{\ve,P_i^*}\|_{i,\ve,\Omega}^2}{\mu_i\mathcal{B}_{\widetilde{u}_i^{\ve,P_i^*},4,\Omega}^{4}}\bigg)^{\frac12}
\end{eqnarray*}
are bounded from above by a constant $C$ independent of $\ve$.  By \eqref{eqn0071} and the fact that $u_i^{\ve,P_i^*}\to v_i^0$ strongly in $\mathcal{H}_{i,\bbr^4}$ as $\ve\to0^+$, we also have that $$\prod_{i=1}^2\mu_i\mathcal{B}_{\widetilde{u}_i^{\ve,P_i^*},4,\Omega}^{4}-\beta^2\mathcal{B}_{|\widetilde{u}_1^{\ve,P_1^*}|^2|\widetilde{u}_2^{\ve,P_2^*}|^2,1,\Omega}\geq \ve^4C; \quad \mathcal{B}_{\widetilde{u}_i^{\ve,P_i^*},4,\Omega}^{4}\geq \ve^4C.$$
Thus, by applying the implicit function theorem as used for $\Gamma_i^n(\overrightarrow{\mathbf{t}},\tau)$ with $\beta<0$ in the proof of Lemma~\ref{lem0001}, we can also see that $\overrightarrow{\widetilde{\mathbf{t}}}_{\ve}$ is unique.
Now, let us consider the system $\overrightarrow{\mathbf{\Gamma}}(\overrightarrow{\mathbf{t}}, \tau)=\overrightarrow{\mathbf{0}}$, where $\overrightarrow{\mathbf{\Gamma}}(\overrightarrow{\mathbf{t}}, \tau)=(\Gamma_1(\overrightarrow{\mathbf{t}},\tau), \Gamma_2(\overrightarrow{\mathbf{t}},\tau))$ with
\begin{eqnarray*}
\Gamma_i(\overrightarrow{\mathbf{t}},\tau)&=&\tau^{-4}(\|t_i\widetilde{u}_i^{\tau,P_i^*}\|_{i,\tau,\Omega}^2-\alpha_i\mathcal{B}_{t_i\widetilde{u}_i^{\tau,P_i^*},p,\Omega}^{p}
-\mu_i\mathcal{B}_{t_i\widetilde{u}_i^{\tau,P_i^*},4,\Omega}^{4}\\
&&-\beta\mathcal{B}_{|t_1\widetilde{u}_1^{\tau,P_1^*}|^2|t_2\widetilde{u}_2^{\tau,P_2^*}|^2,1,\Omega}).
\end{eqnarray*}
Clearly, $\overrightarrow{\mathbf{\Gamma}}$ is of $C^1$ and $\overrightarrow{\mathbf{\Gamma}}(\overrightarrow{\widetilde{\mathbf{t}}}_{\ve}, \ve)=\overrightarrow{\mathbf{0}}$.  Moreover, as stated in the proof of Lemma~\ref{lem0011}, all the integrals in the proof of \cite[Lemmas~5.2]{NW95} also make sense in our case, thus, by the similar arguments with some trivial modifications, we can see from Lemma~\ref{lem0010} that
\begin{eqnarray*}
\Gamma_i(\overrightarrow{\mathbf{t}},\tau)&=&\|t_iv_i^0\|_{i,\bbr^4}^2-\alpha_i\mathcal{B}_{t_iv_i^0,p,\bbr^4}^{p}-\mu_i\mathcal{B}_{t_iv_i^0,4,\bbr^4}^{4}\\
&&-(\gamma_it_i^2-t_i\gamma_i(t_i)+o(1))\varphi_{i}^{\tau,P_i^*}(0))-\beta t_1^2t_2^2(1+o(1))I_\tau[P_1^*, P_2^*],
\end{eqnarray*}
where $o(1)$ is uniformly   bounded and  $t_i$, $\gamma_i$ are given by \eqref{eqn0072} and
\begin{eqnarray*}
\gamma_i(t_i)&=&\int_{\bbr^4}(\alpha_ip (t_iv_i^0)^{p-1}+4\mu_i(t_iv_i^0)^3)V_i^{*,0}.
\end{eqnarray*}
Since $p>2$, it is easy to see that $\gamma_it_i^2-t_i\gamma_i(t_i)<0$ for $t_i\geq1$.  Thus, by $\beta<0$, we can extend the functions $\Gamma_i(\overrightarrow{\mathbf{t}},\tau)$ to $\tau=0$ as continuous differentiable maps  for $t_i\geq1$.  Since $p>2$, by applying the implicit function theorem, we can see that
\begin{eqnarray}\label{eqn0073}
t_i^\ve=1+O(\delta_\ve[P_1^*, P_2^*]), \;\;\; i=1,2.
\end{eqnarray}
Now, by similar computations as  in \cite[(5.15)]{LW05} and the Taylor's expansion, we see that
\begin{eqnarray}
c_{\ve,\Omega,T}&\leq&\mathcal{J}_{\ve,\Omega,T}(\overrightarrow{\widetilde{\mathbf{t}}}_{\ve}\circ\overrightarrow{\widetilde{\mathbf{u}}}_{\ve,\overrightarrow{\mathbf{P}}_*})\notag\\
&=&\mathcal{J}_{\ve,\Omega,T}(\overrightarrow{\widetilde{\mathbf{u}}}_{\ve,\overrightarrow{\mathbf{P}}_*})
+\mathcal{J}_{\ve,\Omega,T}'(\overrightarrow{\widetilde{\mathbf{u}}}_{\ve,\overrightarrow{\mathbf{P}}_*})
\bigg((\overrightarrow{\widetilde{\mathbf{t}}}_{\ve}-\overrightarrow{\mathbf{1}})\circ\overrightarrow{\widetilde{\mathbf{u}}}_{\ve,\overrightarrow{\mathbf{P}}_*}\bigg)\notag\\
&&+O(|\overrightarrow{\widetilde{\mathbf{t}}}_{\ve}-\overrightarrow{\mathbf{1}}|^2).\label{eqn0074}
\end{eqnarray}
Thanks to Lemmas~\ref{lem0011}-\ref{lem0012}, \eqref{eqn0073} and \eqref{eqn0074}, we obtain the conclusion.
\end{proof}


\vskip0.3in

 {Due to the criticalness,  we do not know whether the solution of $(\mathcal{P}_i)$ in $\bbr^4$ is non-degeneracy,}  we have to  establish the low-bound of $c_{\ve,T,\Omega}$ for $\beta<0$ in a way different from that in \cite{LW05} which is inspired by \cite{DSW11}.  Consider the following equations
$$
\left\{\aligned&-\ve^2\Delta u+\lambda_iu=\mu_iu^3+\alpha_iu^{p-1}\quad&\text{in }\mathbb{B}_{r_i}(P_i),\\
&u>0\quad\text{in }\mathbb{B}_{r_i}(P_i),\quad&u=0\quad\text{on }\partial\mathbb{B}_{r_i}(P_i),\endaligned\right.\eqno{(\mathcal{P}_{i}^{\ve,P_i})}
$$
where $P_i\in\Omega, i=1,2, $ are two different points.  By Proposition~\ref{propA0001}, we can see that $(\mathcal{P}_{i}^{\ve,P_i})$ has a solution $u_{i,P_i}^\ve$ satisfying the following
\begin{eqnarray*}
\mathcal{E}_{i,\ve,\mathbb{B}_{r_i}(P_i)}(u_{i,P_i}^\ve)=d_{i,\ve,\mathbb{B}_{r_i}(P_i)}
\end{eqnarray*}
and
\begin{eqnarray}\label{eqn0082}
u_{i,P_i}^\ve(\ve y+P_i)\to v_i^0\text{ strongly in }\mathcal{H}_{i,\bbr^4},
\end{eqnarray}
where $v_i^0$ is a ground state solution of $(\mathcal{P}_i)$ and $\mathcal{E}_{i,\ve,\mathbb{B}_{r_i}(P_i)}(u), d_{i,\ve,\mathbb{B}_{r_i}(P_i)}$ are given by \eqref{eqn0065}.  Moreover,   similar to the proof of Lemmas~\ref{lem0007}-\ref{lem0009}, we also have
\begin{eqnarray*}
d_{i,\bbr^4}+Ce^{-\frac{2(1+\sigma)\sqrt{\lambda_i}r_i}{\ve}}\leq \frac{d_{i,\ve,\mathbb{B}_{r_i}(P_i)}}{\ve^4}\leq d_{i,\bbr^4}+Ce^{-\frac{2(1-\sigma)\sqrt{\lambda_i}r_i}{\ve}}.
\end{eqnarray*}
where $d_{i,\bbr^4}$ is given by \eqref{eqn0066} and $\sigma\in(0, 1)$.  For every $b_1,b_2>0$, we also define
\begin{equation}\label{eqn5453}
\aligned
&\varphi_{b_1,b_2}^*(p_1^\ve,p_2^\ve)\\
&=\min\bigg\{\frac{\sqrt{\lambda_1}b_1+\sqrt{\lambda_2}b_2}{b_1+b_2}|p_1^\ve-p_2^\ve|,
\sqrt{\lambda_1}dist(p_1^\ve, \partial\Omega), \sqrt{\lambda_2}dist(p_2^\ve, \partial\Omega)\bigg\},
\endaligned
\end{equation}
where $p_i^\ve$ is the maximum point of $u_i^\ve$  and $\overrightarrow{\mathbf{u}}_\ve=(u_1^\ve, u_2^\ve)$ is the positive solution obtained  by Proposition~\ref{prop0001}.

\vskip0.23in

\begin{lemma}\label{lemn0016}
Let $\alpha_T$ be  given in Lemma~\ref{lem0002} and $\ve>0$ be small enough.  Then for $\beta<0$ and $|\overrightarrow{\mathbf{\alpha}}|<\alpha_T$, we have
\begin{eqnarray*}
c_{\ve,\Omega,T}\geq\ve^4(\sum_{i=1}^2d_{i,\bbr^4}+Ce^{-\frac{2(1+\sigma)(\varphi_{b_1,b_2}^*(p_1^\ve,p_2^\ve)-\delta)}{\ve}}),
\end{eqnarray*}
where $O(\delta)\to0$ as $\delta\to0$ uniformly for  $\ve$.
\end{lemma}
\begin{proof}
Since $p>2$, it is standard to show that there exists $t_i^\ve>0$ such that $t_i^\ve u_i^\ve\in\mathcal{M}_{i,\ve,\Omega}$, where $\mathcal{M}_{i,\ve,\Omega}$ is given by \eqref{eqn0087}.  Moreover, by Proposition~\ref{prop0004}, we also have $t_i^\ve=1+o(1)$.  Thus, by $|\overrightarrow{\mathbf{\alpha}}|<\alpha_T$ and $\beta<0$,   Lemma~\ref{lem0001}, we see  that
\begin{eqnarray}\label{eqn0088}
c_{\ve,\Omega,T}\geq\mathcal{J}_{\ve,\Omega,T}(\overrightarrow{\mathbf{t}}_\ve\circ\overrightarrow{\mathbf{u}}_\ve)
\geq\sum_{i=1}^2\mathcal{E}_{i,\ve,\Omega}(t_i^\ve u_i^\ve)-\frac{\beta}{4}\mathcal{B}_{|t_1^\ve u_1^\ve|^2|t_2^\ve u_2^\ve|^2,1,\Omega}
\end{eqnarray}
for $\ve>0$ small enough.
By  Lemma~\ref{lem0008},   similar to the proof of Lemma~\ref{lem0009}, we have
\begin{eqnarray}\label{eqn9987}
d_{i,\ve,\Omega}\geq \ve^4(d_{i,\bbr^4}+Ce^{-\frac{2(1+\sigma)\sqrt{\lambda_i}dist(p_i^\ve, \partial\Omega)}{\ve}}),\;\;\; i=1, 2.
\end{eqnarray}
On the other hand, let $p^\ve$ be the point in $\Lambda_{b_1,b_2}(p_1^\ve,p_2^\ve)$ such that
$$
|p_1^\ve-p^\ve|+|p_2^\ve-p^\ve|=|p_1^\ve-p_2^\ve|,
$$
where $\Lambda_{b_1,b_2}(p_1^\ve,p_2^\ve)$ is defined by
\begin{eqnarray}\label{eqn5454}
\Lambda_{b_1,b_2}(p_1^\ve,p_2^\ve)=\{x\in\Omega\mid b_1|x-p_1^\ve|=b_2|x-p_2^\ve|\}.
\end{eqnarray}
We re-denote $|p_i^\ve-p^\ve|$ by $D_i^{\ve,*}$.  For the sake of clarity, we divide the following proof into two cases.

{\bf Case.~1}\quad We have either $\varphi_{b_1,b_2}^*(p_1^\ve,p_2^\ve)=\sqrt{\lambda_1}dist(p_1^\ve, \partial\Omega)$ or $\varphi_{b_1,b_2}^*(p_1^\ve,p_2^\ve)=\sqrt{\lambda_2}dist(p_2^\ve, \partial\Omega)$.

Indeed, since $\beta<0$, by \eqref{eqn0088} and \eqref{eqn9987}, we can see that
\begin{eqnarray*}
c_{\ve,\Omega,T}&\geq&\ve^4(\sum_{i=1}^2(d_{i,\bbr^4}+Ce^{-\frac{2(1+\sigma)\sqrt{\lambda_1}dist(p_1^\ve, \partial\Omega)}{\ve}}))\\
&=&\ve^4(\sum_{i=1}^2d_{i,\bbr^4}+Ce^{-\frac{2(1+\sigma)\varphi_{b_1,b_2}^*(p_1^\ve,p_2^\ve)}{\ve}}).
\end{eqnarray*}

{\bf Case.~2}\quad  We have $\varphi_{b_1,b_2}^*(p_1^\ve,p_2^\ve)=\frac{\sqrt{\lambda_1}b_1+\sqrt{\lambda_2}b_2}{b_1+b_2}|p_1^\ve-p_2^\ve|$.

By Proposition~\ref{prop0004}, for every $\delta>0$ we have $\mathbb{B}_{D_i^{\ve,*}+\ve\delta}(p_i^\ve)\subset\Omega$ for $\ve>0$ small enough.  Now, let us consider the intersection $\Gamma_\ve=\mathbb{B}_{D_1^{\ve,*}+\ve\delta}(p_1^\ve)\cap\mathbb{B}_{D_2^{\ve,*}+\ve\delta}(p_2^\ve)$.  Then $\Gamma_\ve\subset\mathbb{B}_{D_i^{\ve,*}+\ve\delta}(p_i^\ve)\backslash\mathbb{B}_{D_i^{\ve,*}-\ve\delta}(p_i^\ve)$.  Note that $D_i^{\ve,*}=\frac{b_i}{b_1+b_2}|p_1^\ve-p_2^\ve|$, by Proposition~\ref{prop0004} and a standard comparison argument, we can see that $u_i^\ve\geq Ce^{-\frac{(1+\sigma)\sqrt{\lambda_i}|x-p_i^\ve|}{\ve}}$ in $\Gamma_\ve$ with $\ve>0$ small enough.  It follows that
\begin{eqnarray*}
\mathcal{B}_{|u_1^\ve|^2|u_2^\ve|^2,1,\Omega}\geq Ce^{-\frac{2(1+\sigma)}{\ve}(\sqrt{\lambda_1}(D_1^{\ve,*}-\delta)+\sqrt{\lambda_2}(D_2^{\ve,*}-\delta))}.
\end{eqnarray*}
This together with \eqref{eqn0088}--\eqref{eqn9987} and $\beta<0$, implies
\begin{eqnarray*}
c_{\ve,\Omega,T}&\geq&\ve^4(\sum_{i=1}^2(d_{i,\bbr^4}+Ce^{-\frac{2(1+\sigma)\sqrt{\lambda_1}dist(p_1^\ve, \partial\Omega)}{\ve}})\\
&&+Ce^{-\frac{2(1+\sigma)}{\ve}(\sqrt{\lambda_1}(D_1^{\ve,*}-\delta)+\sqrt{\lambda_2}(D_2^{\ve,*}-\delta))})\\
&=&\ve^4(\sum_{i=1}^2d_{i,\bbr^4}+Ce^{-\frac{2(1+\sigma)(\varphi_{b_1,b_2}^*(p_1^\ve,p_2^\ve)-\delta)}{\ve}}).
\end{eqnarray*}
It completes the proof.
\end{proof}

Now, we can obtain the following
\begin{proposition}\label{prop0006}
Let $\beta<0$.  Then we have $\varphi(p_1^\ve,p_2^\ve)\to\varphi(P_1^*, P_2^*)$ as $\ve\to0^+$, where $\varphi(P_1, P_2)$ is given by \eqref{eqn9960}.
\end{proposition}
\begin{proof}
By Lemmas~\ref{lem0013} and \ref{lemn0016}, we can see from Proposition~\ref{prop0004} that
\begin{eqnarray*}
\lim_{\ve\to0^+}\varphi_{b_1,b_2}^*(p_1^\ve,p_2^\ve)\geq\frac{(1-\sigma)(\varphi(P_1^*,P_2^*)-\delta)}{1+\sigma'}.
\end{eqnarray*}
Since $\sigma,\sigma'$ and $\delta$ are all arbitrary, we must have that
\begin{eqnarray*}
\lim_{\ve\to0^+}\varphi_{b_1,b_2}^*(p_1^\ve,p_2^\ve)\geq\varphi(P_1^*,P_2^*).
\end{eqnarray*}
Now, the conclusion follows from letting $b_1\to0^+$ or $b_2\to0^+$.
\end{proof}

\noindent\textbf{Proof of Theorem~\ref{thm0003}:}\quad  It follows immediately from Propositions~\ref{prop0005} and \ref{prop0006}.
\hfill$\Box$


\vskip0.36in

\section{Appendix}
In this section, we will list some known results which will be used frequently in this paper.  Let $\mathcal{H}_{i,\bbr^4}$ be the Hilbert space $H^1(\bbr^4)$ equipped with the inner product
$$
\langle u,v\rangle_{i,\bbr^4}=\int_{\bbr^4}\nabla u\nabla v+\lambda_i uv dx.
$$
The corresponding norm is  given by $\|u\|_{i,\bbr^4}=\langle u,u\rangle_{i,\bbr^4}^{\frac12}$.
Define
\begin{eqnarray}\label{eqn0095}
\mathcal{E}_{i,\bbr^4}(u)=\frac{1}{2}\|u\|_{i,\bbr^4}^2-\frac{\alpha_i}{p}\mathcal{B}_{u_,p, \bbr^4}^p-\frac{\mu_i}{4}\mathcal{B}_{u_,4, \bbr^4}^4.
\end{eqnarray}
Then it is well known that $\mathcal{E}_{i,\bbr^4}(u)$ is of $C^2$ in $\mathcal{H}_{i,\bbr^4}$.
Set
\begin{eqnarray}\label{eqnew8006}
\mathcal{M}_{i,\bbr^4}=\{u\in\mathcal{H}_{i,\bbr^4}\backslash\{0\}\mid \mathcal{E}_{i,\bbr^4}'(u)u=0\}.
\end{eqnarray}
and define
\begin{eqnarray}\label{eqn0066}
d_{i,\bbr^4}=\inf_{\mathcal{M}_{i,\bbr^4}}\mathcal{E}_{i,\bbr^4}(u).
\end{eqnarray}
\begin{proposition}\label{propAnew0002}
There holds $0< d_{i,\bbr^4}<\frac{1}{4\mu_i}\mathcal{S}^2$ for both $i=1,2$, where $\mathcal{S}$ is best embedding constant from $H^1(\bbr^4)\to L^4(\bbr^4)$ defined by
\begin{equation*}
\mathcal{S}=\inf\{\mathcal{B}_{\nabla u,2, \bbr^4}^2\mid u\in H^1(\bbr^4), \mathcal{B}_{u_,4, \bbr^4}^2=1\}.
\end{equation*}
Moreover, $d_{i,\bbr^4}$ can be attained by some $U_{i,\bbr^4}\in\mathcal{M}_{i,\bbr^4}$, which is also a solution of the following equation
$$
\left\{\aligned&-\Delta u+\lambda_iu=\mu_iu^3+\alpha_iu^{p-1}\quad&\text{in }\bbr^4,\\
&u>0\text{ in }\bbr^4,\quad&u\to0\text{ as }|x|\to+\infty.\endaligned\right.\eqno{(\mathcal{P}_{i})}
$$
\end{proposition}
\begin{proof}
See the results in \cite{ZZ12}.
\end{proof}

Let $\Omega\subset\bbr^4$ be a bounded domain, $\lambda_i,\mu_i,\alpha_i>0(i=1,2)$ are constants, $\ve>0$ is a small parameter and $2<p<2^*=4$.  Let
\begin{eqnarray}\label{eqn0015}
\mathcal{E}_{i,\ve,\Omega}(u)=\frac12\|u\|_{i,\ve,\Omega}^2-\frac{\alpha_i}{p}\mathcal{B}_{u_,p, \Omega}^p-\frac{\mu_i}{4}\mathcal{B}_{u_,4, \Omega}^4.
\end{eqnarray}
Then it is well known that $\mathcal{E}_{i,\ve,\Omega}(u)$ is of $C^2$ in $\mathcal{H}_{i,\ve,\Omega}$.
Set
\begin{eqnarray}\label{eqn0087}
\mathcal{M}_{i,\ve,\Omega}=\{u\in\mathcal{H}_{i,\ve,\Omega}\backslash\{0\}\mid \mathcal{E}_{i,\ve,\Omega}'(u)u=0\}
\end{eqnarray}
and define
\begin{eqnarray}\label{eqn0065}
d_{i,\ve,\Omega}=\inf_{\mathcal{M}_{i,\ve,\Omega}}\mathcal{E}_{i,\ve,\Omega}(u),\quad i=1, 2.
\end{eqnarray}

\begin{proposition}\label{propA0001}
Let $\ve>0$ be small enough.  Then $C'\leq d_{i,\ve,\Omega}\leq\frac{\ve^4}{4\mu_i}\mathcal{S}^2-\ve^4C$ for both $i=1,2$.  Moreover, there exists $\widetilde{U}_{i,\ve}\in\mathcal{M}_{i,\ve,\Omega}$ such that $\mathcal{E}_{i,\ve,\Omega}(\widetilde{U}_{i,\ve})=d_{i,\ve,\Omega}$, which is also a solution of the following equation
$$
\left\{\aligned&-\ve^2\Delta u+\lambda_iu=\mu_iu^3+\alpha_iu^{p-1}\quad&\text{in }\Omega,\\
&u>0\quad\text{in }\Omega,\quad&u=0\quad\text{on }\partial\Omega,\endaligned\right.\eqno{(\mathcal{P}_{i,\ve})}
$$
$i=1,2$,
Moreover, $\widetilde{U}_{i,\ve}(\ve y+q_i^\ve)\to v_i^0$ strongly in $\mathcal{H}_{i,\bbr^4}$ as $\ve\to0^+$ and $\{\widetilde{U}_{i,\ve}\}$ is uniformly bounded in $L^\infty(\Omega)$, where $q_i^\ve$ is the maximum point of $\widetilde{U}_{i,\ve}$ and $v_i^0$ is a ground state solution of $(\mathcal{P}_{i})$.
\end{proposition}
\begin{proof}
By the results in \cite{BZZ13}, we have $C'\leq d_{i,\ve,\Omega}\leq\frac{\ve^4}{4\mu_i}\mathcal{S}^2-\ve^4C$ for both $i=1,2$.  For the remaining results, we believe that they exist but we can not find the references around, thus, we will sketch the proofs of them here.  We only give the proof for $(\mathcal{P}_{1,\ve})$ since that of $(\mathcal{P}_{2,\ve})$ is similar.  By the result in \cite{BZZ13} once more, we can see that $d_{1,\ve,\Omega}<\frac{\ve^4}{4\mu_1}\mathcal{S}^2-\ve^4C$ for $\ve>0$ small enough.  Since $\mathcal{M}_{i,\ve,\Omega}$ is a natural constraint due to $p>2$, by applying the concentration-compactness principle in a standard way, we can show that $(\mathcal{P}_{i,\ve})$ has a ground state solution $\widetilde{U}_{i,\ve}$ satisfying $\widetilde{U}_{i,\ve}\in\mathcal{M}_{i,\ve,\Omega}$ and $\mathcal{E}_{i,\ve,\Omega}(\widetilde{U}_{i,\ve})=d_{i,\ve,\Omega}$.  Now, let us consider the functions $\widetilde{U}_{i,\ve}(\ve y+q_i^\ve)$.  By a similar argument as used for Proposition~\ref{prop0003}, we can show that the only solution of the following equation is $u\equiv0$, where the equation is
$$
\left\{\aligned&-\Delta u+\lambda_1u=\mu_1u^3+\alpha_1u^{p-1}\quad&\text{in }\bbr^4_+,\\
&u\geq0\quad\text{in }\bbr^4_+,\quad&u=0\quad\text{on }\partial\bbr^4_+.\endaligned\right.
$$
Thus, if $\Omega^*_{i,\ve}\to\bbr^4_+$ as $\ve\to0^+$, then by a similar argument as used for Lemma~\ref{lem0004}, we can obtain a contradiction due to $d_{1,\ve,\Omega}<\frac{\ve^4}{4\mu_1}\mathcal{S}^2-\ve^4C$ for $\ve>0$ small enough.  Hence, we must have $\Omega^*_{i,\ve}\to\bbr^4$ as $\ve\to0^+$.  Now, by the result in \cite{BZZ13} and a similar argument as used in the Case.~1 of the proof to Proposition~\ref{prop0004}, we can show that $\widetilde{U}_{i,\ve}(\ve y+q_i^\ve)\to v_i^0$ strongly in $\mathcal{H}_{i,\bbr^4}$ as $\ve\to0^+$, where $v_i^0$ is a ground state solution of $(\mathcal{P}_{i})$.  The uniformly boundedness of $\{\widetilde{U}_{i,\ve}\}$ in $L^\infty(\Omega)$ can be obtained by standard elliptic estimates (cf. \cite{BZZ13,CZ121}).
\end{proof}

Let
\begin{eqnarray}\label{eqnew8001}
\mathcal{I}_{\ve}(\overrightarrow{\mathbf{u}})=\sum_{i=1}^2(\frac{\ve^2}{2}\mathcal{B}_{\nabla u_i,2,\bbr^4}^{2}-\frac{\mu_i}{4}\mathcal{B}_{u_i,4,\bbr^4}^{4})-\frac{\beta}{2}\mathcal{B}_{|u_1|^2|u_2|^2,1,\bbr^4}.
\end{eqnarray}
Then $\mathcal{I}_{\ve}$ is of $C^2$ in $\mathcal{D}=D^{1,2}(\bbr^4)\times D^{1,2}(\bbr^4)$.  Set $A_\ve=\inf_{\mathcal{V}_{\ve}}\mathcal{I}_{\ve}(\overrightarrow{\mathbf{u}})$, where
\begin{eqnarray}\label{eqn0034}
\mathcal{V}_{\ve}=\{\overrightarrow{\mathbf{u}}\in\widetilde{\mathcal{D}}\mid \mathcal{I}'_{\ve}(\overrightarrow{\mathbf{u}})\overrightarrow{\mathbf{u}}_1=\mathcal{I}'_{\ve}(\overrightarrow{\mathbf{u}})\overrightarrow{\mathbf{u}}_2=0\}
\end{eqnarray}
with $\widetilde{\mathcal{D}}=(D^{1,2}(\bbr^4)\backslash\{0\})\times(D^{1,2}(\bbr^4)\backslash\{0\})$.
\begin{proposition}\label{propAnew0001}
There holds $A_\ve=\ve^4A_1$.  Moreover, $A_1=\frac{1}{4\mu_1}\mathcal{S}^2+\frac{1}{4\mu_2}\mathcal{S}^2$ for $\beta<0$ and $A_1=\frac{k_1+k_2}{4}\mathcal{S}^2$ for $0<\beta<\min\{\mu_1,\mu_2\}$ or $\beta>\max\{\mu_1,\mu_2\}$ with $k_1,k_2$ satisfying
\begin{equation}\label{eqnew0003}
\left\{\aligned\mu_1k_1+\beta k_2&=1,\\
\mu_2k_2+\beta k_1&=1.\endaligned\right.
\end{equation}
\end{proposition}
\begin{proof}
See the results in \cite{CZ121}.
\end{proof}


\begin{thebibliography}{999}
\bibitem{AA99}
N. Akhmediev, A. Ankiewicz, Partially coherent solitons on a finite background,  {\it  Phys. Rev. Lett., }  {\bf82}(1999), 2661-2664.

\bibitem{AFP09}
B. Abdellaoui, V. Felli, I. Peral, Some remarks on systems of elliptic
equations doubly critical the whole $\bbr^N$, {\it  Calc. Var. PDEs, } {\bf34}(2009), 97-137.

\bibitem{BDW10}
T. Bartsch, N. Dancer, Z.-Q. Wang, A Liouville theorem, a-priori bounds, and bifurcating branches of positive solutions for a nonlinear elliptic system,  {\it  Calc. Var. PDEs, }  {\bf37}(2010),  345-361.

\bibitem{BJM09}
J. Byeon, L. Jeanjean, M. Maris, Symmetric and monotonicity of least energy solutions, {\it Calc.
Var. PDEs,} {\bf36}(2009) 481--492.

\bibitem{B10}
J. Byeon, Singularly rerturbed nonlinear Dirichlet problems with a general nonlinearity, {\it Trans. Amer. Math. Soc.,} {\bf362}(2010), 1981--2001.

\bibitem{B15}
J. Byeon, Semi-classical standing waves for nonlinear Schr\"odinger systems, {\it Calc. Var. PDEs,} {\bf54}(2015), 2287--2340.

\bibitem{BZZ13}
J. Byeon, J. Zhang, W. Zou, Singularly perturbed nonlinear Dirichlet problems involving critical growth, {\it Calc. Var. PDEs,} {\bf47}(2013), 65--85.

\bibitem{CTV05} M. Conti, S. Terracini, G. Verzini, Asymptotic estimates for the spatial segregation of competitive systems,
{\it Adv. Math.,} {\bf 195} (2005), 524-560.

\bibitem{CZ121}
Z. Chen, W. Zou, Positive least energy solutions and phase separation for coupled Schr\"odinger equations with critical exponent,   {\it  Arch. Rational Mech. Anal., }  {\bf205}(2012), 515-551.

\bibitem{CZ131}
Z. Chen, W. Zou, Positive least energy solutions and phase separation for coupled Schr\"odinger equations with critical exponent: higher dimensional case, {\it  Calc. Var. PDEs,} {\bf 52} (2015), 423--467.

\bibitem{CLZ141}
Z. Chen, C.-S. Lin, W. Zou,   Sign-changing solutions and phase separation for an
elliptic system with critical exponent,    {\it Comm. PDEs, } {\bf 39}(2014), 1827-1859.

\bibitem{CZ15}
Z. Chen, W. Zou, Existence and symmetry of positive ground states for a doubly critical Schr\"odinger system, {\it Trans. Amer. Math. Soc.,} {\bf367}(2015),
3599--3646.

\bibitem{CL15}
Z. Chen, C.-S. Lin, Removable singularity of positive solutions for a critical elliptic system with isolated singularity, {\it Math. Ann.,} {\bf363} (2015),  501--523.

\bibitem{DW09}
E. Dancer, J. Wei, Spike Solutions in coupled nonlinear Schr\"odinger equations with Attractive Interaction, {\it Trans. Amer. Math. Soc.,} {\bf361}(2009), 1189--1208.

\bibitem{DSW11}
E. Dancer, S. Santra, J. Wei, Least energy nodal solution of a singular perturbed problem with jumping nonlinearity, {\it Annali della Scuola Normale Superiore di Pisa,} {\bf10}(2011), 19--36.

\bibitem{EGBB97}
B. Esry, C. Greene, J. Burke, J. Bohn, Hartree-Fock theory for double condesates,
{\it   Phys. Rev. Lett., } {\bf78}(1997),  3594-3597.

\bibitem{GT98}
D. Gilbarg, N. Trudinger, Elliptic Partial Differential Equations of Second Order, 2nd edn, Springer,
Berlin (1998)

\bibitem{GNN81}
B. Gidas, W.-M. Ni, L. Nirenberg, Symmetry of positive solutions of nonlinear elliptic equations in $\bbr^n$, in: Mathematical Analysis and
Applications, Part A, in: Adv. Math. Suppl. Stud., vol. 7A, Academic Press, New York, 1981, pp. 369--402.

\bibitem{HMEWC98}
D. Hall, M. Matthews, J. Ensher, C. Wieman, E. Cornell, Dynamics of
component separation in a binary mixture of Bose-Einstein condensates,
{\it Phys. Rev. Lett.,} {\bf81}(1998), 1539-1542.

\bibitem{HWW15}
Y. Huang, T.-F. Wu, Y. Wu, Multiple positive solutions for a class of
concave-convex elliptic problems in $\mathbb{R}^N$ involving sign-changing
weight (II), {\it Comm. Contemp. Math.,} {\bf 17} (2015), 1450045 (35 pages).

\bibitem{IT11}
N. Ikoma, K. Tanaka, A local mountain pass type result for a system of nonlinear Schr\"odinger equations, {\it Calc. Var. PDEs,} {\bf40}(2011), 449--480.

\bibitem{LN01}
Y. Li, W.-M. Ni, Radial	symmetry of	positive solutions of nonlinear	elliptic equations in $\bbr^n$, {\it Comm. PDEs,} {\bf18}(1993), 1043--1054.

\bibitem{LW051}
T.-C. Lin, J. Wei, Ground state of $N$ coupled nonlinear Schr\"odinger equations in $\bbr^n$, $n\leq3$, {\it Comm. Math. Phys.,}
{\bf255}(2005),  629--653.

\bibitem{LW05}
T.-C. Lin, J. Wei, Spikes in two coupled nonlinear Schr\"odinger equations,   {\it  Ann. Inst. H. Poincar\'e Anal. Non Lin\'eaire, }  {\bf22}(2005),  403-439.

\bibitem{LW06}
T.-C. Lin, J. Wei, Spikes in two-component systems of nonlinear Schr\"odinger equations with trapping potentials,
{\it J. Differential Equations, }  {\bf229}(2006),  538-569.

\bibitem{LW13}
T.-C. Lin, T.-F. Wu, Existence and multiplicity of positive solutions for two coupled nonlinear Schr\"odinger equations,  {\it  Discrete Contin. Dyn. Syst., }  {\bf 33}(2013),  2911-2938.

%

\bibitem{LW08}
Z. Liu, Z.-Q. Wang, Multiple bound states of nonlinear Schr\"odinger systems, {\it Comm. Math. Phys.,} {\bf282}(2008), 721--731.

\bibitem{LW10}
Z. Liu, Z.-Q. Wang, Ground states and bound states of a nonlinear Schr\"odinger system, {\it Adv. Nonlinear Stud.,} {\bf10}(2010), 175--193.

\bibitem{LP14}
W. Long, S. Peng, Segregated vector solutions for a class of Bose--Einstein systems, {\it J. Differential Equations,} {\bf 257}(2014), 207--230.

\bibitem{LZ16}
S. Luo, W. Zou, Existence, nonexistence, symmetry and uniqueness of ground state for critical Schr\"odinger system involving Hardy term, arXiv:1608.01123v1  [math.AP]

\bibitem{MPS06}
E. Montefusco, B. Pellacci, M. Squassina, Semiclassical states for weakly coupled nonlinear Schr\"oodinger systems, {\it J. Eur. Math. Soc.,} {\bf10}(2006), 47--71.

\bibitem{NW95}
W.-M. Ni, J. Wei, On the location and profile of spike-Layer solutions to singularly perturbed semilinear Dirichlet problems, {\it Comm. Pure
Appl. Math.,} {\bf48}(1995), 731--768.

\bibitem{NTTV10}
B. Noris, H. Tavares, S. Terracini, G. Verzini, Uniform H\"older bounds
for nonlinear Schr\"odinger systems with strong competition,  {\it  Comm. Pure
Appl. Math., } {\bf63}(2010),  267-302.

\bibitem{R14}
J. Royo-Letelier, Segregation and symmetry breaking of strongly coupled two-component Bose-Einstein condensates in a harmonic trap,   {\it  Calc. Var. PDEs,}  {\bf 49}(2014),  103-124.

\bibitem{S07}
B. Sirakov, Least energy solitary waves for a system of nonlinear Schr\"odinger equations in $\bbr^N$, {\it Comm. Math. Phys.,} {\bf 271} (2007), 199-221.

\bibitem{TV09}
S. Terracini, G. Verzini, Multipulse phases in k-mixtures of Bose-Einstein condensates,
{\it Arch. Ration. Mech. Anal., } {\bf194}(2009), 717-741.

\bibitem{TT12}
H. Tavares, S. Terracini, Sign-changing solutions of competition-diffusion elliptic systems and optimal partition problems, {\it Ann. Inst. H. Poincar\'e Anal. Non Lin\'eaire,} {\bf29}(2012), 279--300.

\bibitem{W93}
X. Wang, On concentration of positive bound states of nonlinear Schr\"odinger equations, {\it Comm. Math. Phys.,} {\bf153}(1993), 229--244.

\bibitem{WW08}
J. Wei, T. Weth, Radial solutions and phase separation in a system of two
coupled Schr\"odinger equations,   {\it  Arch. Ration. Mech. Anal., }  {\bf190}(2008),
83-106.

\bibitem{WS16}
J. Wang, J. Shi, Standing waves of a weakly coupled Schr\"odinger system with distinct potential functions, {\it J. Differential Equations,} {\bf 260}(2016), 1830--1864.

\bibitem{WWZ17}
Y. Wu, T.-F. Wu, W. Zou, On a two-component Bose¨CEinstein condensate with steep potential wells, {\it Annali di Matematica,} {\bf196}(2017), 1695-1737.

\bibitem{W17}
Y. Wu, On a $K$-component elliptic system with the Sobolev critical exponent in high dimensions: the repulsive case, {\it  Calc. Var. PDEs,} {\bf56}(2017), article 151, 51pp.

\bibitem{ZZ12}
J. Zhang, W. Zou, A Berestycki-Lion theorem revisited, {\it Commun. Contemp. Math.,}  {\bf 14}(2012), 1250033 (14 pages).
\end{thebibliography}
\end{document}